\documentclass{amsart}
\usepackage{graphicx, amsmath, amsthm, amsfonts, amssymb, tikz-cd, multirow, mathtools, mathrsfs}
\usepackage{hyperref}
\hypersetup{colorlinks=true,
linkcolor = blue,
citecolor = black,
}

\newtheorem{theorem}{Theorem}
\newtheorem{lemma}[theorem]{Lemma}
\newtheorem{proposition}[theorem]{Proposition}
\newtheorem{definition}[theorem]{Definition}

\newtheorem{corollary}[theorem]{Corollary}
\theoremstyle{remark}
\newtheorem{remark}[theorem]{Remark}
\newtheorem{example}[theorem]{Example}
\numberwithin{theorem}{section}
\numberwithin{equation}{section}

\newcommand{\R}{\mathbb{R}}
\newcommand{\C}{\mathbb{C}}

\newcommand{\U}{\mathrm{U}}
\newcommand{\SU}{\mathrm{SU}}
\newcommand{\GL}{\mathrm{GL}}
\newcommand{\SO}{\mathrm{SO}}
\newcommand{\Sp}{\mathrm{Sp}}
\newcommand{\SL}{\mathrm{SL}}
\newcommand{\F}{\mathrm{F}}

\newcommand{\PGL}{\mathrm{PGL}}
\renewcommand{\O}{\mathrm{O}}
\newcommand{\Stab}{\mathrm{Stab}}

\newcommand{\Ad}{\mathrm{Ad}}

\newcommand{\Lie}{\mathrm{Lie}}
\newcommand{\fg}{\mathfrak{g}}

\newcommand{\fz}{\mathfrak{z}}
\newcommand{\fm}{\mathfrak{m}}

\newcommand{\fn}{\mathfrak{n}}

\newcommand{\fsl}{\mathfrak{sl}}
\newcommand{\fgl}{\mathfrak{gl}}

\newcommand{\mD}{\mathcal{D}}

\newcommand{\mW}{\mathcal{W}}

\newcommand{\End}{\mathrm{End}}
\newcommand{\Aut}{\mathrm{Aut}}
\newcommand{\Hom}{\mathrm{Hom}}
\newcommand{\Sym}{\mathrm{Sym}}

\newcommand{\rk}{\mathrm{rank}}

\newcommand{\Ker}{\mathrm{Ker}}

\newcommand{\im}{\mathrm{Im}}

\newcommand{\Spec}{\mathrm{Spec}}
\newcommand{\B}{{\bf B}}
\newcommand{\Maps}{\mathrm{Maps}}
\newcommand{\Sec}{\mathrm{Sec}}

\newcommand{\Gm}{\mathbb{G}_m}

\newcommand{\mI}{\mathcal I}

\newcommand{\hmJ}{\hat{\mathcal J}}

\newcommand{\hmP}{\hat{\mathcal{P}}}
\newcommand{\Irr}{\mathrm{Irr}}
\newcommand{\supp}{\mathrm{supp}}

\newcommand{\mA}{\mathcal A}

\newcommand{\fK}{\mathfrak{K}}
\newcommand{\fP}{\mathfrak{P}}

\newcommand{\hp}{\hat{p}}
\newcommand{\hS}{\hat{S}}
\newcommand{\hSa}{\hat{\Sigma}}

\newcommand{\hO}{\hat{\mathcal O}}

\newcommand{\ev}{\mathrm{ev}}

\newcommand{\Sa}{\Sigma}

\newcommand{\mL}{\mathcal L}
\newcommand{\mM}{\mathcal M}

\newcommand{\mO}{\mathcal O}

\newcommand{\tm}{\times}
\newcommand{\ra}{\rightarrow}

\newcommand{\Dix}{\mathrm{Dix}}
\newcommand{\sq}{\mathrm{sq}}
\newcommand{\oO}{\overline{\mO}}

\newcommand{\Ab}{\mathrm{Ab}}

\newcommand{\gen}{\mathrm{gen}}

\makeatletter
\newcommand*{\doublerightarrow}[2]{\mathrel{
  \settowidth{\@tempdima}{$\scriptstyle#1$}
  \settowidth{\@tempdimb}{$\scriptstyle#2$}
  \ifdim\@tempdimb>\@tempdima \@tempdima=\@tempdimb\fi
  \mathop{\vcenter{
    \offinterlineskip\ialign{\hbox to\dimexpr\@tempdima+1em{##}\cr
    \rightarrowfill\cr\noalign{\kern.5ex}
    \rightarrowfill\cr}}}\limits^{\!#1}_{\!#2}}}
\makeatother

\makeatletter
\newcommand*\bigcdot{\mathpalette\bigcdot@{.5}}
\newcommand*\bigcdot@[2]{\mathbin{\vcenter{\hbox{\scalebox{#2}{$\m@th#1\bullet$}}}}}
\makeatother

\begin{document}
\setcounter{page}{1}

\centerline{}

\centerline{}

\title[The singular Hitchin fibration]{The singular Hitchin fibration, cameral data, and representation theory}

\author[A. Fr\"{u}h]{Alexander Fr\"{u}h}
\email{afruh99@gmail.com}
\address{School of Mathematics, University of Birmingham, Birmingham, UK.}



\thanks{The author acknowledges support from the EPSRC grant EP/V520275/1, and the projects ``Wobbly and Very stable Higgs bundles'' (Proyectos de generación del conocimiento ref. PID2023-147785NA-I00) and ``Real Calabi-Yau and Hitchin systems'' (Proyectos de Consolidación ref. CNS2022-136042), funded by MICIU/AEI/10.13039/501100011033 and NextGenerationEU/PRTR}

\begin{abstract}
    We consider the Hitchin fibration on the moduli stack of Higgs bundles with arbitrary reductive structure group, and study its singular locus using the centraliser of the Higgs field. We restrict to the case where the Higgs field has constant centraliser dimension, and describe a non-abelian structure on the corresponding locus in the moduli stack. On a class of components of this locus, we construct a factorisation of the Hitchin map through an abelianised fibration, and describe the abelianised fibres with a generalisation of the cameral data of Donagi and Gaitsgory. We apply our results to Hitchin fibrations for real groups, and we also determine a connection between the geometry of the singular Hitchin fibration and the representation theory of the Lie algebra via the orbit method.
\end{abstract}
\maketitle

\tableofcontents

\section{Introduction}

The moduli space of Higgs bundles on an algebraic curve has been a fruitful source of study since its introduction nearly 40 years ago \cite{Hitchin1}. It has found a wide range of applications, e.g. in non-abelian Hodge theory and higher Teichmüller theory \cite{Simpson1}, \cite{Hitchin3}, \cite{BIW}, mirror symmetry \cite{HT}, \cite{GWZ}, the geometric Langlands program \cite{BD}, \cite{KW}, \cite{DP}, and mathematical physics \cite{KW}, \cite{GMN}. An important feature for many of these applications is the Hitchin fibration, which gives the moduli space the structure of an algebraic integrable system \cite{Hitchin2}.

The Hitchin fibration determines an ``abelianisation'' for Higgs bundles: the Hitchin map gives a way of converting the non-abelian data of the Higgs bundle on the curve (which has an underlying structure group of $\GL_n$) into abelian data. There are two main approaches for this: spectral data, introduced by Hitchin \cite{Hitchin2}, which diagonalises the Higgs field (a global twisted linear endomorphism) to produce a sheaf of eigenspaces on a spectral curve; and the cameral data of Donagi \& Gaitsgory \cite{DG}, which reduces the structure group of the Higgs bundle to a centraliser group scheme and uses Grothendieck-Springer theory to determine an equivariant torus bundle on a cameral curve. The latter approach has the benefit of applying to $G$-Higgs bundles with arbitrary structure group $G$.

These methods give a complete description of the Hitchin fibration over its regular locus. There are extensions of both spectral and cameral data beyond the regular locus \cite{BNR}, \cite{Schaub}, \cite{deCataldo}, \cite{Pantev}, and singular fibres have been studied from various points of view (e.g. \cite{Ngo2}, \cite{Hitchin4}, \cite{dCFFHM}). Nevertheless, there is much which remains to be understood about the singular Hitchin locus.

The geometric regularity of points in a Hitchin fibre coincides with the Lie-theoretic regularity of the Higgs field, a minimality condition on its pointwise centraliser \cite{BiR}. Thus, the Higgs field centraliser dimension may be used as a measure to organise the study of the singular Hitchin locus. The failure of regularity corresponds exactly to the failure of abelianisation. In the case where regularity fails at only finitely many points on the curve, there are examples of ``semi-abelian spectral data''  \cite{Hitchin4}, \cite{Horn1}, \cite{Horn2} \cite{FraPN}; but when regularity fails everywhere, the spectral data is qualitatively different. Nevertheless, if the centraliser has \emph{constant} dimension, we can still hope for a ``nonabelianisation'' of the corresponding Higgs bundles: i.e. that they can be described by locally isomorphic data with a simpler nonabelian structure group.\footnote{There is an unfortunate clash of terminology here, as the term ``nonabelianisation'' is more widely used to refer to an unrelated construction of Gaiotto, Moore and Neitzke in the context of spectral networks \cite{GMN1}. Nonetheless, we feel that the terminology of \cite{HS} is well-suited to the phenomenon we describe in this paper.} This has been considered in the context of real groups in \cite{Peon-Nieto1}, \cite{HS}, \cite{GPPN} and \cite{HM}, and in special cases for the usual Hitchin fibration in \cite{BR}. It also fits into a larger framework of ``generalised Hitchin fibrations'' of Morrissey \& Ngô (outlined in \cite{Ngo3}).

In this paper, we consider this phenomenon systematically. We define a stratification for the moduli space $\mM(G)$ of $G$-Higgs bundles.

\begin{definition}\label{Centstratdef}
    The \emph{centraliser stratification} on $\mM(G)$ is the stratification
    \begin{equation}
    \mM(G) = \coprod_{\rk(G) \leq d \leq \dim(G)} \mM^{\gen, d}(G)
\end{equation}
where the stratum $\mM^{\gen,d}(G)$ is the locally closed subspace of $\mM$ consisting of $G$-Higgs bundles $(E,\Phi)$ such that $\Phi_x$ has centraliser dimension $d$ for all but finitely many points $x\in \Sa$.
\end{definition}

The centraliser stratification is compatible with many standard constructions on the moduli space. On the nilpotent cone, it refines to a stratification defined in \cite{BD}; and each stratum is stable under Hecke modifications and the $\C^*$-action. Each $\mM^{\gen,d}(G)$ contains an open subspace $\mM^d(G)$ of Higgs bundles whose Higgs field has constant centraliser dimension $d$ at every point in $\Sa$. If $r$ is the rank of $G$, $\mM^r(G)$ is the space of regular $G$-Higgs bundle while $\mM^{\gen, r}(G)$ is the space of generically regular $G$-Higgs bundles, which features, e.g., in the recent work of Toda in the Dolbeault geometric Langlands conjecture \cite{Toda1}. 

The main goal of this paper is to initiate the study the Hitchin fibration on the strata $\mM^{\gen,d}(G)$ for $d>r$. In the case $G\neq \GL_n$, we will exclusively restrict to considering the fibration on the sublocus $\mM^d(G)$, where we can guarantee a suitable homogeneity of the corresponding $G$-Higgs bundles. This is a significant restriction, but it is not prohibitive: e.g., in all possible cases, the locus $\mM^d(G)$ contains a full-dimensional part of the nilpotent cone by the description of \cite{BD}. We hope to consider the case of non-constant centraliser dimension in future work.

We will show in a great level of generality that the structure of the locus $\mM^d(G)$ defines a non-abelian analogue of the Hitchin fibration. It can no longer be abelianised in the sense described above; but for many cases, we will demonstrate that there is a canonical way to factorise it through an abelian fibration, which can be described via cameral data. We also show that the coarse geometry of $\mM^d(G)$ has a direct relationship with the representation theory of the Lie algebra of $G$. 

The ideas for our approach originate in the study of $G_\R$-Higgs bundles for real groups $G_\R$, following \cite{Peon-Nieto1} and \cite{HM}, and we will apply our results to determine an abelianised form of cameral data for $G_\R$-Hitchin fibrations for an arbitrary real form $G_\R$.

We first illustrate our ideas and results by describing the centraliser strata for the usual Hitchin fibration with structure group $G=\GL_n$.

\subsection{The centraliser stratification for \texorpdfstring{$G = \GL_n$}{G=GLn}}\label{IntroGLnscn}

We will in general use the stacky formulation of the Hitchin fibration (as in e.g. \cite{Ngo1}).

We first recall the usual abelianisation of the $\GL_n$-Hitchin fibration. A \emph{Higgs bundle} on a curve $\Sa$ is a pair $(V,\Phi)$ where $V$ is a vector bundle on $\Sa$ and $\Phi$ is a global section of $\End(V) \otimes K$ (where $K$ is the canonical bundle of $\Sa$). The Higgs field $\Phi$ can be thought of as a twisted family of matrices over $\Sa$, and in particular it has a well defined characteristic polynomial, albeit with coefficients which are sections of powers of $K$. The space of characteristic polynomials of rank $n$ Higgs bundles defines an affine space $\mA$, and the Hitchin map on the stack $\mM$ of rank $n$ Higgs bundles
\begin{equation}
h:\mM \ra \mA 
\end{equation}
sends a Higgs bundle to its characteristic polynomial.

The characteristic polynomial $p \in \mA$ for a Higgs bundle $(V,\Phi)$ defines a finite ramified cover of $\Sa$ of degree $n$, the \emph{spectral curve} $\Sa_p$. The points on $\Sa_p$ parametrise pairs $(x,\lambda)$ where $x$ is a point on the curve $\Sa$ and $\lambda$ is an eigenvalue of $\Phi_x$. The spectral curve carries an ``eigensheaf'' for the Higgs field, which is a line bundle if the spectral curve is smooth: this is the \emph{spectral data} of \cite{Hitchin2}. Spectral data can be defined on arbitrary spectral curves using a more general version of the eigensheaf, but this is not ``abelian'' for singular spectral curves, i.e. the generalised eigensheaf may fail to be a line bundle \cite{BNR}, \cite{Schaub}.

The Higgs bundles whose generalised eigensheaves are line bundles are those such that exactly the regular Higgs bundles \cite[Section 9]{DG}. Regular Higgs bundles occur over every point of the Hitchin base and regularity is an open condition in $\mM$. However, when the spectral curve is non-reduced, there are irreducible components of the Hitchin fibre containing no regular points. We illustrate this in the following example, studied in \cite{Branco} and \cite{SV}.

\begin{example}\label{Ribbonsexm}
Suppose $n=2k$. Let $p \in \mA$ be of the form $p=q^2$; then the spectral curve $\Sa_p$ is non-reduced, and for suitably generic $q$, $\Sa_p$ is a twice-thickened copy of a smooth curve $\Sa_q$ (which can be thought of as a spectral curve of degree $k$). There are two qualitatively different possibilities for a Higgs bundle $(V,\Phi)$ in the corresponding Hitchin fibre.

If $(V,\Phi)$ is \emph{generically regular}, i.e., it is regular at all but finitely many points of $\Sa$, then it is determined by a rank $k$ Higgs bundle $(W,\Psi)$, an effective divisor on $\Sa$, and a twisted self-extension of $(W,\Psi)$. If the divisor is trivial, then $(V,\Phi)$ is regular, and corresponds to a line bundle on $\Sa_p$ via spectral data. We note in this case that the Higgs field $\Phi$ is generically \emph{not} diagonalisable.

In contrast, if $(V, \Phi)$ is \emph{nowhere regular}, then the Higgs field is generically diagonalisable, but each of the eigenvalues has geometric multiplicity $2$ and the generalised eigensheaf for the Higgs field defines a rank $2$ vector bundle on the reduced spectral curve $\Sa_q$. In this case, the Higgs field $\Phi$ has a constant centraliser dimension of $2n$. For a fixed topological invariant, the locus of nowhere regular Higgs bundles constitutes an irreducible component in the reduced substack of the Hitchin fibre.
\end{example}  

Example \ref{Ribbonsexm} gives a first indication of how the centraliser strata of Definition \ref{Centstratdef} interact with the Hitchin fibration. The stratification organises the components of the Hitchin fibre according to the form of their spectral data; and moreover a nonabelian version of spectral data is possible for the nowhere regular components.

For a fixed value of $d$,  we can characterise the Higgs bundles $(V,\Phi)$ which occur as points in the centraliser stratum $\mM^{\gen,d}$. If the Higgs field $\Phi$ is generically diagonalisable on $\Sa$, then the multiplicities of its eigenvalues at the generic point of $\Sa$ determines a \emph{multiplicity partition} ${\bf m} =(s^{l_{s}};\, \cdots;\,1^{l_1}) $ of $n$; e.g., the nowhere regular Higgs bundles in Example \ref{Ribbonsexm} determine the partition $(2^k)$. In these cases $(V,\Phi)$ defines a point of $\mM^{\gen,d}$ exactly when
\begin{equation}
    S_2({\bf m}) :=\sum_{m=1}^sl_mm^2 = d.
\end{equation}

Moreover, there is a refinement of the centraliser stratification as
\begin{equation}
    \mM = \coprod_{{\bf m} \, \vdash n} \mM^{\gen;{\bf m}}
\end{equation}
such that each $\mM^{\gen;{\bf m}}$ in particular contains all generically diagonalisable Higgs bundles with multiplicity partition ${\bf m}$. It turns out to be much more natural to study the Hitchin fibration on this refined stratification than directly on the centraliser stratification.

We can easily describe the generic form of spectral data for $\mM^{\gen;{\bf m}}$. The partition ${\bf m}$ dictates the form of the spectral curve for a Higgs bundle $(V,\Phi)$; the characteristic polynomial $p$ of $(V,\Phi)$ decomposes as 
\begin{equation}\label{IntroCharspliteqn}
    p = (p_s)^s\cdots(p_1)
\end{equation}
where $p_m$ is the characteristic polynomial for a spectral curve of degree $l_m$ (with $p_m=1$ when $l_m=0$).

\begin{theorem}[Proposition \ref{GLnSmSpectralDataprp}, Corollary \ref{GLnSHFibredimcor}]
    Let $p \in \mA$ be a generic point of the Hitchin base which has the form (\ref{IntroCharspliteqn}). Then a Higgs bundle $(V,\Phi)$ with spectral curve $\Sa_p$ determines a point in $\mM^{\gen;{\bf m}}$ exactly when the generalised eigensheaf on $\Sa_p$ for $(V, \Phi)$ determined by spectral data is the pushforward of a sheaf from the reduced subscheme $(\Sa_p)^{red}$ of $\Sa_p$.

    Moreover, the intersection of the Hitchin fibre $h^{-1}(p)$ with $\mM^{\gen;{\bf m}}$ constitutes the reduced substack of an irreducible component of $h^{-1}(p)$ (after fixing topological invariants).\footnote{By the ``intersection'' of these stacks we mean the fibre product $\Spec(\C) \tm_\mA\mM^{\gen;{\bf m}}$ determined by the $\C$-point corresponding to $p \in \mA$.}
\end{theorem}

Thus, generically the spectral data for the stratum $\mM^{\gen;{\bf m}}$ determines naturally defined irreducible components in the compactified Jacobian of $\Sa_p$. For ${\bf m} \neq (1^n)$, we can exhibit a global non-abelian structure on this family of irreducible components as they vary over the Hitchin base. 

We restrict to the locus $\mM^{\bf m}$ of Higgs bundles in $\mM^{\gen,{\bf m}}$ whose Higgs field has constant centraliser dimension. Let
\begin{equation}
    \mA^{\bf m} = \prod_{m=1}^s \mA(\GL_{l_m})
\end{equation}
and let $\mu:\mA^{\bf m} \ra \mA$ be the map sending a tuple $(p_1,\, \cdots, p_s)$ to the polynomial $p$ defined by (\ref{IntroCharspliteqn}); the map $\mu$ is a normalisation map onto its image.

The following is the statement of our main results in the case that $G=\GL_n$.

\begin{theorem}[Theorem \ref{GlobalNonabnthm}, Theorem \ref{GlobalSHiggsAbnthm}, Proposition \ref{GLnAbHitchinprp}]\label{IntroMainthm}
Let ${\bf m} \neq (1^n)$.
    \begin{itemize}
        \item[(i)] The restriction of the Hitchin map on $\mM^{\bf m}$ factors through a morphism $h^{\bf m}:\mM^{\bf m} \ra \mA^{\bf m}$.
        \item[(ii)] There is a (non-canonical) section $\varepsilon^{\bf m}:\mA^{\bf m} \ra \mM^{\bf m}$ of $h^{\bf m}$.
        \item[(iii)] $\mM^{\bf m}$ has a \emph{nonabelianisation} over $\mA^{\bf m}$: i.e., there is a nonabelian group scheme $\mI$ over $\mA^{\bf m} \tm \Sa$ such that for any $p \in \mA^{\bf m}$, the fibre of $\mM^{\bf m}$ over $p$ is isomorphic to the stack ${\bf B}_\Sa\mI_p$ of torsors for the group scheme $\mI_p = \mI|_{\{p\}\tm \Sa}$ on $\Sa$.
        \item[(iv)] There is a canonical factorisation of $h^{\bf m}$ as
        \begin{equation}
\begin{tikzcd}
\mathcal M^{\bf m} \arrow[r, "\Ab"] & \mathcal M\left(\prod_{i=1}^s \GL_{l_m}\right) \arrow[r, "h_{G_{\bf m}}"] & \mA^{\bf m}
\end{tikzcd}
        \end{equation}
        where the right-hand arrow is the Hitchin map for $G_{\bf m} := \prod_{i=1}^s \GL_{l_m}$. 
        \item[(v)] Over a point $p \in \mA^{\bf m}$, $h_{G_{\bf m}}^{-1}(p)$ can be identified with ${\bf B}_\Sa\mI_p^{ab}$, where $\mI_p^{ab}$ is the abelianisation of the group scheme $\mI_p$ in (iii).
    \end{itemize}
\end{theorem}

\subsection{Centraliser strata for \texorpdfstring{$G \neq \GL_n$}{G=/=GLn} and abelianised cameral data}

Theorem \ref{IntroMainthm} can be generalised to other structure groups, but there are a number of additional features and complications.

There is a decomposition of the centraliser strata
\begin{equation}\label{IntroCentraliserSheet2eqn}
\mM^{\gen;d}(G) = \bigcup_{S \in \Irr(\fg_d)} \mM^{\gen}(G;S);
\end{equation}
this is determined by a decomposition of the Lie algebra centraliser strata $\fg_d$ into irreducible components $S$, known as \emph{sheets}. The invariant theory of sheets is particularly well-suited to the study of the Hitchin map because of its similarities to the invariant theory of the Lie algebra \cite{Borho2}, \cite{Katsylo}. Thus, instead of studying the stratum $\mM^{\gen;d}$ directly, we study each piece $\mM^{\gen}(G;S)$; more precisely, we study the open sublocus $\mM(G;S)$ in $\mM^{\gen}(G;S)$ on which the Higgs field centraliser dimension is constant.

Unlike in the case for $G = \GL_n$, the loci $\mM(G;S)$ may intersect one another (so the decomposition (\ref{IntroCentraliserSheet2eqn}) does not define a refinement of the centraliser stratification). Moreover, sheets need not contain semisimple elements (so that generic points in $\mM(G;S)$ may define Higgs bundles which are nowhere semisimple); a sheet containing a semisimple element is called \emph{Dixmier}.

In Theorem \ref{GlobalNonabnthm}, we prove an analogue for parts (i) - (iii) of Theorem \ref{IntroMainthm} which applies to an arbitrary structure group $G$ as long as the sheet $S$ is smooth; for $G$ simple, there are finitely many exceptions to this condition, which occur only when $G$ is an exceptional group \cite{ImHof}. In particular, when $S$ is smooth (and not the regular sheet) we justify that $\mM(G;S)$ admits a nonabelianisation. The precise statements are more complicated than those of Theorem \ref{IntroMainthm} due to an obstruction related to the representation theory of $\Lie(G)$ (see Section \ref{IntroRepsscn}).

When the sheet $S$ is Dixmier, and satisfies an additional simplifying assumption (which holds whenever $G$ is a classical group), we can also prove versions of parts (iv) \& (v) of Theorem \ref{IntroMainthm}. We determine a canonical factorisation of the Hitchin map on $\mM(G;S)$ through an abelian fibration $\mM^{ab}(G;S)$, which ``abelianises'' the non-abelian structure on $\mM(G;S)$ (Proposition \ref{AbnSHiggsprp} \& Theorem \ref{GlobalSHiggsAbnthm}). We do not directly identify $\mM^{ab}(G;S)$ with the Hitchin fibration for a smaller group, but we show that the fibres of $\mM^{ab}(G;S)$ can be described with a variant of cameral data with the same form as that for the usual $G$-Hitchin fibration in \cite{DG}. 

In particular, for a generic point $p \in \mA(G)$ in the Hitchin base which is in the image of $\mM(G;S)$, we determine an \emph{$S$-cameral curve} $\hat{\Sigma}_p$ which is a finite ramified cover of $\Sa$, together with an algebraic torus $\bar{Z}$ and a finite group $W_L$ acting on both $\hat{\Sigma}_p$ and $\bar{Z}$.

\begin{theorem}[Theorem \ref{SCameralDatathm}]\label{IntroCamthm}
    For a generic point $p \in \mA(G)$ in the image of $\mM(G;S)$, the fibre of $\mM^{ab}(G;S)$ over $p\in \mA(G)$ can be identified up to isogeny with the stack $\hat{\mathcal P}_{S,p}$ of $W_L$-equivariant $\bar{Z}$-torsors $\mathcal Z$ on $\hat\Sigma_p$ such that: 
    \begin{itemize}
    \item[$(*)$] for each ramification point $x \in \hSa_p$, there is an isomorphism $\mathcal Z_x\cong \bar{Z}$ equivariant with respect to the actions of $\bar{Z}$ and $\Stab_{W_L}(x)$.
\end{itemize}
\end{theorem}

The stack $\hmP_{S,p}$ can be viewed as a stack of \emph{abelianised pseudo-cameral data}. We use the prefix ``\emph{pseudo-}'' to account for the fact that the identification in Theorem \ref{IntroCamthm} is only an isogeny in general; we believe it should be possible to correct this discrepancy with an additional root-vanishing condition at ramification as in \cite{DG}.

\subsection{Abelianised cameral data for \texorpdfstring{$G_\R$}{GR}-Higgs bundles}

We apply these results to make general statements about the structure of $G_\R$-Hitchin fibrations for real groups $G_\R$. $G_\R$-Higgs bundles are the objects which correspond to points in $G_\R$-character varieties under the non-abelian Hodge correspondence and have been studied alongside $G$-Higgs bundles throughout the history of the subject \cite{Hitchin1}, \cite{Hitchin3}, \cite{Simpson1}, \cite{BGPG}, \cite{GPGMIR}. They feature in higher Teichmüller theory \cite{BIW}, and provide natural examples for objects appearing in mirror symmetry and the relative Langlands program \cite{KW}, \cite{Hitchin5}, \cite{HLM}.

There is a $G_\R$-Hitchin map $h_{G_\R}:\mM(G_\R) \ra \mA(G_\R)$, but its structure is quite different to that of the $G$-Hitchin fibration. If $G_\R$ is quasi-split, there is an abelianisation for the $G_\R$-Hitchin fibration via cameral data \cite{GPPN}, \cite{HM}. For non-quasi-split groups $G_\R$, a nonabelianisation for the $G_\R$-Hitchin fibration is known \cite{Peon-Nieto1}, \cite{GPPN}, \cite{HM}, and examples have been calculated using spectral data exhibiting a variety of different structures \cite{HS}, \cite{BS}, \cite{LeePN}.

As the main application of our results, we determine a general form of ``abelianised cameral data'' which applies to an arbitrary real form $G_\R$. This determines a uniform structure for all of the qualitatively different cases which can occur. 

For any real form $G_\R$ of $G$ we can associate a sheet $S$ in $G$, and to a generic point $p \in \mA(G_\R)$ we can associate the stack $\hmP_{S,p}$ of abelianised pseudo-cameral data as in Theorem \ref{IntroCamthm}. Let $\mM^{reg}(G_\R)$ be the open substack of $\mM(G_\R)$ of regular $G_\R$-Higgs bundles.
\begin{theorem}[Theorem \ref{HCameralDatathm}]\label{IntroGRthm}
    There is a canonical factorisation of the $G_\R$-Hitchin map on $\mM^{reg}(G_\R)$ through an abelian fibration $h^{ab}_{G_\R}:\mM^{ab}(G_\R) \ra \mA(G_\R)$; it abelianises the non-abelian structure on $\mM^{reg}(G_\R)$ described in \cite{HM} (in the sense of Theorem \ref{IntroMainthm} (v)).

    For a generic point $p \in \mA(G_\R)$, there is an involution $\Theta$ on the stack $\hmP_{S,p}$ such that the fibre of $h_{G_\R}^{ab}$ over $p \in \mA(G_\R)$ can be identified with the fixed point stack $(\hmP_{S,p})^{\Theta}$, up to a finite discrepancy.
\end{theorem}

The ``finite discrepancy'' in the statement of Theorem \ref{IntroGRthm} hides a number of technical details which we largely do not consider, including a finite cover of the usual $G_\R$-Hitchin base arising from a ``regular quotient'' considered in detail in \cite{HM}.

\subsection{Centraliser strata and invariants in representation theory}\label{IntroRepsscn}

Our main results for $G \neq \GL_n$ cannot be stated as simply as the case for $G = \GL_n$; the main complications arise from the failure of smoothness of the relative centraliser over a sheet $S$. Nevertheless, this obstruction can be controlled by a finite group $F_S$ attached to the sheet $S$, which we call the \emph{Katsylo group}; an example where $F_S$ is non-trivial already occurs for the group $G=\Sp_4$. 

We show that this obstruction group controls various geometric properties of the sheet, and also determines discrete invariants for representations of the group $G$ and its Lie algebra $\fg$ constructed via deformation quantisations of nilpotent orbit covers. These arise in the context of the orbit method of Losev \cite{Losev4}, the 1-dimensional representations of finite $\mW$-algebras \cite{Topley}, and the unipotent representations of Losev, Mason-Brown and Matvieievskyi \cite{LMM}.

\subsection{Structure of the paper}

From Sections \ref{Sheetsscn} - \ref{NonabnHitchinscn}, we build the necessary theory for our main results for arbitrary structure groups. In Section \ref{Sheetsscn}, we recall the background on the theory of sheets needed for the framework we use. Sections \ref{AdjointQscn} and \ref{Abelianisationscn} develop the local theory for the nonabelianisation and abelianised cameral data respectively, while Section \ref{NonabnHitchinscn} gives the statements and proofs of the main results. 

Section \ref{Hitchinegscn} focuses on the specific cases $G = \GL_n$ and $G=\Sp_4$. In particular we show how the results stated in Section \ref{IntroGLnscn} follow from the general statements in Section \ref{NonabnHitchinscn}, and we consider how the case of $G= \Sp_4$ differs from the case for $G=\GL_4$. In Section \ref{RealHitchinscn}, we construct the abelianised cameral data for real forms.

We have isolated some of our results on the geometry of sheets and the connections with representation theory in Section \ref{AdjointQuotientKatsyloscn}, as they may be of separate interest. It is intended that this section may be read largely independently from the rest of the paper. 

\subsection{Notational conventions}

We will use $G$ to denote a connected reductive algebraic group over $\mathbb{C}$, and $\mathfrak{g}$ to denote its Lie algebra. Similarly, for an arbitrary algebraic group denoted by a capital letter (e.g. $A$), the lowercase gothic script (e.g. $\mathfrak{a}$) will denote its Lie algebra, unless otherwise stated. For any $g \in G$, the automorphism on $G$ defined by conjugation by $g$ will be denoted $I_g$.

We will always fix a Cartan subgroup $T \leq G$, and denote by $W$ the group $N_G(T)/T$, which for our purposes is the Weyl group of $G$ with respect to $T$. Moreover, if $L$ is a Levi subgroup of $G$, we will denote by $W_L$ the group $N_G(L)/L$, sometimes known as the \emph{relative Weyl group for $L$}.

We will denote by $\Sigma$ a non-singular connected projective curve over $\mathbb{C}$ of genus $g > 1$, and denote by $K$ its canonical line bundle. 

We adopt the following shorthand for restriction of schemes: if $X$ is a $Y$-scheme, and $Y'$ is a subscheme of $Y$, we will denote by $X_{Y'}$ the fibre product $X \times_Y Y'$. For a map of stacks $f:\mathcal X \rightarrow \mathcal Y$ and a $\mathbb{C}$-point $a \in \mathcal Y(\mathbb{C})$, we will abuse notation and write $f^{-1}(a)$ to denote the fibre product $\mathbb{C} \times_{a,f}\mathcal X$.

\subsection{Acknowledgements}

I would like to thank the following people for helpful discussions and correspondence at various stages of this project: M. Bulois, M. Chaffe, A. Fernandez Herrero, S. Goodwin, J. Kimberley, I. McIntosh, B.C. Ngô, J. Summerfield, L. Topley, and G. Wilkin. I would like to thank T. Pantev for his generous hospitality during a research visit to the University of Pennsylvania in February 2025. I would especially like to thank my supervisor A. Peón-Nieto for her constant support and guidance, and for providing careful comments on numerous drafts of this paper. 

\section{The sheets of a reductive Lie algebra}\label{Sheetsscn}

We give an overview of the theory of sheets of reductive Lie algebras and of their quotients under the adjoint action. We give a number of examples to illustrate the range of geometric phenomena which can occur. The content of this section is well-established; we provide it for the reader's convenience and to set our notation.

\subsection{Sheets and their geometric quotients}\label{SheetsQuotientssbn}
We recall the necessary constructions from the theory of sheets, as developed in \cite{BK} and \cite{Borho2}, and two constructions of their quotients under the adjoint action, one which generalises the Chevalley restriction theorem \cite{Borho2} and one which generalises Kostant's section \cite{Katsylo}. A readable and comprehensive overview of the theory of sheets from a purely algebraic viewpoint can be found in Sections 1 and 2 of \cite{ImHof}.

The adjoint action of $G$ on $\mathfrak{g}$ determines the \emph{centraliser group scheme} $\mathcal I$ over $\mathfrak{g}$; the fibre $\mathcal I_x$ of $\mathcal I$ over a point $x$ in $\mathfrak{g}$ is the centraliser group $C_G(x)$. The adjoint and scaling actions on $\mathfrak{g}$ determine actions of $G$ and $\mathbb{G}_m$ on $\mathcal I$, and these actions commute.

We can use the dimension of $\mathcal I$ to stratify the Lie algebra $\mathfrak{g}$; we decompose $\mathfrak{g}$ as   
\begin{equation}\label{Centraliserstrateqn}
    \mathfrak{g} = \bigcup_{d \in \mathbb{N}} \mathfrak{g}_d
\end{equation}
where
\begin{equation}\label{Centraliserdimeqn}
    \mathfrak{g}_d = \{ x \in \mathfrak{g} \, | \, \text{dim } \mathcal I_x = d \}.
\end{equation}

\begin{remark}\label{Centraliserstratrmk}
    The minimal value of $d$ for which $\mathfrak{g}_d$ is non-empty is equal to the rank $r$ of the Lie algebra, and the stratum $\mathfrak{g}_r$ is the regular locus of $\mathfrak{g}$, which is open and dense in $\mathfrak{g}$ \cite{Kostant}. The actions of $G$ and $\mathbb{G}_m$ on $\mathfrak{g}$ restrict to actions on each stratum $\mathfrak{g}_d$.
\end{remark}

In general, the strata $\mathfrak{g}_d$ are not irreducible (e.g. see Example \ref{Sp4Sheetexm}); thus we have the following definition from \cite{BK}.

\begin{definition}\label{Sheetdef}
    A \emph{sheet} is an irreducible component of some $\fg_d$.
\end{definition}

\begin{remark}\label{ActionsSheetsrmk}
   Since $G$ and $\mathbb{G}_m$ are connected, the $G \times \mathbb{G}_m$-action restricts to an action on each sheet.
\end{remark}

The related notion of decomposition classes is important for the classification of sheets. We define an equivalence relation $\sim$ on $\mathfrak{g}$ where $x \sim y$ if there exists $g \in G$ such that 
\begin{equation}
    C_G(\Ad_g(x_{ss})) = C_G(y_{ss})
\end{equation}
and $\Ad_g(x_n) = y_n$. Here $x = x_{ss}+x_n$ and $y=y_{ss}+y_n$ are the Jordan decompositions of $x$ and $y$ into semisimple and nilpotent parts.

\begin{definition}\label{Decompclassdef}
    A \emph{decomposition class} of $\mathfrak{g}$ is an equivalence class of the equivalence relation $\sim$.
\end{definition}

\begin{remark}\label{Decompdatarmk}
    By the definition of the equivalence relation $\sim$, the decomposition classes of $\mathfrak{g}$ are in bijection with $G$-conjugacy classes of pairs $(L,\mathcal O)$, where $L \leq G$ is a Levi subgroup and $\mathcal O \subset \mathfrak{l}$ is a nilpotent orbit of $L$. The $G$-conjugacy class of $(L, \mathcal O)$ is called the \emph{decomposition data} of the corresponding decomposition class. We shall often implicitly fix a representative of the conjugacy class and simply refer to the pair $(L,\mathcal O)$ as the decomposition data.
    
    Every sheet $S$ contains a unique decomposition class $\mathcal D$ such that $\mathcal D$ is dense in $S$. By a slight abuse of terminology, we shall refer to the decomposition data $(L,\mathcal O)$ for $\mathcal D$ as the decomposition data for the sheet $S$. The pairs $(L, \mathcal O)$ which define decomposition data for a sheet are exactly those for which the nilpotent orbit $\mathcal O$ itself constitutes a sheet in the subalgebra $\mathfrak{l}$ \cite[Satz 4.3 and Korollar 4.4]{Borho2}; such nilpotent orbits are called \emph{rigid}.
\end{remark}

The following class of sheets are of particular interest for our purposes.

\begin{definition}\label{DixmierSheetdef}
    A sheet $S$ is called \emph{Dixmier} if $S$ contains a semisimple element of $\mathfrak{g}$.
\end{definition}

\begin{remark}\label{DixmierSheetrmk}
   $S$ is Dixmier if and only if the $L$-orbit $\mathcal O$ in its decomposition data is the zero orbit \cite[4.4]{Borho2}. As a result, we refer to the Dixmier sheet with decomposition data $(L,0)$ as the \emph{(Dixmier) sheet associated to $L$}. See Example \ref{Sp4Sheetexm} for an example of a sheet which is not Dixmier. 
\end{remark}

We will now consider the quotient for the $G$-action on a sheet $S$ given in \cite{Borho2}. Let $(L,\mathcal O)$ be the decomposition data for $S$, with $L$ chosen such that its maximal torus is $T$. Let $\mathfrak{z}$ be the centre of $\mathfrak{l}$, which by our choice of $L$ is contained in $\mathfrak{t}$.

Let $\overline{S}$ be the closure of $S$ in $\mathfrak{g}$. Then $\overline{S}$ is an affine variety, and the $G$-action on $S$ extends to an action on $\overline{S}$. So, in particular, we can consider the affine GIT quotient $\chi_{\overline{S}}:\overline{S} \rightarrow \overline{S}//G$, where
$$\overline{S}//G = \Spec(\mathbb{C}[\overline{S}]^G).$$
Since $G$ is reductive, the closed embedding $\iota:\overline{S} \hookrightarrow \mathfrak{g}$ induces a commutative diagram
\begin{equation}\label{AdjointGITcompareqn}
    \begin{tikzcd}
        \overline{S} \arrow[d, "\chi_{\overline{S}}"'] \arrow[r, hook, "\iota"] & \mathfrak{g} \arrow[d, "\chi"] \\
        \overline{S}//G \arrow[r, hook, "\iota_G"]                              & \mathfrak{t}/W,                
    \end{tikzcd}  
\end{equation}
where the lower horizontal arrow is a closed embedding; in (\ref{AdjointGITcompareqn}), we have used the Chevalley restriction theorem to identify $\mathfrak{g}//G$ with $\mathfrak{t}/W$. Then, Borho's generalisation of the Chevalley restriction theorem is given by the following proposition.

\begin{proposition}\label{BorhoChevalleyprp}\cite[Satz 6.3 and Korollar 6.4]{Borho2}
    There is an inclusion $\mathfrak{z} \hookrightarrow \overline{S}$, and the induced map $\nu_L:\mathfrak{z}/W_L \rightarrow \overline{S}//G$ is a normalisation map. Moreover, the composition $\iota_G \circ \nu_L:\mathfrak{z}/W_L \rightarrow \mathfrak{t}/W$ is the map induced by the inclusion $\mathfrak{z} \hookrightarrow \mathfrak{t}$.
\end{proposition}

\begin{remark}\label{GeometricQuotientrmk}
	We will denote the normalisation of $\overline{S}//G$ by $\mathfrak{c}_S$, and we will denote by $\nu_S:\mathfrak{c}_S \rightarrow \mathfrak{t}/W$ the map induced by (\ref{AdjointGITcompareqn}). By Proposition \ref{BorhoChevalleyprp}, $\mathfrak{c}_S$ is canonically isomorphic to $\mathfrak{z}/W_L$, but it is defined independently of the choice of pair $(L,\mathcal O)$.
\end{remark}

The situation is complicated by the fact that $\overline{S}$ may not be normal (even when $S$ is normal), e.g., see Example \ref{GLnSheetexm}. Nonetheless, if the sheet itself is normal, there is a geometric quotient for $S$ compatible with Proposition \ref{BorhoChevalleyprp}.

\begin{theorem}\label{BCGeometricQthm}\cite[Theorem 6.5]{Borho2}
    If $S$ is normal, then there is a geometric quotient $\chi_S:S \rightarrow \mathfrak{c}_S$ for the action of $G$ on $S$ which satisfies $\chi_{\overline{S}}|_S = \nu_S \circ \chi_S$. In particular, the following diagram commutes:
    \begin{equation}\label{BCGeometricQ1eqn}
        \begin{tikzcd}
            S \arrow[d, "\chi_{S}"'] \arrow[r, hook] & \mathfrak{g} \arrow[d, "\chi"] \\
            \mathfrak{c}_S \arrow[r, "\nu_S"]         & \mathfrak{t}/W.                
        \end{tikzcd}
    \end{equation}
\end{theorem}

\begin{remark}\label{BCEquivariancermk}
    The $\mathbb{G}_m$-action on $S$ extends to an action on $\overline{S}$, and the scaling actions on $\mathfrak{z}$ and $\mathfrak{t}$ induce $\mathbb{G}_m$-actions on $\mathfrak{c}_S$ and $\mathfrak{t}/W$ respectively. All of the relevant maps above (in particular $\chi_S$ and $\nu_S$) are equivariant with respect to these actions.
\end{remark}

For the regular sheet, there is a section to the Chevalley map $\chi$ defined in \cite{Kostant}; there is an analogous construction for an arbitrary sheet \cite{Katsylo}.

Let $e \in S$ be a nilpotent element (which always exists by \cite[Korollar 3.2]{BK}), and complete it to an $\mathfrak{sl}_2$-triple $(e,h,f)$, i.e. $h \in \mathfrak {g}$ is semisimple, $f \in \mathfrak{g}$ is nilpotent and the following relations are satisfied:
\begin{equation}\label{SL2Tripleeqn}
    [h,e] = 2e, \, [h,f] = 2f, \, [e,f] = h.
\end{equation}
We recall that a transverse slice at $e$ to the adjoint action on $\mathfrak{g}$ is given by the \emph{Slodowy slice} \cite[Section 7.4]{Slodowy}:
    \begin{equation}\label{SlodowySliceeqn}
        \mathfrak{S}= e+\mathfrak{c}_{\mathfrak{g}}(f).
    \end{equation}

\begin{definition}\label{KatsyloSlicedef}\cite{Katsylo}
    Let $(e,h,f)$ be an $\mathfrak{sl}_2$-triple with $e \in S$. The \emph{Katsylo slice} $\mathfrak{K}$ to $e$ is the affine subvariety $\mathfrak{K} = \mathfrak{S} \cap S$ of $S$.
\end{definition}

The first key property of the Katsylo slice is that it is a global transverse slice for the sheet (not just a slice at $e$).

\begin{proposition}\cite[Theorem 0.1]{Katsylo}\label{KatsyloSliceprp}
    The map $\Ad:G \times \mathfrak{K} \rightarrow S$ given by the adjoint action is smooth and surjective.
\end{proposition}

There is also a natural $\mathbb{G}_m$-action on $\mathfrak{K}$, which is a restriction of a $\mathbb{G}_m$-action on $\mathfrak{S}$. We require the following lemma, whose content is contained in \cite[Sections 7.3 \& 7.4]{Slodowy}. Let 
\begin{equation}\label{WtDecompositioneqn}
    \mathfrak{g} = \bigoplus_{w \in \mathbb{Z}} \mathfrak{g}^w
\end{equation}
be the weight decomposition for the adjoint action of $h$ on $\mathfrak{g}$, i.e. $[h,x] = wx$ for all $x \in \mathfrak{g}^w$.

\begin{lemma}\cite{Slodowy}\label{KazhdanActionlem}
    There is a one-parameter subgroup $\lambda:\mathbb{G}_m \rightarrow G$ which acts as $\Ad_{\lambda(t)}(x) = t^wx$ for all $x \in \mathfrak{g}^w$. In particular, $\Ad_{\lambda(t)}(e)=t^2e$.
\end{lemma}

\begin{remark}\label{KazhdanActionrmk}
    The context of \cite[Sections 7.3 \& 7.4]{Slodowy} only applies when $G$ is semisimple of adjoint type; however, since we work over $\mathbb{C}$, Lemma \ref{KazhdanActionlem} holds for arbitrary reductive $G$. 
\end{remark}

\begin{definition}\label{KazhdanActiondef}
    The \emph{Kazhdan action} of $\mathbb{G}_m$ on $\mathfrak{S}$ is defined by
    $$t \cdot x = \Ad_{\lambda(t^{-1})}(t^2x),$$
    where $\lambda$ is the cocharacter of Lemma \ref{KazhdanActionlem}.
\end{definition}

\begin{remark}\label{KatsyloActionsrmk}
    The Kazhdan action restricts to an action on $\mathfrak{K}$ and lifts to an action on the restriction of the centraliser $\mathcal I_{\mathfrak{K}}$ to $\mathfrak{K}$.
\end{remark}

The second key property of the Katsylo slice is that its intersection with a given $G$-orbit in $S$ is the orbit of a group acting on the slice.

\begin{definition}\label{ReductiveCentraliserdef}
    The \emph{reductive centraliser} of $e$ (with respect to the $\mathfrak{sl}_2$-triple $(e,h,f)$) is the group
    \begin{equation}
        A = \{g \in G \, | \, \Ad_g(x) = x, \, \forall x \in \mathfrak{s}\} = Z_G(\mathfrak{s}),
    \end{equation}
    where $\mathfrak{s}$ is the copy of $\mathfrak{sl}_2$ generated by $(e,h,f)$.
\end{definition}

The restriction of the adjoint action of $G$ to an action of $A$ on $\mathfrak{g}$ defines an $A$-action on $\mathfrak{K}$, which commutes with the Kazhdan action.

\begin{theorem}\label{RedCentraliseractionthm}\cite[Theorems 0.2 and 0.3]{Katsylo}
    The $A$-action on $\mathfrak{K}$ satisfies the following properties.
    \begin{itemize}
        \item The identity component $A^\circ$ of $A$ acts trivially on $\mathfrak{K}$.
        \item If $x, y \in \mathfrak{K}$ are conjugate under $G$, they are also conjugate under $A$.
    \end{itemize}
\end{theorem}

This gives a second description for the geometric quotient for the $G$-action on a sheet $S$. Consider the finite group
\begin{equation}\label{ComponentGroupeqn}
\Gamma = A/A^\circ;
\end{equation}
the map $\Gamma \rightarrow C_G(e)/C_G^\circ(e)$ induced by the inclusion $A \hookrightarrow C_G(e)$ is an isomorphism, i.e. $\Gamma$ is the \emph{component group of $e$}. 

\begin{theorem}\label{KatsyloGeometricQuotientthm}\cite[Theorem 0.4]{Katsylo}
There exists a map $\chi_e:S \rightarrow \mathfrak{K}/\Gamma$ which is a geometric quotient for the $G$-action on $S$. This map makes the diagram
\begin{equation}\label{KastyloGQeqn}
\begin{tikzcd}
\mathfrak{K} \arrow[rd, two heads] \arrow[r, hook] & S \arrow[d] \arrow[d, "\chi_e"] \\
                                                   & \mathfrak{K}/\Gamma            
\end{tikzcd}
\end{equation}
commute.
\end{theorem}

\begin{remark}\label{KatsyloGQrmk}
    Theorem \ref{KatsyloGeometricQuotientthm} is true for arbitrary $S$, with no normality assumption. If $S$ is normal, by uniqueness of geometric quotients and Theorem \ref{BCGeometricQthm}, there is a canonical isomorpism $\mathfrak{c}_S \cong \mathfrak{K}/\Gamma$. 
    
    If $S$ is normal, the quotient map $\mathfrak{K} \rightarrow \mathfrak{c}_S$ is $\mathbb{G}_m$-equivariant with respect to the Kazhdan action on $\mathfrak{K}$ and the \emph{square} of the $\mathbb{G}_m$-action on $\mathfrak{c}_S$ defined in Remark \ref{BCEquivariancermk}. We show in Corollary \ref{KatsyloAffinecor} that if $S$ is non-singular, the Kazhdan action admits a square-root compatible with the usual $\mathbb{G}_m$-action on $\mathfrak{c}_S$.
\end{remark}

\subsection{Examples of sheets}\label{Sheetsegsbn}
We recall the descriptions of the sheets in $\mathfrak{gl}_n$ and $\mathfrak{sp}_4$, and note how far the geometric features of the regular sheet carry over to these cases. These examples form the basis for Section \ref{Hitchinegscn}. 

We begin by considering the sheets for $G=\GL_n$, which have much in common with the regular case. Sheets in Dynkin type A have been well-studied, e.g. see \cite{OW} and \cite{Kraft}.

\begin{example}\label{GLnSheetexm}
    Let $G=\GL_n$. Every Levi subgroup $L$ of $G$ is a product of general linear groups $\GL_{m_i}$, for some $m_i \in \mathbb{N}^+$ with $$\sum_{i}m_i = n.$$
    In particular, the conjugacy classes of Levis of $\GL_n$ are in bijection with partitions ${\bf m} =(m_1 \geq ... \geq m_r)$ of $n$ (we use a different notational convention for partitions compared with Section \ref{IntroGLnscn}). The nilpotent orbit contained in the Dixmier sheet associated to the Levi $L$ corresponding to a partition ${\bf m}$ is the nilpotent orbit with Jordan partition ${\bf m}^*$ (the partition \emph{conjugate} to ${\bf m}$) \cite[Satz 2.2]{Kraft}. In particular, every sheet of $\mathfrak{gl}_n$ is Dixmier and the sheets are pairwise disjoint.

We consider the adjoint quotient space $\mathfrak{c}_S$ of Section \ref{SheetsQuotientssbn}; we use the notation of Proposition \ref{BorhoChevalleyprp} and Theorem \ref{BCGeometricQthm}. All the sheets of $\mathfrak{gl}_n$ are non-singular (by Theorem \ref{ClassicalSheetsthm} below) so that Theorem \ref{BCGeometricQthm} applies. 
    
Let $S$ be a sheet of $\mathfrak{gl}_n$ associated to a Levi subgroup $L$, with partition ${\bf m} = (m_1 \geq ... \geq m_r)$, and let ${\bf n} = (n_1 \geq ... \geq n_s)$ be its conjugate partition. Then $\mathfrak{z}$ is a vector space of dimension $r=n_1$, and $W_L$ is the product 
    $$W_L = \prod_{i=1}^sW_i,$$
    where $W_i=\Sym_{l_i}$ is the symmetric group on $l_i=n_i - n_{i+1}$ elements.
    
    Moreover, there is a direct sum decomposition
    \begin{equation}\label{GLncentredecompeqn}
        \mathfrak{z} = \bigoplus_{i=1}^s\mathfrak{z}_i
    \end{equation}
    such that $\mathfrak{z}_i$ has dimension $l_i$, and the action of $W_L$ on $\mathfrak{z}$ decomposes into coordinate permutation actions of $W_i$ on $\mathfrak{z}_i$. Thus, $\mathfrak{c}_S = \mathfrak{z}/W_L$ is an affine space with a product decomposition
    \begin{equation}\label{GLnQuotientdecompeqn}
        \mathfrak{c}_S = \prod_{i=1}^s\mathfrak{c}_i = \prod_{i=1}^s\Sym^{l_i}(\mathbb{A}^1)
    \end{equation}
    arising from (\ref{GLncentredecompeqn}); $\Sym^{l_i}(\mathbb{C})$ denotes the $l_i$-th symmetric product of the variety $\mathbb{A}^1$, which is the space of unordered $l_i$-tuples of complex numbers. The map $\chi_S:S \rightarrow \mathfrak{c}_S$ can also be decomposed into its constituents $\chi_i:S \rightarrow \mathfrak{c}_i$. It is straightforward to describe $\chi_i:S \rightarrow \mathfrak{c}_i$ explicitly on the open subset $S^{ss} \subseteq S$ of semisimple elements. For any $x \in S^{ss}$, $x$ has exactly $l_m$ distinct eigenvalues of multiplicity $m$ for each value of $m \in \{1, \,  ...,\, s\}$; then
    \begin{equation}\label{ChevalleyMapGLneqn}
        \chi_i(x) = \Sym^{l_i}(\boldsymbol{ \lambda}_i),
    \end{equation}
    where $\boldsymbol{ \lambda}_i$ is the tuple of eigenvalues of $x$ which occur with multiplicity $i$.

    For any element $x \in \mathfrak{gl}_n$, the group centraliser of $x$ is connected; so in particular, the component group of any nilpotent element in $\mathfrak{gl}_n$ is trivial. As a result, a choice of Katsylo slice $\mathfrak{K}$ determines a section $\mathfrak{c}_S \hookrightarrow S$ of $\chi_S$.

    While the sheet itself is non-singular, so in particular normal, its closure $\overline{S}$ in $\mathfrak{gl}_n$ is not normal and thus the map $\nu_S:\mathfrak{c}_S \rightarrow \mathfrak{t}/W$ of Remark \ref{GeometricQuotientrmk} may fail to be injective. This occurs, e.g. for the subregular sheet in $\mathfrak{gl}_4$ \cite[6.1]{Borho2}.
\end{example}
The sheets of $\mathfrak{gl}_n$ are particularly well-behaved and display many similarities with the regular sheet, but this is somewhat unrepresentative of the situation for general $G$. The low-rank example of $\Sp_4$ provides a glimpse of the complications which can arise in general.
\begin{example}\label{Sp4Sheetexm}
    Let $G=\Sp_4$. There are four Levi subgroups of $G$ up to conjugacy: the torus $T$, a copy of $\GL_2$, a copy of $\mathbb{G}_m \times \Sp_2$, and the full group $\Sp_4$. There are also four nilpotent orbits: the regular orbit $\mathcal O_{reg}$, the subregular orbit $\mathcal O_{sub}$, the minimal orbit $\mathcal O_{min}$, and the zero orbit $0$.
    
    \begin{table}[ht]
        \centering
        \caption{Sheets in $Sp_4$}
        \begin{tabular}{|c|c|c|c|c|}
        \hline
            Sheet&Decomposition data & Nilpotent orbit & dim $\mathcal I_S$ & dim $\mathfrak{c}_S$  \\
            \hline
             $\mathfrak{g}^{reg}$ & $(T,0)$& $\mathcal O_{reg}$ & $2$ & $2$ \\
             $S_{\Dix}$&$(\mathbb{G}_m \times \Sp_2,0)$ & $\mathcal O_{sub}$ & $4$ & $1$ \\
             $S_{\Dix}'$&$(\GL_2,0)$& $\mathcal O_{sub}$ & $4$ & $1$ \\
             $\mathcal O_{min}$&$(\Sp_4,\mathcal O_{min})$ & $\mathcal O_{min}$ & $6$ & $0$ \\
             $0$&$(\Sp_4,0)$ & $0$ & $10$ & $0$ \\
             \hline
        \end{tabular}
        \label{Sp4Sheettbl}
    \end{table}
    There are five sheets in $\mathfrak{sp}_4$, listed in Table \ref{Sp4Sheettbl}. The table gives decomposition data for the sheet, the nilpotent orbit contained in the sheet, the relative dimension of the centraliser $\mathcal I_S$, and the dimension of the adjoint quotient space $\mathfrak{c}_S$. In contrast with Example \ref{GLnSheetexm}, there is a non-Dixmier sheet, the  rigid orbit $\mathcal O_{min}$, and additionally, the sheets $S_{\Dix}$ and $S_{\Dix}'$ have non-trivial intersection along the subregular orbit $\mathcal O_{sub}$.

    As for $\mathfrak{gl}_n$, all of the sheets of $\mathfrak{sp}_4$ are non-singular by Theorem \ref{ClassicalSheetsthm} below, and the adjoint quotient spaces $\mathfrak{c}_S$ are affine spaces. However, the geometric quotient map $\chi_S:S \rightarrow \mathfrak{c}_S$ for $S=S_{\Dix}$ does not admit a section transverse to the $G$-action. This is a consequence of the non-triviality of the action of the component group $\Gamma$ of a nilpotent $e \in S$ on the Katsylo slice $\mathfrak{K}$ to $S_{\Dix}$ at $e$ (see Theorem \ref{KatsyloGeometricQuotientthm}). As we will see in Section \ref{AdjointQscn}, this phenomenon is significant for the behaviour of the Hitchin fibration, so we consider this particular case in greater detail.

    We choose a matrix representation for $\Sp_4$ as follows. Let
    $$J = \begin{pmatrix}
        0 & 0 & 0 & 1 \\
        0 & 0 & 1 & 0 \\
        0 & -1 & 0 & 0 \\
        -1 & 0 & 0 & 0
    \end{pmatrix}$$
    and define a non-degenerate skew-symmetric form $\Omega$ on $\mathbb{C}^4$ by $\Omega(v,w) = v^TJw$. Using the corresponding matrix representations for $\Sp_4$ and $\mathfrak{sp}_4$, we can describe the adjoint orbits contained in $S_{\Dix}$ explicitly as the orbits of
    $$\begin{pmatrix}
        t & 0 & 0 & 0 \\
        0 & 0 & 0 & 0 \\
        0 & 0 & 0 & 0 \\
        0 & 0 & 0 & -t
    \end{pmatrix}$$
    for all non-zero $t \in \mathbb{C}$ together with the orbit of
    $$e =\begin{pmatrix}
        0 & 0 & 1 & 0 \\
        0 & 0 & 0 & 1 \\
        0 & 0 & 0 & 0 \\
        0 & 0 & 0 & 0
    \end{pmatrix}.$$
    We can complete $e$ to the $\mathfrak{sl}_2$-triple $(e,h,f)$, where
    $$h = \begin{pmatrix}
        1 & 0 & 0 & 0\\
        0 & 1 & 0 & 0 \\
        0 & 0 & -1 & 0 \\
        0 & 0 & 0 & -1
    \end{pmatrix}
    \text{and }
    f = \begin{pmatrix}
        0 & 0 & 0 & 0 \\
        0 & 0 & 0 & 0 \\
        1 & 0 & 0 & 0 \\
        0 & 1 & 0 & 0
    \end{pmatrix}.$$
    The Katsylo slice $\mathfrak{K}$ for $S_{\Dix}$ associated with this triple is given by
    \begin{equation}\label{Sp4Katsylo1eqn}
        \mathfrak{K} = \left\{ x_t \, |\,  t \in \mathbb{C} \right\},
    \end{equation}
    where
    \begin{equation}\label{Sp4Katsylo2eqn}
        x_t = \frac{1}{4}\begin{pmatrix}
            2t & 0 & 4 & 0 \\
            0 & -2t & 0 & 4 \\
            t^2 & 0 & 2t & 0 \\
            0 & t^2 & 0 & -2t
        \end{pmatrix}.
    \end{equation}
    The component group $\Gamma$ of $e$ is a group of order $2$, generated by the class of
    $$s= \begin{pmatrix}
        0 & 1 & 0 & 0 \\
        1 & 0 & 0 & 0 \\
        0 & 0 & 0 & 1 \\
        0 & 0 & 1 & 0
    \end{pmatrix},$$
    and since $\Ad_s(x_t) = x_{-t}$, $\Gamma$ acts non-trivially on $\mathfrak{K}$.
    \end{example}

    We conclude this section by stating the main theorem of \cite{ImHof} which guarantees smoothness for sheets in classical Lie algebras (this is not true in general, see e.g. \cite[8.11]{Slodowy}). We make precise our terminology: by a \emph{classical Lie algebra} we mean a reductive Lie algebra whose semisimple part is a sum of copies of $\mathfrak{sl}_n$, $\mathfrak{so}_n$ and $\mathfrak{sp}_{2n}$. By a \emph{classical group} we mean any reductive group $G$ with classical Lie algebra $\Lie(G)$. This is broader than the usual definition.

    \begin{theorem}\label{ClassicalSheetsthm}\cite{ImHof}
        If $S$ is a sheet in a classical Lie algebra, $S$ is non-singular.
    \end{theorem}
    
    For the main body of this work, we will focus exclusively on sheets which are non-singular. This ensures that $\chi_S:S \rightarrow \mathfrak{c}_S$ is a geometric quotient (by Theorem \ref{BCGeometricQthm}) and for any Katsylo slice $\mathfrak{K}$ there is a canonical identification $\mathfrak{K}/\Gamma \cong \mathfrak{c}_S$ (by Remark \ref{KatsyloGQrmk}).
\section{The adjoint quotient stack of a sheet}\label{AdjointQscn}
For a non-singular sheet $S$, we describe the structure of the map $\chi_{S,G}:[S/G] \rightarrow \mathfrak{c}_S$ induced by the geometric quotient map $\chi_S$ of Theorem \ref{BCGeometricQthm}. This dictates the behaviour of the Hitchin fibration on corresponding loci in the Higgs moduli stack, which can be viewed as a particular instance of a generalised Hitchin fibration as in \cite{Ngo3}. As an intermediate step, we prove a result characterising the smoothness properties of the centraliser group scheme on a sheet.

\subsection{The smooth centraliser}\label{SmoothCentralisersbn}
It is well known that the centraliser is always smooth over the regular sheet (see e.g. \cite{DG}). This does not extend to the general case, but we can consider instead a finite-index subgroup scheme of the centraliser which is smooth. We will use the technology and terminology of groupoid schemes and quotients throughout this section (see Sections 1 and 2 of \cite{KM}).

Motivated by smoothness arguments in the regular case using the Kostant section (see e.g. \cite[Section 3.3]{Riche}), we first consider the restriction of the centraliser to a Katsylo slice. We begin by stating a simple but important lemma of \cite{PT}.

\begin{lemma}\cite[Remark 6(b)]{PT}\label{KatsyloSmoothlem}
    Any Katsylo slice $\mathfrak{K}$ for a non-singular sheet $S$ is non-singular and irreducible.
\end{lemma}

In fact, if $S$ is non-singular we show that the Katsylo slice is an affine space in Corollary \ref{KatsyloAffinecor} below.

Let $S$ be a non-singular sheet of $\mathfrak{g}$ and fix a choice of $\mathfrak{sl}_2$-triple $(e,h,f)$ with $e \in S$, using the notation of (\ref{SL2Tripleeqn}). Let $\mathfrak{K}$ be the Katsylo slice to $e$ as in Definition \ref{KatsyloSlicedef} and let $A$ be the reductive centraliser as in Definition \ref{ReductiveCentraliserdef}.

\begin{definition}\label{KatsyloGpdef}
    The \emph{Katsylo group} for $S$ (with respect to the $\mathfrak{sl}_2$-triple $(e,h,f)$) is the group $F=A/N$ where $N$ is the kernel of the $A$-action on $\mathfrak{K}$ defined in Theorem \ref{RedCentraliseractionthm}.
\end{definition}

\begin{remark}\label{KatsyloGprmk}
    By Theorem \ref{RedCentraliseractionthm}, $F$ is a quotient of the component group $\Gamma$ of $e$ (defined by  (\ref{ComponentGroupeqn})); so in particular $F$ is finite. By construction, the action of $A$ on $\mathfrak{K}$ factors through a faithful action of $F$ on $\mathfrak{K}$, and $\mathfrak{K}/\Gamma = \mathfrak{K}/F$. Since the $A$-action on $\mathfrak{K}$ commutes with the Kazhdan action, so does the $F$-action.
\end{remark}

We define the following $\mathfrak{K}$-group scheme from the $F$-action on $\mathfrak{K}$.

\begin{definition}\label{FInertiadef}
    The \emph{$F$-inertia group scheme} $\mathcal F$ is the stabiliser of the action groupoid 
    \begin{equation}\label{FInertiaeqn}
        F \times \mathfrak{K} \rightrightarrows \mathfrak{K}.
    \end{equation}
\end{definition}

\begin{remark}\label{FInertia1rmk}
     For any point $x \in \mathfrak{K}$, $\mathcal F_x = \Stab_F(x)$; moreover, since $F$ is finite, and the action of $F$ on $\mathfrak{K}$ is faithful, there is a dense open set $\mathfrak{K}^\circ \subseteq \mathfrak{K}$ such that $\mathcal F_{\mathfrak{K}^\circ}$ is the trivial group scheme.

    As a variety, we can decompose $\mathcal F$ as a disjoint union of connected components
    \begin{equation}\label{FInertia2eqn}
        \mathcal F = \coprod_{a \in F} \mathcal F^a
    \end{equation}
    where $\mathcal F^{\text{id}}$ is the image of the identity section, and for any $a \in F\backslash\{\text{id}\}$, the component $\mathcal F^a$ is supported on a proper closed subvariety of $\mathfrak{K}$ containing the nilpotent $e \in \mathfrak{K}$. In particular, $\mathcal F$ is not flat over $\mathfrak{K}$ unless $F$ is trivial.
\end{remark}

The following proposition is the key to understanding the broad geometric structure of the centraliser $\mathcal I_S$ on $S$.

\begin{proposition}\label{QuasiSteinHomprp}
    There is a smooth surjective homomorphism $\sigma:\mathcal I_{\mathfrak{K}} \rightarrow \mathcal F$ of group schemes over $\mathfrak{K}$.
\end{proposition}

We split the proof of the proposition into a number of intermediate lemmas. We first define a related scheme which contains the centraliser on $\mathfrak{K}$ as a closed subscheme.

\begin{definition}\label{RestrictedActiondef}
    The \emph{restricted action scheme} $\mathcal R$ is the scheme defined by the Cartesian diagram
    \begin{equation}\label{RestrictedAction1eqn}
        \begin{tikzcd}
            \mathcal R \arrow[d] \arrow[r] & G \times \mathfrak{K} \arrow[d, "\Ad"] \\
            \mathfrak{K} \arrow[r, hook]   & S.                                    
        \end{tikzcd}
    \end{equation}
\end{definition}

\begin{remark}\label{RestrictedActionrmk}
    The restricted action scheme $\mathcal R$ is non-singular. The diagram (\ref{RestrictedAction1eqn}) defines two natural maps $s:\mathcal R \rightarrow \mathfrak{K}$ (induced by the upper horizontal arrow) and $t:\mathcal R \rightarrow \mathfrak{K}$ (which is the left-hand arrow) which are smooth morphisms; indeed, the diagram
    \begin{equation}\label{RestrictedAction2eqn}
        \mathcal R \doublerightarrow{s}{t} \mathfrak{K} 
    \end{equation}
    has the structure of a groupoid scheme, which is the restriction of the action groupoid
    \begin{equation}\label{RestrictedAction3eqn}
        G \times S \doublerightarrow{\pi_2}{\Ad} S 
    \end{equation}
    along the inclusion $\mathfrak{K} \hookrightarrow S$. Hence, $\mathcal I_{\mathfrak{K}}$ embeds into $\mathcal R$ as its stabiliser group scheme.

    The group $A$ acts on $S$ by conjugation and there is also an $A$-action on $G \times \mathfrak{K}$ given by the left multiplication action on $G$. These actions together induce an $A$-action on $\mathcal R$. With respect to this action, the map $s:\mathcal R \rightarrow \mathfrak{K}$ is $A$-invariant and the map $t:\mathcal R \rightarrow \mathfrak{K}$ is $A$-equivariant.
\end{remark}

\begin{lemma}\label{RADecomplem}
    Consider the $\mathfrak{K}$-scheme $s:\mathcal R\rightarrow \mathfrak{K}$ and let $\mathfrak{K}^\circ \subseteq \mathfrak{K}$ be the open set over which $F$ acts freely (as in Remark \ref{FInertia1rmk}). There is an isomorphism $\phi:\mathcal R_{\mathfrak{K}^\circ} \xrightarrow{\sim} F \times \mathcal I_{\mathfrak{K}^\circ}$ of $\mathfrak{K}^\circ$-schemes such that $\phi$ maps $\mathcal I_{\mathfrak{K}^\circ}$ to $\{\text{\emph{id}}_F\} \times \mathcal I_{\mathfrak{K}^\circ}$.
\end{lemma}

\begin{proof}
 Choose representatives $a_A \in A$ for each element $a \in F = A/N$, choosing $\text{id}_A$ as the representative for $\{\text{id}_F\}$. Then, viewing $\mathcal I_{\mathfrak{K}^\circ}$ as a subscheme of $\mathcal R$, the $A$-action defines a map $A \times \mathcal I_{\mathfrak{K}^\circ} \rightarrow \mathcal R_{\mathfrak{K^\circ}}$, and thus a map $\psi:F \times \mathcal I_{\mathfrak{K}^\circ} \rightarrow \mathcal R_{\mathfrak{K}^\circ}$, via the choice of representatives $f_A \in A$. Since the $F$-action is free on $\mathfrak{K^\circ}$, $\psi$ is injective; and by Theorem \ref{RedCentraliseractionthm} and the definitions of $F$ and $\mathcal R$, $\psi$ is surjective. Since $\mathcal R_{\mathfrak{K}^\circ}$ is non-singular, $\psi$ is an isomorphism by Zariski's main theorem, and we set $\phi$ to be its inverse.
 \end{proof}

\begin{lemma}\label{RAQuasiSteinlem}
    Let $\hat{\sigma}_{\mathfrak{K}^\circ}:\mathcal R_{\mathfrak{K}^{\circ}} \rightarrow F \times \mathfrak{K}^{\circ}$ be defined as the composition
    \begin{equation}\label{RAQuasiStein1eqn}
\begin{tikzcd}
\hat{\sigma}_{\mathfrak{K}^\circ}:\mathcal R_{\mathfrak{K}^\circ} \arrow[r, "\phi"] & F \times \mathcal I_{\mathfrak{K}^\circ} \arrow[r, "\pi_\mathcal I"] & F \times \mathfrak{K}^{\circ},
\end{tikzcd}
    \end{equation}
    where $\phi$ is the morphism of Lemma \ref{RADecomplem}, and $\pi_{\mathcal I}$ is induced by the structure map $\mathcal I_{\mathfrak{K}^\circ}\rightarrow \mathfrak{K}^\circ$. Then $\hat{\sigma}_{\mathfrak{K}^\circ}$ extends to a smooth surjective morphism $\hat{\sigma}:\mathcal R \rightarrow F \times \mathfrak{K}$.
\end{lemma}

\begin{proof}
We can construct $\hat{\sigma}$ on each connected component of $\mathcal R$ individually. Let $X$ be a connected component of $\mathcal R$. Since $s:\mathcal R \rightarrow \mathfrak{K}$ is smooth, so is the restriction $s|_X:X \rightarrow \mathfrak{K}$. In particular, the support of $X$ is open in $\mathfrak{K}$, so $X_{\mathfrak{K}^\circ}$ is non-empty, and since $X$ is connected and non-singular, so is $X_{\mathfrak{K}^\circ}$. In particular, $\hat{\sigma}_{\mathfrak{K}^\circ}$ must map $X_{\mathfrak{K}^\circ}$ to $\{a\} \times \mathfrak{K}^\circ$ for some fixed $a \in F$, and this map is naturally identified with the structure map for $X_{\mathfrak{K}^\circ}$ as a $\mathfrak{K}^\circ$-scheme, as $\phi$ is an isomorphism of $\mathfrak{K}^\circ$-schemes. Hence, the structure map $s|_X:X\rightarrow \mathfrak{K}$ induces the extension $\hat{\sigma}|_X:X \rightarrow \{a\} \times \mathfrak{K}$, which is smooth since $s|_X$ is smooth.

To show that $\hat{\sigma}$ is surjective, we let $\mathfrak{K}_{\text{id}}$ be the identity section of $\mathcal I_{\mathfrak{K}}$, and let $\mathfrak{K}_{\text{id}}^\circ$ be its restriction to $\mathfrak{K}^\circ$. By the construction of $\phi$, $\hat{\sigma}_{\mathfrak{K}^\circ}$ maps $a_A\cdot\mathfrak{K}_{\text{id}}^\circ \subseteq \mathcal R_{\mathfrak{K}^\circ}$ surjectively to $\{a\} \times \mathfrak{K}^\circ$ (where $a_A$ is as in the proof of Lemma \ref{RADecomplem}). So by continuity of $\hat{\sigma}$, it must map $a_A\cdot\mathfrak{K}_{\text{id}}$ surjectively to $\{a\} \times \mathfrak{K}$; so $\hat{\sigma}$ is surjective.
\end{proof}

\begin{lemma}\label{RAQSGroupoidlem}
    The map $\hat{\sigma}:\mathcal R \rightarrow F \times \mathfrak{K}$ defines a morphism of groupoid schemes.
\end{lemma}

\begin{proof}
We first show that the restriction $\hat{\sigma}_{\mathfrak{K}^\circ}$ of $\sigma$ to $\mathcal R_{\mathfrak{K}^\circ}$ defines a morphism of groupoid schemes. By the construction of $\phi$, the $A$-invariance of $s$ and $A$-equivariance of $t$, $\hat{\sigma}_{\mathfrak{K}^\circ}$ respects the source and target maps of (\ref{RestrictedAction2eqn}) and (\ref{FInertiaeqn}). So we need only check that $\hat{\sigma}_{\mathfrak{K}^\circ}$ respects the composition morphisms for the groupoids; we will denote these both by $c$. Since the $F$-action is free on $\mathfrak{K}^\circ$, a point in $F \times \mathfrak{K}^\circ$ is determined by its images under the source and target maps of (\ref{FInertiaeqn}). But then for any composable points $x, y \in \mathcal R_{\mathfrak{K}^\circ}$, $\hat{\sigma}(c(x,y))$ and $c(\hat{\sigma}(x),\hat{\sigma}(y))$ both have source $s(y)$ and target $t(x)$, so must be the same point. Hence, $\hat{\sigma}_{\mathfrak{K}^\circ}$ defines a morphism of groupoid schemes.

Finally, to conclude that $\hat{\sigma}$ defines a morphism of groupoid schemes we use a continuity argument. We elaborate on this for the compatibility of composition; the compatibility of the source and target maps can be shown in a similar way.

We need to show that the maps $\hat{\sigma} \circ c$ and $c \circ (\hat{\sigma} \times \hat{\sigma})$ are equal. We consider the Cartesian diagram
\begin{equation}\label{RAQSGroupoideqn}
\begin{tikzcd}
{\mathcal R \times_{s,t} \mathcal R} \arrow[r, "\pi_s"] \arrow[d, "\pi_t"] & \mathcal R \arrow[d, "t"] \\
\mathcal R \arrow[r, "s"]                                                  & \mathfrak{K}             
\end{tikzcd} 
\end{equation}
and view $\mathcal R \times_{s,t}\mathcal R$ as a $\mathfrak{K}$-scheme via $t \circ \pi_s$. All the arrows in (\ref{RAQSGroupoideqn}) are smooth, so $\mathcal R \times_{s,t}\mathcal R$ is smooth as a $\mathfrak{K}$-scheme. So $\mathcal R_{\mathfrak{K}^\circ} \times_{s,t} \mathcal R_{\mathfrak{K}^\circ} = (\mathcal R \times_{s,t} \mathcal R)_{\mathfrak{K}^\circ}$ is dense in $\mathcal R \times_{s,t} \mathcal R$, and we have already shown that $\hat{\sigma} \circ c$ and $c \circ (\hat{\sigma} \times \hat{\sigma})$ agree on $\mathcal R_{\mathfrak{K}^\circ} \times_{s,t} \mathcal R_{\mathfrak{K}^\circ}$. Hence, since $\mathfrak{K}$ is separated, these maps agree on all of $\mathcal R \times_{s,t}\mathcal R$.
\end{proof}

\begin{proof}[Proof of Proposition \ref{QuasiSteinHomprp}] 
Since groupoid morphisms respect stabilisers, $\hat{\sigma}$ maps $\mathcal I_{\mathfrak{K}}$ to $\mathcal F$, i.e. defines a group homomorphism $\sigma:\mathcal I_{\mathfrak{K}} \rightarrow \mathcal F$ by restriction. Moreover, the diagram
\begin{equation}\label{QuasiSteinHomeqn}
\begin{tikzcd}
\mathcal I_{\mathfrak{K}} \arrow[r, "\sigma"] \arrow[d, "\iota"'] & \mathcal F \arrow[d]  \\
\mathcal R \arrow[r, "\hat{\sigma}"]                              & F \times \mathfrak{K}
\end{tikzcd}
\end{equation}
is Cartesian. Hence, $\sigma$ is smooth and surjective.
\end{proof}

\begin{corollary}\label{SmoothCentralisercor}
    Denote 
    \begin{equation}\label{SmoothCentraliser1qn}
        \mathcal I^{sm}_\mathfrak{K} := \Ker(\sigma);
    \end{equation}
    then $\mathcal I^{sm}_{\mathfrak{K}}$ is a smooth normal closed subgroup scheme of the $\mathfrak{K}$-group scheme $\mathcal I_{\mathfrak{K}}$.
\end{corollary}

\begin{proof} 
Smoothness of $\mathcal I^{sm}_{\mathfrak{K}}$ over $\mathfrak{K}$ follows from the Cartesian diagram
\begin{equation}\label{SmoothCentraliser2eqn}
\begin{tikzcd}
\mathcal I^{sm}_\mathfrak{K} \arrow[r] \arrow[d] & \mathcal I_{\mathfrak{K}} \arrow[d, "\sigma"] \\
\mathfrak{K} \arrow[r, "\text{id}_\mathcal F"]           & \mathcal F                                   
\end{tikzcd}
\end{equation}
where $\text{id}_{\mathcal F}$ is the identity section of $\mathcal F$. The fact that $\mathcal I^{sm}_\mathfrak{K}$ is normal and closed in $\mathcal I_{\mathfrak{K}}$ is automatic since it is the kernel of a homomorphism of group schemes.
\end{proof}

We now wish to extend these considerations to the full sheet $S$. We first consider the pullback of $\mathcal I_S$ to $G \times \mathfrak{K}$ along the action map. There is a pullback diagram of group schemes
\begin{equation}\label{CentraliserPullbackeqn}
\begin{tikzcd}
G \times \mathcal I_{\mathfrak{K}} \arrow[r] \arrow[d] & \mathcal I_S \arrow[d] \\
G \times \mathfrak{K} \arrow[r, "Ad"]                  & S                     
\end{tikzcd}
\end{equation}
where the top arrow is defined by the $G$-action map on $\mathcal I_S$. This suggests that the desired extension of the group scheme $\mathcal I^{sm}_{\mathfrak{K}}$ defined in Corollary \ref{SmoothCentralisercor} should be constructed via faithfully flat descent. 

\begin{proposition}\label{SmCentraliserDescentprp}
    The closed group subscheme $G\times \mathcal I_{\mathfrak{K}}^{sm}$ of $G \times \mathcal I_{\mathfrak{K}}$ (as group schemes over $G\times \mathfrak{K}$) descends to a closed subgroup scheme $\mathcal I_S^{sm}$ of $\mathcal I_S$ (as group schemes over $S$) along the action map $\Ad:G \times \mathfrak{K} \rightarrow S$.
\end{proposition}

In order for the descent to work, we need some auxiliary lemmas.

\begin{lemma}\label{CentraliserCompslem}
    As a subscheme, $\mathcal I_{\mathfrak{K}}^{sm}$ is a union of connected components of $\mathcal I_{\mathfrak{K}}$. Any connected component of $\mathcal I_{\mathfrak{K}}$ which is not contained in $\mathcal I_{\mathfrak{K}}^{sm}$ has support on a subset of $\mathfrak{K}$ of codimension at least 1.
\end{lemma}

\begin{proof}
These statements follow from the definition of $\mathcal I_{\mathfrak{K}}^{sm}$ and Remark \ref{FInertia1rmk}.
\end{proof}

\begin{lemma}\label{RedCentSCActionlem}
    The $A$-action on $\mathcal I_{\mathfrak{K}}$ (defined via the $G$-action on $\mathcal I_S$) restricts to an $A$-action on $\mathcal I_{\mathfrak{K}}^{sm}$.
\end{lemma}

\begin{proof}
Since $A \times \mathcal I_{\mathfrak{K}}^{sm}$ is smooth over $\mathfrak{K}$, all of its connected components have dense support in $\mathfrak{K}$. If $X$ is a component of $A \times \mathcal I_{\mathfrak{K}}^{sm}$ supported on $\supp(X) \subseteq \mathfrak{K}$, then the image of $X$ under the action morphism has support on the image of $\supp(X)$ under an automorphism of $\mathfrak{K}$. So in particular the image of $X$ also has dense support in $\mathfrak{K}$, and must be contained in $\mathcal I_{\mathfrak{K}}^{sm}$ by Lemma \ref{CentraliserCompslem}.
\end{proof}

\begin{proof}[Proof of Proposition \ref{SmCentraliserDescentprp}] Consider the Cartesian diagram
\begin{equation}\label{SmCentraliserDescent1eqn}
\begin{tikzcd}
(G \times \mathfrak{K}) \times_S (G \times \mathfrak{K}) \arrow[r, "\pi_1"] \arrow[d, "\pi_2"'] & G \times \mathfrak{K} \arrow[d, "\Ad"] \\
G \times \mathfrak{K} \arrow[r, "\Ad"]                                                           & S                                    
\end{tikzcd}
\end{equation}
Then the diagram (\ref{CentraliserPullbackeqn}) induces a canonical isomorphism $\psi:\pi_1^*(G \times \mathcal I_{\mathfrak{K}}) \rightarrow \pi_2^*(G \times \mathcal I_{\mathfrak{K}})$. Since the action morphism is smooth and surjective, to descend $G \times \mathcal I_{\mathfrak{K}}^{sm}$ as a closed subgroup scheme of $G \times \mathcal I_{\mathfrak{K}}$, it suffices to show that $\psi$ maps $\pi_1^*(G \times \mathcal I_{\mathfrak{K}}^{sm})$ to $\pi_2^*(G \times \mathcal I_{\mathfrak{K}}^{sm})$ (note that the cocycle condition for $G \times \mathcal I_{\mathfrak{K}}^{sm}$ will be satisfied automatically since it is satisfied for $G \times \mathcal I_{\mathfrak{K}}$).

We can make explicit identifications of both $\pi_i^*(G \times \mathcal I_{\mathfrak{K}})$ with
\begin{equation}\label{SmCentraliserDescent2eqn}
    (G \times \mathcal I_{\mathfrak{K}}) \times _S (G \times \mathfrak{K})
\end{equation}
such that $\psi$ is identified with the map whose action on $\mathbb{C}$-points is
\begin{equation}\label{SmCentraliserDescent3eqn}
    ((g_1,(h,x_1)),(g_2,x_2)) \mapsto ((g_2, g_2^{-1}g_1 \cdot (h,x_1)),(g_1,x_1)),
\end{equation}
where we have $g_1,g_2 \in G$, $x_1,x_2 \in \mathfrak{K}$ and $h \in C_G(x_1)$ with
\begin{equation}\label{SmCentraliserDescent4eqn}
    \Ad_{g_1}(x_1)=\Ad_{g_2}(x_2).
\end{equation}
By Theorem \ref{RedCentraliseractionthm}, there exist $a \in A$ and $k \in C_G(x_1)$ such that $g_2^{-1}g_1 = ak$. Then, by Corollary \ref{SmoothCentralisercor} and Lemma \ref{RedCentSCActionlem}, the subscheme
\begin{equation}\label{SmCentraliserDescent5eqn}
    (G \times \mathcal I_{\mathfrak{K}}^{sm}) \times _S (G \times \mathfrak{K})
\end{equation}
of (\ref{SmCentraliserDescent2eqn}) is stable under the map defined by (\ref{SmCentraliserDescent3eqn}); this gives the required statement.
\end{proof}

The group scheme $\mathcal I_S^{sm}$ has a number of desirable properties.

\begin{proposition}\label{SmCentraliserPropsprp}
    $\mathcal I_S^{sm}$ is a smooth closed normal subgroup scheme of $\mathcal I_S$, and is the maximal smooth subgroup scheme of $\mathcal I_S$, i.e. if $\mathcal H$ is a smooth subgroup scheme of $\mathcal I_S$ then $\mathcal H$ is a subgroup scheme of $\mathcal I^{sm}_S$. The $G$-action on $\mathcal I_S$ restricts to a $G$-action on $\mathcal I_S^{sm}$. 
\end{proposition}

\begin{proof}
The first statement follows from Corollary \ref{SmoothCentralisercor}, since these properties are preserved under fppf descent. For maximality, since $\mathcal I_S^{sm}$ is a closed subgroup of $\mathcal I_S$, it suffices to check that any smooth subgroup $\mathcal H$ of $\mathcal I_S$ is contained in $\mathcal I_S^{sm}$ over a dense open subset of $S$; but this is automatic since over $S^\circ = \Ad(G)(\mathfrak{K}^\circ)$ (which is open by Proposition \ref{KatsyloSliceprp}), $\mathcal I^{sm}_{S^\circ} = \mathcal I_{S^\circ}$. 

For the final statement, it suffices to observe that in the diagram (\ref{CentraliserPullbackeqn}) the $G$-action on $\mathcal I_S$ pulls back to the action on $G \times \mathcal I_{\mathfrak{K}}$ induced by left multiplication on $G$, which clearly leaves the subscheme $G \times \mathcal I_{\mathfrak{K}}^{sm}$ stable.
\end{proof}

Proposition \ref{SmCentraliserPropsprp} shows that $\mathcal I_S^{sm}$ is a centraliser for the action of $G$ on $S$ in the category of smooth $S$-group schemes; as such we refer to $\mathcal I_S^{sm}$ as the \emph{smooth centraliser} on $S$. We observe the following corollary of the proposition which encapsulates the relationship between the centraliser $\mathcal I_S$ and the Katsylo group $F$.

\begin{corollary}\label{CentraliserKatsylocor}
    The $S$-group scheme $\mathcal I_S$ is smooth if and only if $F$ is trivial.
\end{corollary}

\begin{proof} 
If $F$ is trivial, then so is $\mathcal F$, so by construction $\mathcal I_{\mathfrak{K}}^{sm} =\mathcal I_{\mathfrak{K}}$ and $\mathcal I_S^{sm}=\mathcal I_S$. 

Conversely, if $F$ is non-trivial, then $\mathcal F$ is non-trivial, so since the homomorphism $\sigma$ of Proposition \ref{QuasiSteinHomprp} is surjective, $\mathcal I_\mathfrak{K}^{sm}$ is a proper subscheme of $\mathcal I_{\mathfrak{K}}$. So $\mathcal I_S^{sm}$ is a proper subscheme of $\mathcal I_S$, so by Proposition \ref{SmCentraliserPropsprp}, $\mathcal I_S$ cannot be smooth.
\end{proof}

The constructions above are $\mathbb{G}_m$-equivariant with respect to the relevant actions.

\begin{proposition}\label{SmCentGmEquivprp}
    The Kazhdan action on $\mathcal I_{\mathfrak{K}}$ restricts to a $\mathbb{G}_m$-action on $\mathcal I_{\mathfrak{K}}^{sm}$. Similarly, the scaling action on $\mathcal I_S$ restricts to a $\mathbb{G}_m$-action on $\mathcal I_S^{sm}$.
\end{proposition}
\begin{proof}
The proofs are the same as that of Lemma \ref{RedCentSCActionlem}.
\end{proof}

\subsection{The adjoint quotient as a gerbe}\label{AdjointQGerbesbn}

We can use the considerations of the previous subsection to describe the structure of the quotient stack $[S/G]$. As before, let $S$ be a non-singular sheet of $\mathfrak{g}$. First, we note that since the map $\chi_S:S \rightarrow \mathfrak{c}_S$ of Theorem \ref{BCGeometricQthm} is a geometric quotient for the $G$-action, it induces a map $\chi_{S,G}:[S/G] \rightarrow \mathfrak{c}_S$ which is bijective on $\mathbb{C}$-points. This suggests a gerbe structure for $\chi_{S,G}$; we recall the definition.

\begin{definition}\cite[Section 2]{Giraud}\label{Gerbedef}
    A morphism of algebraic stacks $f: \mathcal X \rightarrow \mathcal Y$ is a \emph{gerbe} if there is an fppf cover $g:Z \rightarrow \mathcal Y$ and a sheaf of groups $\mathcal G$ on $Z$ such that there is an isomorphism
    \begin{equation}\label{Gerbeeqn}
        Z \times_\mathcal Y \mathcal X \cong {\bf B}\mathcal G
    \end{equation}
    where ${\bf B}\mathcal G$ is the classifying stack for $\mathcal G$, i.e. the stack quotient $[Z/\mathcal G]$ for the trivial action of $\mathcal G$ on $Z$.

    Such a $Z$ is called a \emph{trivialisation} of the gerbe, and a gerbe is called \emph{trivial} if there is a section $s:\mathcal Y\rightarrow \mathcal X$.
\end{definition}

\begin{remark}\label{Gerbermk}
    We have combined \cite[Définition 2.1.1]{Giraud} and \cite[Corollaire 2.2.6]{Giraud} in the above definition; \cite[Corollaire 2.2.6]{Giraud} also implies that if the gerbe is trivial any fppf cover $g:Z\rightarrow \mathcal Y$ is a trivialisation, and moreover there is a sheaf of groups $\mathscr G$ on $\mathcal Y$ such that $\mathcal G = g^*\mathscr G$ (in the notation of (\ref{Gerbeeqn})).

    If the gerbe is non-trivial, there may in general be no sheaf of groups $\mathscr G$ for which $\mathcal G$ is the pullback of $\mathscr G$ on $\mathcal Y$ for any given trivialisation; however, if such a sheaf $\mathscr G$ exists, we call it the \emph{structure group} of the gerbe. If $\mathscr G$ is commutative, we say that the gerbe is \emph{banded by $\mathscr G$}. There is a notion of a band for non-abelian gerbes but we will not use it here.
\end{remark}

If $S$ is the regular sheet, the map $\chi_{S,G}$ is a gerbe (by \cite[Proposition 3.5]{Ngo1}). For general $S$, this is no longer the case, but we can replace $\mathfrak{c}_S$ by a Deligne-Mumford enhancement over which $[S/G]$ is a gerbe. We use the process of \emph{rigidification}, as defined in Appendix A of \cite{AOV}.

The smooth centraliser $\mathcal I^{sm}_S$, defined by Proposition \ref{SmCentraliserDescentprp}, descends under the quotient map $S \rightarrow [S/G]$ to a closed subgroup stack of the inertia stack $\mathcal I_{S,G}$ for $[S/G]$; moreover $\mathcal I_{S,G}^{sm}$ is representable by schemes over $[S/G]$. Thus $\mathcal I_{S,G}^{sm}$ satisfies the conditions of  \cite[Theorem A.1]{AOV}, and we can make the following definition.

\begin{definition}\label{StackyChevalleyBasedef}
We define the \emph{$S$-Chevalley base} $\mathcal B$ for $S$ to be the algebraic stack obtained by the rigidification of $[S/G]$ by $\mathcal I_{S,G}^{sm}$. We will denote the rigidification map by $\rho_{S,G}:[S/G] \rightarrow \mathcal B$.
\end{definition}

\begin{remark}\label{RegularQuotientrmk}
    The $G$-action on the closure $\overline{S}$ lifts to an action on the normalisation $\tilde{S}$ of $\overline{S}$, and since $S$ is non-singular, $S$ can be identified with the points in $\tilde{S}$ which are regular for the $G$-action (i.e. have minimal centraliser dimension). In the terminology of \cite[Section 4.2]{Ngo3}, the $S$-Chevalley base is the regular quotient for the action of $G$ on $\tilde{S}$. In general, the regular quotient construction requires a choice of open flat subgroup scheme of the centraliser $\mathcal I_S$, but in this case the choice is canonical.

    If $S$ is not Dixmier, $\tilde{S}$ fails the \emph{Luna-Richardson} criterion of \cite[Section 4.1]{Ngo3} (i.e. there is no closed regular orbit in $\tilde{S}$), but still admits an ``invariant-theoretic'' description for the GIT quotient $\mathfrak{c}_S = \tilde{S}//G$ by Proposition \ref{BorhoChevalleyprp}.
\end{remark}

\begin{proposition}\label{StackyChevalleyBaseprp}
The map $\rho_{S,G}:[S/G] \rightarrow \mathcal B$ is a gerbe, which is smooth as a morphism of stacks. In particular, for any scheme $X$ and any morphism $f:X \rightarrow [S/G]$, there is a Cartesian diagram
\begin{equation}\label{StackyChevalleyBaseeqn}
\begin{tikzcd}
{\bf B} f^*\mathcal I_{S,G}^{sm}  \arrow[r] \arrow[d] & {[S/G]} \arrow[d,"\rho_{S,G}"] \\
X \arrow[r, "\rho_{S,G} \circ f"]                  & \mathcal B                     
\end{tikzcd}
\end{equation}
where ${\bf B} f^*\mathcal I_{S,G}^{sm}$ is the classifying stack for the $X$-group scheme $f^*\mathcal I_{S,G}^{sm}$.
\end{proposition}

\begin{proof} 
The first statement follows from \cite[Theorem A.1]{AOV}, while the second statement follows from \cite[Remark A.2]{AOV} after unravelling the definitions.
\end{proof}

\begin{remark}\label{StackyChevalleyMaprmk}
    It is unclear if the gerbe $\rho_{S,G}$ has a structure group in general (in the sense of Remark \ref{Gerbermk}).
    
    We will denote by $\rho_S:S \rightarrow \mathcal B$ the composition
\begin{equation}\label{StackyChevalleyMapeqn}
\begin{tikzcd}
\rho_S:S \arrow[r] & {[S/G]} \arrow[r, "\rho_{S,G}"] & \mathcal B,
\end{tikzcd}
\end{equation}
and refer to the map $\rho_S$ as the \emph{$S$-Chevalley map}. This map is smooth since it is a composition of smooth morphisms.
\end{remark}

\begin{proposition}\label{StackyCBDiagramsprp}
There is a factorisation of the map $\chi_{S,G}$ as
\begin{equation}\label{StackyCBDiagrams1eqn}
\begin{tikzcd}
\chi_{S,G}:{[S/G]} \arrow[r, "\rho_{S,G}"] & \mathcal B \arrow[r, "C"] & \mathfrak{c}_S,
\end{tikzcd}
\end{equation}
and a commutative diagram
\begin{equation}\label{StackyCBDiagrams2eqn}
\begin{tikzcd}
{[S/G]} \arrow[r, hook] \arrow[d, "{\rho_{S,G}}"'] & {[\mathfrak{g}/G]} \arrow[d, "\chi_G"] \\
\mathcal B \arrow[r, "\tilde{\nu}_S"]            & \mathfrak{c}                        
\end{tikzcd}
\end{equation}
where $\chi_G$ is the map induced by the usual Chevalley map $\chi:\mathfrak{g} \rightarrow \mathfrak{c}$, for $\mathfrak{c} =\mathfrak{t}/W$.
\end{proposition}

\begin{proof} 
The factorisation (\ref{StackyCBDiagrams1eqn}) exists by the construction of the rigidification in \cite[Theorem A.1]{AOV}. Define $\tilde{\nu}_S = \nu_S \circ C$ for $\nu_S$ defined as in Remark \ref{GeometricQuotientrmk}. This gives the commutative diagram (\ref{StackyCBDiagrams2eqn}) by commutativity of (\ref{BCGeometricQ1eqn}).
\end{proof}

We also have $\mathbb{G}_m$-equivariant versions of these statements, using the notions of \cite{Romagny} for group actions and quotients of stacks. The group stack $\mathcal I_{S,G}^{sm}$ descends further to a group stack $\mathcal I^{sm}_{G\times \mathbb{G}_m}$ on $[S/G\times \mathbb{G}_m]$ with the same properties. Moreover, the scalar action on $S$ induces strict $\mathbb{G}_m$-actions on $[S/G]$ and $\mathcal B$ making the diagram (\ref{StackyCBDiagrams1eqn}) $\mathbb{G}_m$-equivariant. The corollary below then follows immediately.

\begin{corollary}\label{GmEquivStackyCBcor}
    The induced map $\rho_{S,G\times\mathbb{G}_m}:[S/G \times \mathbb{G}_m] \rightarrow  \mathcal B/\mathbb{G}_m$ is a gerbe: for any scheme $X$ and any morphism $f:X \rightarrow [S/G\times \mathbb{G}_m]$, there is a Cartesian diagram
    \begin{equation}\label{GmEquivStackyCB1eqn}
        \begin{tikzcd}
            {\bf B} f^*\mathcal I_{S,G\times \mathbb{G}_m}^{sm}  \arrow[r] \arrow[d] & {[S/G\times \mathbb{G}_m]} \arrow[d,"\rho_{S,G\times \mathbb{G}_m}"] \\
            X \arrow[r, "\rho_{S,G \times \mathbb{G}_m} \circ f"]                  & {\mathcal B/\mathbb{G}_m}.                    
        \end{tikzcd}
    \end{equation}
\end{corollary}

To ease notation, we will drop the subscripts $G$ and $G \times \mathbb{G}_m$ where there is no possibility of confusion.

We now describe $\mathcal B$ explicitly as the quotient of a Katsylo slice $\mathfrak{K}$ for $S$ by the Katsylo group $F$ (see Definitions \ref{KatsyloSlicedef} and \ref{KatsyloGpdef}).

\begin{proposition}\label{StackyCBKQprp}
    There is an isomorphism $[\mathfrak{K}/F] \cong \mathcal B$ making the diagram
    \begin{equation}\label{StackyCBKQeqn}
\begin{tikzcd}
\mathfrak{K} \arrow[r, hook] \arrow[d] & S \arrow[d, "\rho_S"] \\
{[\mathfrak{K}/F]} \arrow[r, "\cong"]  & \mathcal B         
\end{tikzcd}
    \end{equation}
    commute.
    
    Under this isomorphism, the map $C:\mathcal B \rightarrow \mathfrak{c}_S$ is identified with the map $[\mathfrak{K}/F] \rightarrow \mathfrak{K}/F$ sending $[\mathfrak{K}/F]$ to its coarse moduli space.
\end{proposition}

 \begin{proof}
 We first observe that the embedding $\mathfrak{K} \hookrightarrow S$ induces an isomorphism of stacks $[\mathfrak{K}/\mathcal R] \cong [S/G]$, where $\mathcal R$ is the restricted action scheme of Definition \ref{RestrictedActiondef} with its groupoid structure (\ref{RestrictedAction2eqn}); this is because the map $\mathfrak{K} \rightarrow [S/G]$ is a smooth cover, and $\mathcal R$ is the restriction of the $G$-action groupoid to $\mathfrak{K}$.

Under this isomorphism, $\mathcal B$ is identified with the rigidification of $[\mathfrak{K}/\mathcal R]$ by a group stack which descends from $\mathcal I_\mathfrak{K}^{sm}$. Moreover, by Lemma \ref{RAQSGroupoidlem}, the map $\hat{\sigma}:\mathcal R \rightarrow F \times \mathfrak{K}$ induces a morphism $\rho_{\sigma}:[\mathfrak{K}/\mathcal R] \rightarrow [\mathfrak{K}/F]$, and by the construction of $\mathcal I_{\mathfrak{K}}^{sm}$, $\rho_\sigma$ induces the required isomorphism on the rigidification.  The remaining properties are clear. 
\end{proof}

We have the following immediate corollaries.

\begin{corollary}\label{StackyCBSmDMcor}
    The $S$-Chevalley base $\mathcal B$ is a smooth Deligne-Mumford stack. It is a scheme if and only if $F$ is trivial, and in this case it is the geometric quotient space $\mathfrak{c}_S$.
\end{corollary}

\begin{corollary}\label{GerbeTrivncor}
    Any choice of Katsylo slice defines an \'{e}tale trivialising cover $\mathfrak{K} \rightarrow \mathcal B$ for the gerbe $\rho_S:[S/G]\rightarrow \mathcal B$. In particular, if $F$ is trivial, the gerbe is trivial.
\end{corollary}

There are two possible ways to upgrade the theorem to a $\mathbb{G}_m$-equivariant version. First, let $(L, \mathcal O)$ be decomposition data for $S$ (see Remark \ref{Decompdatarmk}), and consider $\mathfrak{z} = \Lie(Z(L))$ with its $W_L$-action. We can identify $\mathfrak{K}$ with an affine space $\tilde{\mathfrak{c}}_S = \mathfrak{z}/W_S$, where a $W_S$ is a canonically-defined subgroup of $W_L$: the construction is given in Section \ref{KatsyloWeylsbn} below. We consider the $\mathbb{G}_m$-action on $\tilde{\mathfrak{c}}_S$ induced by the scaling action on $\mathfrak{z}$.

\begin{proposition}\label{GmEquivStackyCBKAprp}
    The isomorphism $[\tilde{\mathfrak{c}}_S/F] \cong \mathcal B$ is $\mathbb{G}_m$-equivariant, i.e. it induces an isomorphism $[\tilde{\mathfrak{c}}_S/F \times \mathbb{G}_m] \cong \mathcal B/\mathbb{G}_m$.
\end{proposition}

\begin{proof}
Since the locus on which $F$ acts freely is dense in $\tilde{\mathfrak{c}}_S$, the schematic locus is dense in $[\tilde{\mathfrak{c}}_S/F]$ (hence also in $\mathcal B)$. Since both stacks are normal separated Deligne-Mumford stacks, by \cite[Proposition A.1]{FMN} it suffices to note that the $\mathbb{G}_m$-actions agree on their schematic locus, or equivalently on the coarse moduli space $\mathfrak{c}_S$.
\end{proof}

Alternatively, we can consider the $\mathbb{G}_m$-action on $[\mathfrak{K}/F]$ induced by the Kazhdan action defined as in Definition \ref{KazhdanActiondef}. As in \cite[Proposition 2.5]{Ngo1} we denote by $\sq:\mathbb{G}_m^{[2]} \rightarrow \mathbb{G}_m$ the squaring homomorphism (for $\mathbb{G}_m^{[2]}$ a copy of $\mathbb{G}_m$), and for a stack $\mathcal X$ with a given $\mathbb{G}_m$-action, we denote the quotient of $\mathcal X$ by the square of this action as $\mathcal X/\mathbb{G}_m^{[2]}$ to distinguish it from the usual quotient $\mathcal X/\mathbb{G}_m$.

We take the $\mathbb{G}_m$-action on $\mathfrak{K}$ to be the Kazhdan action and the $\mathbb{G}_m$-action on $\mathcal B$ to be as above.
\begin{proposition}\label{StackyHKSectionprp}
   There is an isomorphism $[\mathfrak{K}/F\times \mathbb{G}_m] \cong \mathcal B/\mathbb{G}_m^{[2]}$ and a commutative diagram
   \begin{equation}\label{StackyHKSection1eqn}
        \begin{tikzcd}
            {[\mathfrak{K}/\mathbb{G}_m]} \arrow[r] \arrow[d] & {[S/G \times \mathbb{G}_m^{[2]} ]} \arrow[d] \\
            {[\mathfrak{K}/F \times \mathbb{G}_m]} \arrow[r, "\cong"]  & {\mathcal B/\mathbb{G}_m^{[2]}}         
        \end{tikzcd}
    \end{equation}
    compatible with the diagram (\ref{StackyCBKQeqn}).
\end{proposition}

\begin{proof} 
To prove the proposition, we must show that the horizontal arrows in the diagram
\begin{equation}\label{StackyHKSection2eqn}
    \begin{tikzcd}
        \mathfrak{K} \arrow[r, hook] \arrow[d] & {[S/G]} \arrow[d, "\rho_S"] \\
        {[\mathfrak{K}/F]} \arrow[r, "\cong"]  & \mathcal B         
    \end{tikzcd}
\end{equation}
are equivariant with respect to the appropriate $\mathbb{G}_m$-actions. The equivariance of the top arrow follows from the definition of the Kazhdan action. The equivariance of the isomorphism follows from Proposition \ref{GmEquivStackyCBKAprp}.
\end{proof}

\section{Abelianisation and the cameral group}\label{Abelianisationscn}
In this section, we construct a canonical homomorphism from the centraliser on a Dixmier sheet $S$ to a ``cameral group" as in \cite{DG} and \cite{Ngo2}. This homomorphism realises the cameral group as an abelianisation of the centraliser on $S$. Under certain conditions (in particular whenever $G$ is a classical group), we use this to factorise the gerbe $\rho_S:[S/G] \rightarrow \mathcal B$ through an abelian gerbe. As a preliminary, we outline the interaction between the $S$-Chevalley base and the Grothendieck-Springer theory of sheets (we recall the relevant background for this in Section \ref{GSSheetssbn}).

\subsection{The \texorpdfstring{$S$}{S}-Chevalley base and Grothendieck-Springer theory}\label{SCBGSsbn}
In the regular case, the Grothendieck-Springer resolution plays an important role in the cameral descriptions of the centraliser in \cite{DG} and \cite{Ngo2}. In order to adapt this description to our setting, we incorporate the $S$-Chevalley base into the Grothendieck-Springer theory for sheets (see Section \ref{GSSheetssbn}). In fact, doing this resolves some of the complications of the theory which are outlined in \cite{Broer1}; in particular, replacing the geometric quotient $\mathfrak{c}_S$ with the $S$-Chevalley base $\mathcal B$ in the Grothendieck-Springer diagram (\ref{GSDiagrameqn}) makes the diagram Cartesian.

We fix a non-singular sheet $S$ together with decomposition data $(L,\mathcal O)$ as in Remark \ref{Decompdatarmk}. First, we must replace the quotient map $\mathfrak{z} \rightarrow \mathfrak{z}/W_L = \mathfrak{c}_S$ with a map to the $S$-Chevalley base $\mathcal B$ constructed in Definition \ref{StackyChevalleyBasedef}.

\begin{lemma}\label{StackyZquotientlem}
    There is a unique morphism $p:\mathfrak{z} \rightarrow \mathcal B$ inducing a factorisation
    \begin{equation}\label{StackyZquotient1eqn}
        \begin{tikzcd}
            \mathfrak{z} \arrow[r, "p"] & \mathcal B \arrow[r, "C"] & \mathfrak{c}_S
        \end{tikzcd}
    \end{equation}
    of the quotient map $\mathfrak{z} \rightarrow \mathfrak{c}_S$, where $C$ is defined as in Proposition \ref{StackyCBDiagramsprp}. The morphism $p$ is finite, flat, representable by schemes, and $\mathbb{G}_m$-equivariant.
\end{lemma}

\begin{proof} 
We identify $\mathcal B$ with $[\tilde{\mathfrak{c}}_S/F]$ as in Proposition \ref{GmEquivStackyCBKAprp}, where $\tilde{\mathfrak{c}}_S = \mathfrak{z}/W_S$ for the subgroup $W_S \leq W_L$ defined in Corollary \ref{KatsyloGpWeylGpcor}. Then we can construct a morphism $p$ satisfying the factorisation (\ref{StackyZquotient1eqn}) by
\begin{equation}\label{StackyZquotient2eqn}
        \begin{tikzcd}
            p:\mathfrak{z} \arrow[r] & \mathfrak{z}/W_S = \tilde{\mathfrak{c}}_S \arrow[r] & {[\tilde{\mathfrak{c}}_S/F]},
        \end{tikzcd}
    \end{equation}
where each of the constituent arrows is the relevant quotient map.

To see that this is the only map inducing a factorisation (\ref{StackyZquotient1eqn}), we observe that any other morphism $p':\mathfrak{z} \rightarrow \mathcal B$ which also induces such a factorisation agrees with $p$ over the schematic locus of $\mathcal B$. So, $p$ and $p'$ are the same morphism by \cite[Proposition A.1]{FMN}.

The morphism $p$ is finite, flat and representable since both of the constituent arrows in (\ref{StackyZquotient2eqn}) satisfy these properties. It is $\mathbb{G}_m$-equivariant by Proposition \ref{GmEquivStackyCBKAprp}.
\end{proof}

We will need the following lemma in the next section.

\begin{lemma}\label{WLInvariancelem}
    For any scheme $X$ and map $X \rightarrow \mathcal B$, the $W_L$-action on $\mathfrak{z}$ lifts to a $W_L$-action on $X \times_{\mathcal B} \mathfrak{z}$.
\end{lemma}

\begin{proof} 
For any $w \in W_L$, the map $p \circ w:\mathfrak{z} \rightarrow \mathcal B$ satisfies the same factorisation (\ref{StackyZquotient1eqn}) as $p$; so by Lemma \ref{StackyZquotientlem}, $p \circ w$ and $p$ are the same map, and so the map $w:\mathfrak{z} \rightarrow \mathfrak{z}$ lifts to a map $w:X \times_{\mathcal B} \mathfrak{z} \rightarrow X \times_{\mathcal B} \mathfrak{z}$. This defines the required $W_L$-action on $X \times_{\mathcal B} \mathfrak{z}$.
\end{proof}

We now incorporate the $S$-Chevalley base into the Grothendieck-Springer theory for the sheet; we use the notation of Section \ref{GSSheetssbn}. We consider the variety $\hat{S}^{reg}$ of (\ref{GeneralisedGSRegeqn}); let $\hat{p}:\hat{S}^{reg} \rightarrow S$ be the generalised Grothendieck-Springer map, and let $\hat{\chi}_S:\hat{S}^{reg} \rightarrow \mathfrak{z}$ be defined as in the Grothendieck-Springer diagram (\ref{GSDiagrameqn}). In the example of most importance for us, when $S$ is a Dixmier sheet associated to a Levi subgroup $L$ of $G$, $\hat{S}^{reg} = G \times^P \mathfrak{r}^{reg}$ where $P$ is a parabolic subgroup of $G$ with Levi factor $L$ and $\mathfrak{r}$ is the solvable radical of $\mathfrak{p}$ (see Remark \ref{DixmierGSrmk}). In that case, $\hat{p}$ is the $G$-action map and $\hat{\chi}_S$ is induced by the projection $\mathfrak{r} \rightarrow \mathfrak{z}$.

\begin{proposition}\label{StackyGSprp}
    The diagram
    \begin{equation}\label{StackyGSeqn}
        \begin{tikzcd}
         \hat{S}^{reg} \arrow[r, "\hat{\chi}_S"] \arrow[d, "\hat{p}"] & \mathfrak{z} \arrow[d, "p"] \\
            S \arrow[r, "\rho_S"]                                       & \mathcal B              
        \end{tikzcd}
    \end{equation}
    is commutative and Cartesian, where $\rho_S:S \rightarrow \mathcal B$ is the $S$-Chevalley map defined in (\ref{StackyChevalleyMapeqn}).
\end{proposition}

\begin{proof}
As in the proof of Lemma \ref{StackyZquotientlem}, both maps $p \circ \hat{\chi}_S$ and $\rho_S \circ \hat{p}$ agree on the schematic locus of $\mathcal B$ (by the commutativity of the Grothendieck-Springer diagram (\ref{GSDiagrameqn})), so define the same morphism, i.e. the diagram commutes.

The fibre product of stacks $S \times_{\mathcal B} \mathfrak{z}$ is a scheme, since $p$ is representable. Moreover, the map $S \times_{\mathcal B} \mathfrak{z} \rightarrow \mathfrak{z}$ is the pullback of the smooth map $\rho_S$, so is smooth; hence $S \times_{\mathcal B} \mathfrak{z}$ is non-singular.

Consider the induced map $\mu:\hat{S}^{reg} \rightarrow S \times_\mathcal B \mathfrak{z}$; over the schematic locus of $\mathcal B$, the map $\mu$ agrees with the normalisation $\nu:\hat{S}^{reg} \rightarrow S \times_{\mathfrak{c}_S} \mathfrak{z}$ (see Proposition \ref{GSCartesianprp}). It also factors through the map $\hat{p}$, which is finite by Lemma \ref{GSDiagramlem}. Hence, $\mu$ is finite and birational, so by Zariski's main theorem it is an isomorphism.
\end{proof}

\begin{remark}\label{StackyGSrmk}
    Using $\mathcal B$ as a base instead of $\mathfrak{c}_S$ corrects the failure of the Grothendieck-Springer diagram (\ref{GSDiagrameqn}) to be Cartesian. It further ensures that the arrows in (\ref{StackyGSeqn}) have suitably nice properties, i.e. $\rho_S$ is smooth and $p$ is flat; the corresponding arrows in (\ref{GSDiagrameqn}) may fail to have these properties. See \cite{Broer1} for further details.
\end{remark}

\subsection{The cameral group for non-singular Dixmier sheets}\label{Camgpsbn}
For the rest of this section we will assume that $S$ is a non-singular Dixmier sheet corresponding to a fixed Levi subgroup $L$ of $G$, i.e. $S$ has an open dense locus of semisimple elements, whose centralisers are conjugate to $L$. 

The structure of the usual Hitchin fibration depends not only on the description of the regular adjoint quotient stack as a gerbe, but also on a ``galoisian description" of the regular centraliser \cite{DG}. To motivate the results of this section, we recall the role this description plays in the abelianisation for regular Higgs bundles.

If $S = \mathfrak{g}^{reg}$, then the centraliser group scheme $\mathcal I^{reg}$ is smooth, and moreover descends to a smooth group scheme $\mathcal I_{\mathfrak{c}}$ on the Chevalley base $\mathfrak{c} = \mathfrak{t}/W$, by \cite[Proposition 3.5]{Ngo1}. Moreover, there is a more explicit description of $\mathcal I_{\mathfrak{c}}$ as follows. Let $T \times \mathfrak{t}$ be the constant group scheme on $\mathfrak{t}$; then the pushforward of this group scheme under the quotient map $\pi:\mathfrak{t} \rightarrow \mathfrak{c}$ defines a group scheme $\Pi$, the \emph{Weil restriction} of $T \times \mathfrak{t}$. This inherits a $W$-action, and there is an open embedding 
\begin{equation}\label{RegCameralMoreqn}
    \kappa:\mathcal I_{\mathfrak{c}} \hookrightarrow \hat{\mathcal J}
\end{equation}
which is an isomorphism over an open subset of $\mathfrak{c}$, where $\hat{\mathcal J}$ is the $W$-fixed point group scheme of $\Pi$, by \cite[Proposition 2.4.2]{Ngo2}. The image of $\mathcal I_{\mathfrak{c}}$ can be explicitly described using the values of roots at ramification points of the map $\pi$, by \cite[Proposition 2.4.7]{Ngo2}. 

A version of the embedding $\kappa$ is crucial in the construction of cameral data for Hitchin fibres in \cite{DG}. We have already seen in the previous sections that there are a number of complications in trying to generalise $\kappa$ directly for $S$, but a reasonable first attempt would be to consider if $\mathcal I_S$ can be related to the Weil restriction $\Pi_L= p_*(L \times \mathfrak{z})$ under the map $p:\mathfrak{z} \rightarrow \mathcal B$. However, even for $G = \GL_n$, this may not give the desired generalisation, e.g., for the subregular sheet in $\fgl_3$ (this can be compared directly with \cite[Example 4.2]{HM}). 

Despite this, we can give a generalisation of the homomorphism $\kappa$ in (\ref{RegCameralMoreqn}) as an abelianisation of the smooth centraliser $\mathcal I^{sm}_S$ of $S$, and under certain conditions, which are always satisfied if $G$ is a classical group, we prove a smoothness property necessary for constructing the abelianised fibrations in Section \ref{IntAbFibrnsbn}.

Let $\bar{Z}$ be the abelianisation of $L$, i.e. $\bar{Z} = L/L^{der}$, where
$$L^{der} = \langle aba^{-1}b^{-1} \,|\, a,b \in L \rangle.$$

We let
\begin{equation}\label{WeilRestrictioneqn}
    \Pi_S = p_*(\bar{Z} \times \mathfrak{z})
\end{equation}
be the Weil restriction of the constant group scheme $\bar{Z} \times \mathfrak{z}$ under the map $p:\mathfrak{z} \rightarrow \mathcal B$ of Lemma \ref{StackyZquotientlem} (see e.g. \cite{HR} for the definition and properties of Weil restriction in the context of algebraic stacks). The set of $X$-points of $\Pi_S$ over $\mathcal B$, for a scheme $X$ with a given map $X \rightarrow \mathcal B$, can be described as
\begin{equation}
    (\Pi_S)_{\mathcal B}(X) = \Maps(X \times_{\mathcal B} \mathfrak{z}, \bar{Z}). 
\end{equation}

\begin{lemma}\label{WeilRestrictionlem}
    The group stack $\Pi_S$ is smooth and representable by schemes over $\mathcal B$.
\end{lemma}

\begin{proof} 
We recall from Lemma \ref{StackyZquotientlem} that $p$ is finite, flat and representable. To deduce that $\Pi_S \rightarrow \mathcal B$ is representable, it suffices to observe that Weil restriction commutes with pullback, and the Weil restriction of a scheme under a finite flat morphism of schemes is representable by \cite[Section 7.6, Theorem 4]{BLR}. Similarly, smoothness follows by \cite[Section 7.6, Proposition 5]{BLR}.
\end{proof}

\begin{remark}\label{WeilRestrictionrmk}
  Since the morphism $p$ is $\mathbb{G}_m$-equivariant, the $\mathbb{G}_m$-action on $\bar{Z}\times \mathfrak{z}$ given by the scalar action on the second factor determines a strict $\mathbb{G}_m$-action on $\Pi_S$ which lifts the action on $\mathcal B$. 
\end{remark}

The normaliser $N_G(L)$ of $L$ in $G$ acts by conjugation on $L$, and this induces an action of $W_L = N_G(L)/L$ on $\bar{Z}$. By Lemma \ref{WLInvariancelem}, we can define a strict $W_L$-action on $\Pi_S$ over $\mathcal B$ induced by the diagonal action on $\bar{Z} \times \mathfrak{z}$.

\begin{definition}\label{PseudoCameralgpdef}
    We call the fixed point subgroup stack of $\Pi_S$ under $W_L$ the \emph{pseudo-cameral group} and denote it by $\hat{\mathcal J}_S$.
\end{definition}

The set of $X$-points of $\hat{\mathcal J}_S$ over $\mathcal B$ is the set of $W_L$-equivariant maps
\begin{equation}\label{PseudoCameralgpeqn}
    (\hat{\mathcal{J}}_S)_{\mathcal B}(X) = \Maps^{W_L}(X \times_{\mathcal B} \mathfrak{z}, \bar{Z}).
\end{equation}

\begin{remark}\label{PseudoCameralgprmk}
    The pseudo-cameral group $\hat{\mathcal J}_S$ is a smooth closed subgroup stack of $\Pi_S$, representable by schemes over $\mathcal B$. The pseudo-cameral group is canonical in the following sense. Suppose $L'$ is a different choice of Levi subgroup of $G$ corresponding to the Dixmier sheet $S$, i.e. $L'$ is a $G$-conjugate of $L$. Let $\hat{\mathcal J}_S'$ be the group defined by replacing $L$ with $L'$ in the construction of $\hmJ_S$. Then there is a canonical isomorphism $\hat{\mathcal J}_S \cong \hat{\mathcal J}_S'$.

    Since the actions of $\mathbb{G}_m$ and $W_L$ on $\mathfrak{z}$ commute, the $\mathbb{G}_m$-action on $\Pi_S$ restricts to an action on $\hat{\mathcal J}_S$. In particular, $\hat{\mathcal J}_S$ descends to a group stack $\hat{\mathcal J}_{S, \mathbb{G}_m}$ on $\mathcal B/\mathbb{G}_m$.
\end{remark}

 We recall the $S$-Chevalley map $\rho_S:S \rightarrow \mathcal B$ defined in (\ref{StackyChevalleyMapeqn}). Note that the $S$-group scheme $\rho_S^*\Pi_S$ has a natural $G$-action (which restricts to an action on $\rho_S^*\hat{\mathcal J}_S$), since $\rho_S$ factors through the quotient $S \rightarrow [S/G]$. The following construction is the desired generalisation of (\ref{RegCameralMoreqn}).

\begin{proposition}\label{CameralHomprp}
There is a homomorphism of group schemes $\kappa_S:\mathcal I^{sm}_S \rightarrow \mathcal \rho_S^*\hat{\mathcal J}_S$, equivariant with respect to the actions of $G$ and $\mathbb{G}_m$, such that for any semisimple element $x \in S$, $\kappa_{S,x}: \mathcal I^{sm}_{S,x} \rightarrow \hat{\mathcal J}_{S, \rho_S(x)}$ realises the abelianisation of the group $\mathcal I^{sm}_{S,x}$.
\end{proposition}

\begin{proof} 
Let $P$ be a parabolic subgroup of $G$ with Levi factor $L$, and let $\mathfrak{r}$ be the solvable radical of $\mathfrak{p} = \Lie(P)$. Let $\hat{p}:G \times^P\mathfrak{r}^{reg} \rightarrow S$ be the generalised Grothendieck-Springer map defined in (\ref{GSDiagrameqn}); by Proposition \ref{StackyGSprp}, we have that $\rho_S^* \Pi_S = \hat{p}_*(\bar{Z} \times (G \times^P \mathfrak{r}^{reg}))$. To construct a homomorphism $\kappa_S:\mathcal I^{sm}_S \rightarrow \mathcal \rho_S^*\Pi_S$, it is equivalent by adjunction to construct a homomorphism
\begin{equation}\label{CameralHomLifteqn}
    \hat{\kappa}_S:\hat{p}^*\mathcal I^{sm}_S \rightarrow \bar{Z} \times (G \times^P \mathfrak{r}^{reg}).
\end{equation}

There is a group scheme $\mathcal P$ over $G \times^P \mathfrak{r}^{reg}$ pulled back from the universal parabolic subgroup scheme of the constant group $G$ over $G/P$; in particular, for a $\mathbb{C}$-point $P(g,x)$ of $G \times^P \mathfrak{r}^{reg}$, the fibre of $\mathcal P$ over this point is $I_g(P)$. By Proposition \ref{ParabolicCentraliserprp}, there is an inclusion $\mathcal I_S^{sm} \hookrightarrow \mathcal P$. Moreover, there is a homomorphism from $\mathcal P$ to the constant group $\bar{Z}$ over $G \times^P \mathfrak{r}^{reg}$, which over the $\mathbb{C}$-point $P(g,x)$ is given by the composition
\begin{equation}\label{CameralHom1eqn}
\begin{tikzcd}
gPg^{-1} \arrow[r, "I_{g^{-1}}"] & P \arrow[r] & P/P^{der} = \bar{Z}. 
\end{tikzcd}
\end{equation}
To see that this is well-defined, we observe that if $gP = g'P$, then $I_g$ and $I_{g'}$ differ by conjugation by an element of $P$, and so define the same map after projection to $\bar{Z}$. Hence, we have a well-defined homomorphism $\hat{\kappa}_S$, and thus a well-defined homomorphism $\kappa_S:\mathcal I^{sm}_S \rightarrow \rho_S^*\Pi_S$. It is clear from the construction that $\kappa_S$ is equivariant with respect to the actions of $G$ and $\mathbb{G}_m$.

Since $\mathcal I_S^{sm}$ is smooth, to prove that $\kappa_S$ takes its image in the closed subgroup $\rho_S^* \hat{\mathcal J}_S$ of $\rho_S^*\Pi_S$, it suffices to do so over the dense open subset of semisimple elements $S^{ss} \subseteq S$. If we let $\mathfrak{z}^{rs}$ be the open subset of $\mathfrak{z}$ defined by
\begin{equation}\label{WeylLeviActioneqn}
    \mathfrak{z}^{rs} = \{ x \in \mathfrak{g} \, | \,C_G(x) = L \},
\end{equation}
then $S^{ss} = \Ad(G)(\mathfrak{z}^{rs})$. By $G$-equivariance, $\kappa_S|_{S^{ss}}$ is determined by its value on $\mathfrak{z}^{rs}$, so in particular, $\kappa_S|_{S^{ss}}$ takes its image in $\rho_S^* \hat{\mathcal J}_S|_{S^{ss}}$ if and only if $\kappa_S|_{\mathfrak{z}^{rs}}$ takes its image in $\rho_S^* \hat{\mathcal J}_S|_{\mathfrak{z}^{rs}}$. 

There is a Cartesian diagram
\begin{equation}\label{CameralHom2eqn}
\begin{tikzcd}
W_L \times \mathfrak{z}^{rs} \arrow[d, "\pi_2"] \arrow[r, "\hat{\iota}"] & G \times^P \mathfrak{r}^{reg} \arrow[d, "\hat{p}"] \\
\mathfrak{z}^{rs} \arrow[r, hook]                                 & S                                                   
\end{tikzcd}
\end{equation}
where $\hat{\iota}$ is defined on $\mathbb{C}$-points by
\begin{equation}\label{CameralHom3eqn}
    (nL,x) \mapsto P(n^{-1},Ad_n(x)), 
\end{equation}
for $nL \in N_G(L)/L = W_L$ and $x \in \mathfrak{z}^{rs}$. We can identify the pullback of $\hat{\kappa}_S$ under $\hat{\iota}$ with the map
\begin{equation}\label{CameralHom4eqn}
    L \times (W_L \times \mathfrak{z}^{rs}) \rightarrow \bar{Z} \times (W_L \times \mathfrak{z}^{rs}) 
\end{equation}
which over $\{w\} \times \mathfrak{z}^{rs}$ is given by
the constant group homomorphism
\begin{equation}\label{CameralHom5eqn}
\begin{tikzcd}
L \arrow[r] & L/L^{der} = \bar{Z} \arrow[r, "w"] & \bar{Z}.
\end{tikzcd}  
\end{equation}
This is $W_L$-equivariant, where the $W_L$-action on $L \times (W_L \times \mathfrak{z}^{rs}) \cong \mathcal I_S^{sm} \times_{\mathcal B} \mathfrak{z}$ is given by Lemma \ref{WLInvariancelem}; hence $\kappa_S|_{\mathfrak{z}^{rs}}$ takes its image in $\rho_S^* \hat{\mathcal J}_{\mathfrak{z}^{rs}}$. This also gives the required characterisation of $\kappa_{S,x}$ for $x \in S^{ss}$.
\end{proof}

\begin{remark}\label{CameralHomrmk}
    We call $\kappa_S$ the \emph{cameral homomorphism} for $S$ (with respect to $L$); it is independent of the choice of parabolic $P$. If $\phi:G \rightarrow G'$ is an isomorphism of groups, then the isomorphism identifies the cameral homomorphism $\kappa_S$ for $S$ with respect to a Levi subgroup $L \leq G$ with the cameral homomorphism $\kappa_S'$ for $\phi(S)$ with respect to the Levi subgroup $\phi(L) \leq G'$. In particular, if $L$ and $L'$ are conjugate Levi subgroups in $G$, the cameral homomorphisms for $S$ with respect to $L$ and $L'$ are canonically identified.
\end{remark} 

We now restrict to a particular class of examples of sheets for which we prove that the cameral homomorphism is smooth.

\begin{definition}\label{ClassRamTypedef}
    We say that a Dixmier sheet $S$ associated to a Levi subgroup $L$ is of \emph{classical reduction type (CRT)} if, for any Levi subgroup $M$ of $G$ containing $L$ minimally  (i.e. there is no Levi subgroup $M'$ of $G$ with $L \subsetneq M' \subsetneq M$), $M$ is a classical group.
\end{definition}

\begin{remark}\label{ClassRamTypermk}
    If $\mathfrak{g}$ is classical, then every Dixmier sheet in $\mathfrak{g}$ is automatically of classical reduction type. The definition also covers a large class of examples in the exceptional Lie algebras; the CRT condition is rather artificial in this case, but is useful for our purposes since all but one of the examples which we require in Section \ref{RealHitchinscn} are CRT. It is easy to check when a given Dixmier sheet $S$ is CRT using the Dynkin subdiagram of the corresponding Levi subgroup.
\end{remark}

In this case, we have the following key property.

\begin{proposition}\label{CameralSmoothprp}
    If $S$ is a non-singular Dixmier sheet of classical reduction type, the cameral homomorphism $\kappa_S:\mathcal I_S^{sm} \rightarrow \rho_S^*\hat{\mathcal J}_S$ is smooth.
\end{proposition}

Our method of proof is the same as the proof of \cite[Proposition 2.4.7]{Ngo2}, and involves reducing to checking the statement on a finite class of examples by passing to Levi subgroups of $G$. The proof in these examples involves a case-by-case check on the interaction between $\mathfrak{sl}_2$-triples and the abelianisation map for the Levi subalgebra. We have provided the full details in Appendix \ref{DixmierClassicalscn}, and give only a sketch of the argument in the proof of Proposition \ref{CameralSmoothprp} below.  

First, we give some preliminaries on Levi reduction for the Grothendieck-Springer theory of sheets. For the following lemmas we allow $S$ to be a non-singular Dixmier sheet in $\mathfrak{g}$ associated with a Levi subgroup $L \leq G$. Let $x \in S$ and let $x = x_{ss}+x_n$ be its Jordan decomposition. By \cite[Satz 4.8]{Borho2}, there exists $g \in G$ such that the Levi subgroup $M: = C_G(x_{ss})$ contains $I_g(L)$ and $x_n$ is a representative for $\mathcal O_{\mathfrak{m}}$, where
\begin{equation}\label{InducedOrbiteqn}
        \mathcal O_{\mathfrak{m}} = \text{Ind}_{\mathfrak{l}}^{\mathfrak{m}}(0)
    \end{equation}
is the \emph{orbit induced from $0$} in the sense of \cite[Theorem 1.3]{LS}. For simplicity, we will suppose that $g =\text{id}_G$; in particular, $x_{ss} \in \mathfrak{z}$ since $\mathfrak{l} \subseteq \mathfrak{c}_{\mathfrak{g}}(x_{ss})$. Let $S_M$ be the Dixmier sheet in $\mathfrak{m}$ corresponding to $L$ as a Levi subgroup of $M$. Note that this depends on the specific subgroup $L \leq G$, as the $G$-conjugacy class of $L$ may split into distinct $M$-conjugacy classes. 

\begin{lemma}\label{SheetsLIlem}
    There exist parabolic subgroups $P_M$ of $M$ and $P$ of $G$, both with Levi factor $L$, such that $P_M = P \cap M$ and $x \in \mathfrak{r}_M^{reg}$, where $\mathfrak{r}_M$ is the solvable radical of $\mathfrak{p}_M = \Lie(P_M)$. In particular, $x \in S_M$.
\end{lemma}

\begin{proof}
By the construction of the induced orbit in \cite{LS}, there is a parabolic subgroup $P_M$ of $M$ with Levi factor $L$ such that $x_n$ is a representative for the dense $P_M$-orbit in the nilradical $\mathfrak{n}_M$ of $\mathfrak{p}_M$. Then, $x \in \mathfrak{z} \oplus \mathfrak{n}_M = \mathfrak{r}_M$; moreover, the $P_M$-orbits of $x$ and $x_n$ have the same dimension since $x_{ss}$ is central in $\mathfrak{m}$, so $x \in \mathfrak{r}_M^{reg}$.

To define $P$, let $P'$ be a parabolic subgroup of $G$ with Levi factor $M$, and let $U'$ be its unipotent radical. Then $P = P_MU'$ is a parabolic subgroup of $G$ with the required properties.
\end{proof}

We now assume that $S_M$ is non-singular, and with the choices of $P$ and $P_M$ of Lemma \ref{SheetsLIlem}, we consider the generalised Grothendieck-Springer morphisms $\hat{p}:G \times^P \mathfrak{r} \rightarrow \overline{S}$ and $\hat{p}_M:M \times^{P_M} \mathfrak{r}_M \rightarrow \overline{S}_M$ defined in (\ref{GeneralisedGSeqn}).

\begin{lemma}\label{GSLevilem}
     There are closed embeddings $\overline{S}_M \hookrightarrow\overline{S}$ and $\iota_M:M \times^{P_M} \mathfrak{r}_M \hookrightarrow G \times^P \mathfrak{r}$ making the diagram
    \begin{equation}\label{GSLevi1eqn}
\begin{tikzcd}
 M \times^{P_M} \mathfrak{r}_M \arrow[r, "\iota_M", hook] \arrow[d, "\hat{p}_M"] & G \times^P \mathfrak{r} \arrow[d, "\hat{p}"] \\
\overline{S}_M \arrow[r, hook]                                                                                 & \overline{S},                                  
\end{tikzcd}
    \end{equation}
    commute.
\end{lemma}

\begin{proof} 
There is an inclusion $\overline{S}_M \subseteq \overline{S}$ since $\overline{S}_M = M\mathfrak{r}_M$ and $\overline{S} = G\mathfrak{r}$ by \cite[Theorem 5.4]{BK}, which defines a closed embedding. 

Let $P'$ be as in the proof of the previous lemma. Since $P' = MP$ and $P_M = M \cap P$, there is a canonical isomorphism $M \times^{P_M} \mathfrak{r} \cong P' \times^P \mathfrak{r}$. Thus the obvious choice for $\iota_M$ decomposes into closed embeddings as
\begin{equation}\label{GSLevi2eqn}
\begin{tikzcd}
M \times^{P_M} \mathfrak{r}_M \arrow[r, hook] & M \times^{P_M} \mathfrak{r} \cong P' \times^P \mathfrak{r} \arrow[r, hook] & G \times^P \mathfrak{r}.
\end{tikzcd}
\end{equation} 
\end{proof}

We denote $W_L^M=N_M(L)/L$. We recall that there is a $W_L$-action on $G \times^P\mathfrak{r}^{reg}$ lifting the action on $\mathfrak{z}$, and similarly a $W_L^M$-action on $M \times^{P_M} \mathfrak{r}_M^{reg}$. The following lemma should be compared with \cite[Proposition 10.6]{DG}.

\begin{lemma}\label{GSLeviCartesianlem}
There are open subsets $\mathfrak{r}_M^* \subseteq \mathfrak{r}_M$ and $S^*_M \subseteq S_M$, with $x \in \mathfrak{r}_M^* \subseteq S_M^*$, such that there is a Cartesian diagram
\begin{equation}\label{GSLeviCartesian1eqn}
\begin{tikzcd}
W_L \times^{W^M_L} (M \times^{P_M} \mathfrak{r}_M^*) \arrow[r, "\hat{\iota}_M", hook] \arrow[d, "\hat{p}_M \circ \pi_2"] & G \times^P \mathfrak{r}^{reg} \arrow[d, "\hat{p}"] \\
S^*_M \arrow[r, hook]                                                                                 & S.                                  
\end{tikzcd}
    \end{equation}
where $\hat{\iota}_M$ is a $W_L$-equivariant extension of the embedding $\iota_M$ (defined in Lemma \ref{GSLevilem}) over $M \times^{P_M} \mathfrak{r}_M^*$.
\end{lemma}

\begin{proof} 
Let $\mathfrak{z}^* \subseteq \mathfrak{z}$ be the locus defined by
\begin{equation}\label{GSLeviCartesian2eqn}
    \mathfrak{z}^* = \{ x \in \mathfrak{z} \, | \,C_G(x) \leq M \}.
\end{equation}
This is an open set (e.g. since it can be defined by root non-vanishing conditions), and it is stable under the action of $W_L^M$. Hence, if we define $S_M^* = \chi_{S_M}^{-1}(\mathfrak{z}^*/W_L^M)$, where $\chi_{S_M}$ is the geometric quotient map of Theorem \ref{BCGeometricQthm}, then $S_M^* \subseteq S_M$ and $\mathfrak{r}_M^* = S_M^* \cap \mathfrak{r}_M \subseteq \mathfrak{r}_M$ are open inclusions, and $x \in \mathfrak{r}_M^*$ by assumption. Moreover, it is clear that $\hat{p}_M$ restricts to $\hat{p}_M:M \times^{P_M} \mathfrak{r}^* \rightarrow S_M^*$ over $S_M^*$.

For any $y \in S_M^*$, its semisimple part $y_{ss}$ is conjugate under $M$ to an element of $\mathfrak{z}^*$ (by the construction of $\chi_{S_M}$ in \cite[Satz 5.6]{Borho2}). So $C_G(y) \leq C_G(y_{ss}) \leq M$, i.e. $C_G(y) = C_M(y)$, and $y \in S$ (e.g. since $S$ is the locus in $\overline{S}$ of fixed $G$-centraliser dimension equal to $\text{dim}(L) = \text{dim}(C_M(y))$). Hence also, $\iota_M$ maps $M \times^{P_M} \mathfrak{r}_M^*$ into $G \times^P \mathfrak{r}^{reg}$.

Over the dense open subset $S^{ss}_M \subseteq S^*_M$ of semisimple elements, the restriction of $\iota_M$ is $W_L^M$-equivariant; since every $P$-orbit in $\mathfrak{r}^{rs} = S^{ss} \cap \mathfrak{r}$, and similarly every $P_M$-orbit in $\mathfrak{r}^{rs}_M = S^{ss}_M \cap \mathfrak{r}$, is the orbit of a unique point in $\mathfrak{z}^{rs}$ (defined in (\ref{WeylLeviActioneqn})), and the $W_L^M$-actions agree on the $\mathbb{C}$-points of $M \times^{P_M} \mathfrak{r}_M^{reg}$ and $G \times^P\mathfrak{r}^{reg}$ represented by these orbits. Hence by continuity $\iota_M$ is $W_L^M$-equivariant on $M \times^{P_M} \mathfrak{r}_M^*$, so defines the $W_L$-equivariant extension $\hat{\iota}_M$ making the diagram (\ref{GSLeviCartesian1eqn}) commute. The left hand arrow in (\ref{GSLeviCartesian1eqn}) makes sense since $\hat{p}_M$ is $W_L^M$-invariant.

We show that $\hat{\iota}_M$ is a locally closed embedding, by showing that the images of each of the components $\{n\} \times^{W_L^M} (M \times^{P_M} \mathfrak{r}_M^*)$ are disjoint (where each $n$ is a representative of a coset $nW_L^M \in W_L/W_L^M$). By construction of $\iota_M$, it suffices to show that if $n \in N_G(L) \cap P'$, then $n \in M$. We have $N_G(L) \cap P' \leq LN_G(T) \cap P'$, and if $ln \in P'$ for some $l \in L$ and $n \in N_G(T)$, then $n \in P'$. But $ N_G(T) \cap P' = N_M(T)$ (e.g. using the Bruhat decomposition with respect to a suitable set of simple roots), so $LN_G(T) \cap P' \leq LN_M(T) \leq M$.

Then the diagram (\ref{GSLeviCartesian1eqn}) is Cartesian, since both the maps $\hat{p}_M \circ \pi_2$ and $\hat{p}$ are finite and flat, and have the same degree.
\end{proof}

As a result, we have the following compatibility of cameral homomorphisms.

\begin{lemma}\label{CameralLIlem}
    Over $S_M^* \subseteq S_M,$ there are isomorphisms $$\mathcal I_{S_M^*}^{sm} \cong (\mathcal I_{S}^{sm})|_{S_M^*} \text{ and } (\rho_{S_M}^*\hat{\mathcal J}_{S_M})|_{S_M^*} \cong (\rho_S^*\hat{\mathcal J}_S)|_{S_M^*},$$ identifying the cameral homomorphisms $\kappa_S$ and $\kappa_{S_M}$. In particular, at $x$, there is an identification between the cameral homomorphisms $\kappa_{S_M,x}$ and $\kappa_{S,x}$.
\end{lemma}

\begin{proof}
As noted in the proof of the previous lemma, for every $y \in S_M^*$, $C_G(y) \leq M$; and so if $y = \Ad_m(z)$ for $z \in \mathfrak{r}_M^*$, by Proposition \ref{ParabolicCentraliserprp} $$(\mathcal I_{S_M}^{sm})_y = I_m(C_{P_M}(z)) = I_m(C_P(z) \cap M) = I_m(C_P(z)) =(\mathcal I_S^{sm})_y$$
as subgroups of $M$. Thus $\mathcal I_{S_M^*}^{sm}$ and $ (\mathcal I_{S}^{sm})|_{S_M^*}$ coincide as subschemes of the constant group $M$ over $S^*_M$.

By Lemma \ref{GSLeviCartesianlem}, there is a natural identification between $W_L$-invariant maps from $S_M^* \times_S(G \times^P \mathfrak{r}^{reg})$ to $\bar{Z}$ and $W_L^M$-invariant maps from $M \times^{P_M} \mathfrak{r}_M^*$ to $\bar{Z}$; more specifically, the identification is given by restriction to the subscheme $M \times^{P_M} \mathfrak{r}_M^*$. This induces the second isomorphism.

It is straightforward to see that $\hat{\kappa}_S|_{M \times^{P_M} \mathfrak{r}_M^*}$ and $\hat{\kappa}_{S_M}$, as defined in (\ref{CameralHomLifteqn}), agree over $S_M^*$. Hence $\kappa_S$ and $\kappa_{S_M}$ agree on $S_M^*$.
\end{proof}

\begin{proof}[Proof of Proposition \ref{CameralSmoothprp}]
Since both $\mathcal I_S^{sm}$ and $\rho_S^* \hat{\mathcal J}_S$ are smooth over $S$, they are non-singular as varieties. So to check that $\kappa_S$ is smooth it suffices to check that its differential $d\kappa_S$ is a surjective map on tangent bundles. Since $\kappa_S$ is a homomorphism of group schemes, this is equivalent to the induced map
\begin{equation}\label{DiffCameralHomeqn}
    \kappa_{S,\Lie}:\Lie(\mathcal I_S^{sm}) \rightarrow \Lie(\rho_S^* \hat{\mathcal J}_S)
\end{equation}
being a surjective map on vector bundles, where these vector bundles are the relative Lie algebras for the group schemes. If the locus where $\kappa_{S,\Lie}$ fails to be surjective were non-empty, it would be a divisor of $S$, so it suffices to check this on an open set whose complement has codimension 2.

For any Levi subgroup $M \leq G$, we let $\mathcal O_{\mathfrak{m}}$ be defined by (\ref{InducedOrbiteqn}) and denote the decomposition class associated to the decomposition data $(M,\mathcal O_{\mathfrak{m}})$ by $\mathcal D(M,\mathcal O_{\mathfrak{m}})$. Consider
\begin{equation}\label{CameralSmooth1eqn}
    U = \bigcup_{M \geq L} \mathcal D(M,\mathcal O_{\mathfrak{m}}), 
\end{equation}
where $M$ runs over the Levi subgroups of $G$ which contain $L$ minimally (as in Definition \ref{ClassRamTypedef}). By \cite[Korollar 3.6]{Borho2}, $U$ is an open subset of $S$ whose complement has codimension 2. So it suffices to check for all $x \in U$ that $(\kappa_{S,\Lie})_x$ is surjective; in fact, since $\kappa_S$ is $G$-equivariant, it suffices to check surjectivity of $(\kappa_{S, \Lie})_x$ for any representative of the orbit $\Ad(G)(x)$. 

If $x$ is semisimple, then the statement is clear by Proposition \ref{CameralHomprp}, so we may assume that $x \in \mathcal D(M, \mathcal O_{\mathfrak{m}})$, where $M$ contains $L$ as a maximal proper Levi subgroup.  As in the discussion above, by replacing $x$ with a different representative of its $G$-orbit, we may assume that $x \in S_M$ (the Dixmier sheet in $\mathfrak{m}$ associated with $L$). By the assumption on $S$, $M$ is a classical group, and in particular the sheet $S_M$ is non-singular by Theorem \ref{ClassicalSheetsthm}; so $\kappa_{S,x} = \kappa_{S_M,x}$ as in Lemma \ref{CameralLIlem}. Thus $(\kappa_{S,\Lie})_x = (\kappa_{S_M,\Lie})_x$, and so surjectivity of $(\kappa_{S,\Lie})_x$ is implied by the smoothness of $\kappa_{S_M}$.  Thus we have reduced to checking the statement in the case that $S$ is associated with a maximal proper Levi subgroup $L$ in a classical group. This is done in Proposition \ref{MLCameralSmoothprp}; we sketch the argument below.

If $L$ is a maximal Levi subgroup in $G$, the centre $\mathfrak{z}$ of $\mathfrak{l}$ is 1-dimensional; to prove that $\kappa_S$ is smooth, it suffices to do so at a nilpotent element $e \in S$. There are two possibilities for the ramification of the map $p:\mathfrak{z} \rightarrow \mathcal B$ at $0$; it is either unramified or it has ramification of order $2$. In the unramified case, to prove that $\kappa_S$ is smooth it suffices to find an element of $\mathfrak{l}$ which centralises $e$ and does not map to $0$ under the abelianisation map; this is possible by Lemma \ref{KTAblem}. 

In the ramified case, we can complete $e$ to an $\mathfrak{sl}_2$-triple such that the inclusion of the corresponding copy of $\mathfrak{sl}_2$ into $\mathfrak{g}$ induces an inclusion on the respective centraliser subalgebras of $e$ and an isomorphism on the Lie algebras of the pseudo-cameral groups, compatible with the cameral homomorphisms. The smoothness of the cameral homomorphism for $S$ then follows from the smoothness of the cameral homomorphism in the regular case \cite[Proposition 12.5]{DG}.
\end{proof}

Since the cameral homomorphism is smooth, its image $\mathcal {I}m(\kappa_S)$ defines an open subgroup scheme of $\rho_S^*\hat{\mathcal J}_S$.

\begin{proposition}\label{CameralDescentprp}
    The group scheme $\mathcal{I}m(\kappa_S)$ descends under $\rho_S$ to a smooth open subgroup stack $\mathcal J_S$ of $\hat{\mathcal J}_S$, representable by schemes over $\mathcal B$.
\end{proposition}

\begin{proof}
As in Proposition \ref{SmCentraliserDescentprp}, to show that $\mathcal I m(\kappa_S)$ descends under $\rho_S$, it suffices to show that $\mathcal I m(\pi_1^*\kappa_S)$ and $\mathcal I m(\pi_2^*\kappa_S)$ coincide in $(\rho_S \circ \pi_1)^*\hat{\mathcal J}_S =(\rho_S \circ \pi_2)^*\hat{\mathcal J}_S$ over $S \times_{\mathcal B} S$ (where $\pi_1$ and $\pi_2$ are the projection maps to $S$). This follows from the $G$-equivariance of $\kappa_S$.

The resulting subgroup stack $\mathcal J_S$ of $\hat{\mathcal J}_S$ is open and smooth over $\mathcal B$ by descent, and is representable by schemes over $\mathcal B$ since it is an open substack of $\hat{\mathcal J}_S$.
\end{proof}

\begin{remark}\label{SmallCameralgprmk}
    We call $\mathcal J_S$ the \emph{cameral group}. Since it is open in the pseudo-cameral group $\hat{\mathcal J}_S$, $\mathcal J_S$ has finite index in $\hat{\mathcal J}_S$, i.e. for any map $X \rightarrow \mathcal B$ where $X$ is a connected scheme, the group $(\mathcal J_S)_{\mathcal B}(X)$ has finite index in $(\hat{\mathcal J}_S)_{\mathcal B}(X)$.
    
    It seems likely that there should be an analogue of \cite[Proposition 2.4.7]{Ngo2} describing $\mathcal J_S$ inside $\hat{\mathcal J}_S$ in terms of vanishing conditions at ramification points of $p$ determined by the root system for the reflection group $W_S$ constructed in Corollary \ref{KatsyloGpWeylGpcor}.
\end{remark}

The remark shows that if $\hat{\mathcal J}_S$ has connected fibres, then $\mathcal J_S = \hat{\mathcal J_S}$. In particular we have the following special case by connectedness of centralisers in $\GL_n$.

\begin{proposition}\label{GLnCameralgp}
    If $G=\GL_n$, the cameral group $\mathcal J_S$ is equal to the pseudo-cameral group $\hat{\mathcal J}_S$.
\end{proposition}

The following gives an interpretation of the cameral group as an abelianisation of the smooth centraliser.

\begin{proposition}\label{SmAbelianisationprp}
    Let $\mathcal A$ be a separated commutative group scheme over $S$, and let $\phi:\mathcal I_S^{sm} \rightarrow \mathcal A$ be a homomorphism of group schemes. There is a unique homomorphism of group schemes $\phi^{ab}:\rho_S^*\mathcal J_S \rightarrow \mathcal A$ inducing a factorisation of $\phi$ as
    \begin{equation}\label{SmAbelianisationeqn}
\begin{tikzcd}
\mathcal I_S^{sm} \arrow[r, "\kappa_S"] & \rho_S^*\mathcal J_S \arrow[r, "\phi^{ab}"] & \mathcal A.
\end{tikzcd}
    \end{equation}
\end{proposition}

\begin{proof} 
Let $\mathcal N_S$ be the kernel of $\kappa_S$ and $\mathcal N_\mathcal A$ be the kernel of $\phi$. These are both closed subgroup schemes of $\mathcal I_S^{sm}$ since $\rho_S^*\mathcal J$ and $\mathcal A$ are separated over $S$; moreover $\mathcal N_S$ is smooth over $S$ (as in Corollary \ref{SmoothCentralisercor}). Thus, to show that $\mathcal N_S$ is contained in $\mathcal N_A$ it suffices to observe that this containment occurs over the open subset of $S^{ss} \subseteq S$. But this occurs since $\kappa_S|_{S^{ss}}$ is the fibrewise abelianisation of $\mathcal I_S^{sm}$ by Proposition \ref{CameralHomprp}, so that for any $x \in S^{ss}$, $\mathcal N_{S,x} = (\mathcal I_{S,x}^{sm})^{der} \leq \Ker(\phi)_x$. Since $\kappa_S$ is a smooth homomorphism of group schemes with kernel $\mathcal N_S$, the map $\phi:\mathcal I_S^{sm} \rightarrow \mathcal A$ descends to a homomorphism $\phi^{ab}$ with the required factorisation (\ref{SmAbelianisationeqn}).
\end{proof}

Thus the cameral homomorphism realises $\rho_S^*\mathcal J_S$ as the abelianisation of $\mathcal I_S^{sm}$ in the category of separated group schemes over $S$.

\subsection{The abelianised quotient stack}\label{AbQuotsbn}
We can use the cameral homomorphism to give a factorisation of the gerbe $\rho_S:[S/G] \rightarrow \mathcal B$ (constructed in Defintion \ref{StackyChevalleyBasedef}) into abelian and non-abelian parts. We assume that $S$ is a non-singular Dixmier sheet such that the cameral homomorphism is smooth as in Proposition \ref{CameralSmoothprp} (e.g. if $S$ is a sheet in a classical Lie algebra) and keep the notation of the previous section.

\begin{lemma}\label{CamKernelDescentlem}
    The kernel $\mathcal N_S$ of the cameral homomorphism $\kappa_S$, defined in Proposition \ref{CameralHomprp}, descends to a smooth closed normal subgroup stack $\mathcal N_{S,G}$ of the inertia stack $\mathcal I_{S,G}$ for $[S/G]$, representable in schemes over $[S/G]$.
\end{lemma}

\begin{proof} 
We need only note that $\mathcal N_S$ is stable under the $G$-action, and normal in $\mathcal I_S$ (not just $\mathcal I_S^{sm}$), since $\kappa_S$ is $G$-equivariant.
\end{proof}

As in Section \ref{AdjointQGerbesbn}, we use the rigidification construction given in Appendix A of \cite{AOV}.

\begin{definition}\label{AbelianisedStackdef}
    We define the \emph{abelianised quotient stack} $[S/G]^{ab}$ to be the rigidification of $[S/G]$ by $\mathcal N_{S,G}$.
\end{definition}

We will denote the rigidification map by $\rho_S^n:[S/G] \rightarrow [S/G]^{ab}$.

\begin{proposition}\label{AbnGerbeprp}
    The map $\rho_S^n:[S/G] \rightarrow [S/G]^{ab}$ is a gerbe. There is a map $\rho_S^{ab}:[S/G]^{ab} \rightarrow \mathcal B$ which induces a factorisation of the gerbe $\rho_S:[S/G] \rightarrow \mathcal B$ as
    \begin{equation}\label{AbnGerbe1eqn}
\begin{tikzcd}
{[S/G]} \arrow[r, "\rho_S^n"] & {[S/G]^{ab}} \arrow[r, "\rho_S^{ab}"] & \mathcal B.
\end{tikzcd}
    \end{equation}
    The map $\rho_S^{ab}$ is a gerbe over $\mathcal B$ banded by $\mathcal J_S$.
\end{proposition}

\begin{proof}
The first statement is \cite[Theorem A.1 (a)]{AOV}. The construction of the map $\rho_S^{ab}$ and the factorisation (\ref{AbnGerbe1eqn}) are straightforward using the construction of the rigidification in \cite[Theorem A.1]{AOV}. Then since $\rho_S$ is a gerbe by Proposition \ref{StackyChevalleyBaseprp}, $\rho_S^{ab}$ must also be a gerbe, and it has structure group $\mathcal J_S$ by the construction of $\mathcal N_G$ in Lemma \ref{CamKernelDescentlem}.
\end{proof}

\begin{remark}\label{AbnGerbermk}
    The gerbe $\rho_S^{ab}$ is maximally abelian in the following sense: if $\mathcal X$ is an algebraic stack equipped with a map $\mathcal X \rightarrow \mathcal B$ whose relative inertia stack $\mathcal I_{\mathcal X/\mathcal B}$ is representable by separated commutative group schemes over $\mathcal X$, then for any map $[S/G] \rightarrow \mathcal X$ over $\mathcal B$, we can use Proposition \ref{SmAbelianisationprp} to construct an induced map $[S/G]^{ab} \rightarrow \mathcal X$.
\end{remark}

By the $\mathbb{G}_m$-equivariance of the cameral homomorphism, the $\mathbb{G}_m$-action on $\hat{\mathcal J_S}$ of Remark \ref{PseudoCameralgprmk} restricts to a $\mathbb{G}_m$-action on $\mathcal J_S$; so $\mathcal J_S$ descends to a smooth subgroup stack $\mathcal J_{S, \mathbb{G}_m}$ of $\hat{\mathcal J}_{S, \mathbb{G}_m}$, representable by schemes over $\mathcal B/\mathbb{G}_m$. Moreover, there is a strict $\mathbb{G}_m$-action on $[S/G]^{ab}$ such that the maps $\rho_S^n$ and $\rho_S^{ab}$ are $\mathbb{G}_m$-equivariant. Hence, we have the following $\mathbb{G}_m$-equivariant version of Proposition \ref{AbnGerbeprp}. 

\begin{corollary}\label{GmEquivAbnGerbecor}
    The maps
    $$\rho_{S,\mathbb{G}_m}^n:[S/G\times\mathbb{G}_m] \rightarrow [S/G]^{ab}/\mathbb{G}_m,$$
    $$\rho_{S,\mathbb{G}_m}^{ab}:[S/G]^{ab}/\mathbb{G}_m \rightarrow \mathcal B/\mathbb{G}_m$$
    induced by (\ref{AbnGerbe1eqn}) are both gerbes, and $\rho_{S,\mathbb{G}_m}^{ab}$ is banded by $\mathcal J_{S, \mathbb{G}_m}$.
\end{corollary}

We will drop the suffix $\mathbb{G}_m$ from these maps when there is no possibility of confusion.

\section{Non-abelian Hitchin fibrations}\label{NonabnHitchinscn}
We now apply the above considerations to the moduli stack of $G$-Higgs bundles on an algebraic curve $\Sigma$ to prove the main results of the paper.

\subsection{Higgs bundles with fixed centraliser dimension}\label{SheetHiggssbn}
We let $G$ be a connected reductive group, and consider $G$-Higgs bundles on $\Sigma$ twisted by an arbitrary line bundle. As in \cite{Ngo1}, we view the moduli stack of $G$-Higgs bundles with arbitrary twists as a mapping stack.

\begin{definition}\label{HiggsModuliStackdef}
    The \emph{moduli stack of twisted $G$-Higgs bundles on $\Sigma$} is the mapping stack
    \begin{equation}\label{HiggsModuliStack1eqn}
        \mathcal M(G) = \Maps(\Sigma,[\mathfrak{g}/G\times{\mathbb{G}_m}]).
    \end{equation}
    If $\mathcal L$ is a line bundle on $\Sigma$, then the \emph{moduli stack of $\mathcal L$-twisted $G$-Higgs bundles on $\Sigma$}, $\mathcal M_\mathcal L(G)$, is the fibre of the map
    \begin{equation}\label{HiggsModuliStack2eqn}
        \mathcal M(G) \rightarrow \Maps(\Sigma,[\Spec( \mathbb{C})/\mathbb{G}_m]) = {\bf Pic}(\Sigma)
    \end{equation}
    over the $\mathbb{C}$-point of ${\bf Pic}(\Sigma)$ defined by $\mathcal L$.
\end{definition}

\begin{remark}\label{HiggsModuliStackrmk}
    We regard a \emph{twisted $G$-Higgs bundle} as a triple $(E, \mathcal L, \Phi)$ for $E$ a $G$-bundle on $\Sigma$, $\mathcal L$ a line bundle on $\Sigma$ (the \emph{twist}) and $\Phi$ a global section of $\text{ad}(E) \otimes \mathcal L$ (the \emph{Higgs field}). We will always use $\mathcal L$ as a subscript to denote a fixed choice of twist in the subsequent constructions (as in e.g. Lemma \ref{SheetHiggsNonElem} below); we will also often switch between the line bundle $\mathcal L$ and its corresponding $\mathbb{G}_m$-torsor without changing notation. If $\mathcal L = K$ is the canonical bundle on $\Sigma$, we recover the usual $K$-twisted Higgs bundles on $\Sigma$.

    The stack $\mathcal M(G)$ is a quasi-separated algebraic stack, locally of finite presentation over $\mathbb{C}$ by \cite[Theorem 1.2]{HR}. It is usual to restrict $\mathcal M(G)$ to a locus where the twist $\mathcal L$ has sufficiently high degree, to ensure that $\mathcal M_{\mathcal L}(G)$ has reasonable geometric properties; we will in general place no restrictions on $\mathcal L$, but we will sometimes require that it admits a square root and/or a global section (e.g. see Lemma \ref{SheetHiggsNonElem} and Proposition \ref{SHitchinBaseNormprp}).
\end{remark}

As in definition \ref{Centstratdef}, there is a centraliser stratification
\begin{equation}
    \mM(G) = \coprod_{\rk(G) \leq d \leq \dim(G)} \mM^{\gen, d}(G)
\end{equation}
on the moduli stack. In this section, we will exclusively focus on the subloci of fixed Higgs field centraliser dimension. Let $\mathcal M^{d}(G)$ be the locally closed substack
\begin{equation}\label{CentHiggsOpeneqn}
   \mathcal M^{d}(G) = \Maps(\Sigma, [\mathfrak{g}_d/G \times \mathbb{G}_m]) 
\end{equation}
of $\mathcal M(G)$. We can decompose $\mM^d(G)$ into closed substacks of twisted sheet-valued Higgs bundles.

\begin{definition}\label{SheetHiggsdef}
    For a sheet $S$ in the Lie algebra $\mathfrak{g}$, we define the \emph{moduli stack of twisted $S$-valued $G$-Higgs bundles on $\Sigma$} as the substack $\mathcal M(G;S)$ of $\mathcal M$ defined by
    \begin{equation}\label{Sheetiggseqn}
        \mathcal M(G;S) = \Maps(\Sigma,[S/G\times\mathbb{G}_m]).
    \end{equation}
    We refer to the $\mathbb{C}$-points of $\mathcal M(G;S)$ as \emph{twisted $S$-valued $G$-Higgs bundles on $\Sigma$}.
\end{definition}

\begin{remark}\label{SheetHiggsrmk}
   If $S = \mathfrak{g}^{reg}$ is the regular sheet, $\mathcal M(G;S) = \mathcal M^{reg}(G)$ is the dense open substack of regular twisted $G$-Higgs bundles. Every twisted $G$-Higgs bundle is \emph{generically $S$-valued} for some (not necessarily unique) sheet $S$, i.e. the corresponding map $\Sigma \rightarrow [\mathfrak{g}/G \times \mathbb{G}_m]$ factors through $[S/G \times \mathbb{G}_m]$ at all but finitely many points of $\Sigma$; we write the substack of $\mM(G)$ of twisted generically $S$-valued $G$-Higgs bundles as $\mM^{\gen}(G;S)$.

   In the terminology of \cite[Section 4.3]{Ngo3}, $\mathcal M_\mathcal L(G;S)$ is the regular locus of the stack ${\bf M}(\tilde{S}_\mathcal L)$ of generalised Higgs bundles for the $G$-action on $\tilde{S}_\mathcal L := \tilde{S} \times^{\mathbb{G}_m} \mathcal L$, where $\tilde{S}$ is the normalisation of $\overline{S}$ as in Remark \ref{RegularQuotientrmk}.
\end{remark}

\begin{lemma}\label{SheetHiggsNonElem}
   If $\mathcal L$ is a line bundle on $\Sigma$ such that $\mathcal L$ admits a square root, the stack $\mathcal M_{\mathcal L}(G;S)$ is non-empty.
\end{lemma}

\begin{proof}
Let $S_{\mathcal L} = S \times^{\mathbb{G}_m}\mathcal L$; this is a fibre bundle in $S$ over $\Sigma$. Then
\begin{equation}\label{SheetHiggsNonE1eqn}
    \mathcal M_{\mathcal L}(G;S) = \Sec(\Sigma,[S_{\mathcal L}/G]),
\end{equation}
the stack of sections of $[S_{\mathcal L}/G]$ over $\Sigma$.

Let $\mathfrak{K}$ be a Katsylo slice for the sheet $S$ (see Definition \ref{KatsyloSlicedef}), and let $e \in \mathfrak{K}$ be the corresponding nilpotent element. We choose a square root $\mathcal L^{1/2}$ of $\mathcal L,$ and set $\mathfrak{K}_\mathcal L = \mathfrak{K} \times^{\mathbb{G}_m} \mathcal L^{1/2}$ where the $\mathbb{G}_m$-action on $\mathfrak{K}$ is the Kazhdan action of Definition \ref{KazhdanActionlem}. Let $(U_i)_{i\in I}$ be a trivialising cover on $\Sigma$ for the $\mathbb{G}_m$-torsor $\mathcal L^{1/2}$, and let $(s_i)_{i \in I}$ be local sections of $\mathcal L^{1/2}$ (as a $\mathbb{G}_m$-torsor) over this cover. These locally define sections of $\mathfrak{K}_\mathcal L$, given on $\mathbb{C}$-points $x \in U_i$ by
\begin{equation}\label{SheetHiggsNonE2eqn}
    x  \mapsto \mathbb{G}_m(e,s_i(x)); 
\end{equation}
and since $e$ is fixed by the $\mathbb{G}_m$-action, these agree on overlaps and thus define a global section $s:\Sigma \rightarrow \mathfrak{K}_\mathcal L$. The choice of square root of $\mathcal L$ defines a map $\mathfrak{K}_\mathcal L \rightarrow [S_\mathcal L/G]$ over $\Sigma$ (see Proposition \ref{StackyHKSectionprp}); hence $s$ defines a point of $\mathcal M_\mathcal L(G;S)$.
\end{proof}

\begin{proposition}\label{SheetHiggsOpenprp}
    The stack $\mathcal M^{d}(G)$ has a decomposition
    \begin{equation}\label{SheetHiggsOpen1eqn}
        \mathcal M^{d}(G) = \bigcup_{S \subseteq \mathfrak{g}_d} \mathcal M(G;S)
    \end{equation}
    into non-empty closed substacks, where $S$ runs across the sheets of $\mathfrak{g}$ contained in $\mathfrak{g}_d$.
\end{proposition}

\begin{proof}
Fix a point $p \in \Sigma$; this determines an evaluation map $$\ev_p:\mathcal M^{d}(G) \rightarrow [\mathfrak{g}_d/G\times \mathbb{G}_m].$$ For any sheet $S \subseteq \mathfrak{g}_d$, the substack $\mathcal M(G;S)$ maps under $\ev_p$ to the closed substack $[S/G \times \mathbb{G}_m]$; conversely, any point of $\mathcal M^{d}(G)$ mapping to a point of $[S/G\times\mathbb{G}_m]$ under $ev_p$ must be a point of $\mathcal M(G;S)$, since $\Sigma$ is irreducible and $S$ is an irreducible component of $\mathfrak{g}_d$. Hence we have the decomposition (\ref{SheetHiggsOpen1eqn}) by pulling back the decomposition
\begin{equation}\label{SheetHiggsOpen3eqn}
      [\mathfrak{g}_d/G\times\mathbb{G}_m] = \bigcup_{S \subseteq \mathfrak{g}_d} [S/G\times \mathbb{G}_m].
\end{equation}
\end{proof}

\subsection{The \texorpdfstring{$S$}{S}-Hitchin map}\label{SHitchinsbn}
 We study the restriction of the Hitchin map to the locally closed substack $\mathcal M(G;S)$. We first recall the usual definitions \cite{Hitchin2}, again using the perspective of \cite{Ngo1}.

\begin{definition}\label{HitchinBasedef}
    The \emph{Hitchin base} (for $G$-Higgs bundles on $\Sigma$) is the mapping stack
    \begin{equation}\label{HitchinBase1eqn}
        \mathcal A(G) = \Maps(\Sigma,[\mathfrak{c}/\mathbb{G}_m])
    \end{equation}
    where $\mathfrak{c} = \mathfrak{t}/W$.
\end{definition}

\begin{remark}\label{HitchinBasermk}
    As in Definition \ref{HiggsModuliStackdef}, we will continue to use a subscript to denote a fixed twist $\mathcal L$ for the Hitchin base and its variants we consider below. For any $\mathcal L$, the stack $\mathcal A_{\mathcal L}(G)$ is representable by a vector space, namely the space of sections of the vector bundle $(\mathfrak{t} \otimes \mathcal L)/W$ on $\Sigma$. Hence, $\mathcal A(G)$ is a smooth algebraic stack, representable by linear schemes over ${\bf Pic}(\Sigma)$.
\end{remark}

\begin{definition}\label{HitchinMapdef}
    The \emph{Hitchin map} (for $G$-Higgs bundles on $\Sigma$) is the morphism of stacks $h_G:\mathcal M(G) \rightarrow \mathcal A(G)$ induced by the map $\chi:[\mathfrak{g}/G\times \mathbb{G}_m]\rightarrow [\mathfrak{c}/\mathbb{G}_m]$ which descends from the Chevalley map.
\end{definition} 

For the rest of this section, we will fix a non-singular sheet $S$ of $\mathfrak{g}$. We will use the constructions of Section \ref{AdjointQGerbesbn} to define an augmented version of the Hitchin map on $\mathcal M(G;S)$.

\begin{definition}\label{SHitchinmapdef}
    We define the \emph{$S$-Hitchin base} $\mathcal A (G;S)$ to be the mapping stack
    \begin{equation}\label{SHitchinBase1eqn}
        \mathcal A(G;S) = \Maps(\Sigma,\mathcal B/\mathbb{G}_m)
    \end{equation}
    where $\mathcal B$ is the $S$-Chevalley base as defined in Definition \ref{StackyChevalleyBasedef}.

    The \emph{$S$-Hitchin map} $h_S:\mathcal M(G;S) \rightarrow \mathcal A(G;S)$ is the morphism of stacks induced by the map $\rho_S:[S/G\times \mathbb{G}_m] \rightarrow \mathcal B/\mathbb{G}_m$ defined as in Corollary \ref{GmEquivStackyCBcor}.
\end{definition}
\begin{remark}\label{SHitchinBasermk}
     Using the notation of Remark \ref{SheetHiggsrmk}, if ${\bf h}_{\tilde{S}}:{\bf M}(\tilde{S}_\mathcal L) \rightarrow {\bf A}(\tilde{S}_\mathcal L)$ is the generalised Hitchin fibration for the $G$-action on $\tilde{S}$ in the sense of \cite[Section 4.3]{Ngo3}, then the generalised Hitchin base ${\bf A}(\tilde{S}_\mathcal L)$ is the stack of maps from $\Sigma$ to $\mathfrak{c}_{S, \mathcal L} := \mathfrak{c}_S \times^{\mathbb{G}_m} \mathcal L$. From the proof of Proposition \ref{StackyCBDiagramsprp}, the diagram (\ref{StackyCBDiagrams1eqn}) induces a factorisation
     \begin{equation}\label{SHitchinBase2eqn}
        \begin{tikzcd}
            \mathcal M_\mathcal L(G;S) \arrow[r, "h_S"] & \mathcal A_\mathcal L(G;S) \arrow[r, "\gamma_S"] & {\bf A}(\tilde{S}_\mathcal L) \arrow[r, "\mu_S"] & \mathcal A_\mathcal L(G),
        \end{tikzcd}
     \end{equation}
     of the restriction of the usual Hitchin map $h$ to $\mathcal M(G;S)$; in general, neither of the last two maps in (\ref{SHitchinBase2eqn}) are isomorphisms. The composition $\gamma_S \circ h_S$ of the first two maps in (\ref{SHitchinBase2eqn}) is the restriction of ${\bf h}_{\tilde{S}}$ to $\mathcal M_\mathcal L(G;S)$, and we denote the composition $\mu_S \circ \gamma_S$ of the last two maps in (\ref{SHitchinBase2eqn}) by $\tilde{\mu}_S:\mathcal A(G;S) \rightarrow \mathcal A(G)$.
\end{remark}

\begin{proposition}\label{SHitchinBasePropsprp}
    Let $\mathcal L$ be a fixed twisting line bundle on $\Sigma$.
    \begin{itemize}
        \item[(i)] $\mathcal A_\mathcal L(G;S)$ is a smooth non-empty Deligne-Mumford stack.
        \item[(ii)] The connected components of the stack $\mathcal A_\mathcal L(G;S)$ are indexed by the (finite) set $H^1(\Sigma,F)$ of $F$-torsors on $\Sigma$ up to isomorphism, where $F$ is the Katsylo group defined in Definition \ref{KatsyloGpdef}.
        \item[(iii)] The stack $\mathcal A_\mathcal L(G;S)$ is representable by a scheme if and only if $F$ is trivial.
    \end{itemize}
\end{proposition}

\begin{proof} 
We choose a Katsylo slice $\mathfrak{K}$ for $S$ as defined in Definition \ref{KatsyloSlicedef}, and identify $\mathcal B$ with $[\mathfrak{K}/F]$ by Proposition \ref{StackyCBKQprp}; this induces an identification 
\begin{equation}\label{SHitchinBaseProps2eqn}
    \mathcal A_\mathcal L(G;S) \cong \Sec(\Sigma,[\mathfrak{K}_\mathcal{L}/F]),
\end{equation}
for $\mathfrak{K}_\mathcal{L} = \mathfrak{K} \times^{\mathbb{G}_m}\mathcal L$ as in Lemma \ref{SheetHiggsNonElem} (note that $\mathfrak{K}_{\mathcal L}$ now does not depend on the choice of square root of $\mathcal L$ since the Kazhdan action on $\mathfrak{K}$ admits a square root by Corollary \ref{KatsyloAffinecor}).

There is a map $\mathcal A_\mathcal L(G;S) \rightarrow {\bf B}_{\Sigma}F$ induced by the map $[\mathfrak{K}_{\mathcal{L}}/F] \ra \B F_{S,e}$. The fibre over the $\mathbb{C}$-point of ${\bf B}_{\Sigma}F$ defined by an $F$-torsor $\pi:\tilde{\Sigma} \rightarrow \Sigma$ is the space $\Sec(\Sigma,\tilde{\Sigma} \tm^F \fK_{e})$, which can alternatively be viewed as the space of $F$-equivariant maps from $\tilde{\Sigma}$ to $\mathfrak{K}_{\mathcal L}$ over $\Sigma$.

By Corollary \ref{KatsyloAffinecor}, $\mathfrak{K}_{\mathcal L}$ can be identified with a vector bundle on $\Sigma$, so that the space of maps from $\tilde{\Sigma}$ to $\mathfrak{K}_{\mathcal L}$ over $\Sigma$ is the affine space $H^0(\tilde{\Sa},\pi^*\fK_{\mL})$, and the fixed point space
\begin{equation}\label{FEquivMapseqn}
    \Sec(\Sigma,\tilde{\Sigma} \tm^F \fK_{\mL}) = (H^0(\tilde{\Sa},\pi^*\fK_{\mL}))^F
\end{equation}
is a non-singular variety. This proves (i).

The space $(H^0(\tilde{\Sigma},\pi^*\mathfrak{K}_{\mathcal L}))^F$ is connected: since the global scalar action on the vector bundle $\mathfrak{K}_{\mathcal L}$ commutes with the $F$-action, there is a $\mathbb{G}_m$-action on $(H^0(\tilde{\Sigma},\pi^*\mathfrak{K}_{\mathcal L}))^F$ contracting the variety to the point $0_{\tilde{\Sigma}} \in (H^0(\tilde{\Sigma},\pi^*\mathfrak{K}_{\mathcal L}))^F$ representing the zero-section of $\pi^*\mathfrak{K}_{\mathcal L}$. This proves (ii).

From the above, if $F$ is trivial then $\mathcal A_\mathcal L(G;S)$ is representable by an affine space. Conversely, if $F$ is non-trivial, the image in $\mathcal A_\mathcal L(G;S)$ of the point $0_{\tilde{\Sigma}} \in (H^0(\tilde{\Sigma},\pi^*\mathfrak{K}_{\mathcal L}))^F$, for an $F$-torsor $\tilde{\Sigma} \rightarrow \Sigma$, has non-trivial automorphisms. This proves (iii).
\end{proof}

\begin{remark}\label{SHitchinBasePropsrmk}
    The proof of the Proposition \ref{SHitchinBasePropsprp} shows that there is a distinguished component $\mathcal A^0_{\mathcal L}(G;S)$ of $\mathcal A_\mathcal L(G;S)$, corresponding to the trivial $F$-torsor on $\Sigma$. This component is exacly the image of the map $q:H^0(\Sigma,\mathfrak{K}_\mathcal L) \rightarrow \mathcal A(G;S)$ and $q$ realises $\mathcal A^0_{\mathcal L}(G;S)$ as the quotient $[H^0(\Sigma,\mathfrak{K}_\mathcal L)/F]$.
    
    In particular, if $F$ is trivial, then $\mathcal A_{\mathcal L}(G;S) = \mathcal A_{\mathcal L}^0(G;S)$, and this can be represented by the affine space
    \begin{equation}\label{SHitchinBasePropseqn}
        H^0(\Sigma,\mathfrak{K}_\mathcal L) = \bigoplus_{i=1}^rH^0(\Sigma,\mathcal L^{e_i}),
    \end{equation}
    for the weights $e_i \in \mathbb{N}$ of Corollary \ref{KatsyloAffinecor}.
\end{remark}

\begin{proposition}\label{SHitchinBaseNormprp}
    The map $\tilde{\mu}_S:\mathcal A_{\mathcal L}(G;S) \rightarrow \mathcal A_\mathcal L(G)$ of Remark \ref{SHitchinBasermk} is finite. There is an open subscheme $\mathcal A_{\mathcal L}^\heartsuit(G;S)$ of $\mathcal A_{\mathcal L}(G;S)$ on which $\tilde{\mu}_S$ is injective, and if $\mathcal L$ admits a global section, $\mathcal A^\heartsuit_{\mathcal L}(G;S)$ has non-trivial intersection with the component $\mathcal A_\mathcal L^0(G;S)$.
\end{proposition}

\begin{proof}
Since $\tilde{\mu}_S$ is $\Gm$-equivariant with respect to contracting $\Gm$-actions on the connected components of $\mA_\mL(G;S)$ and on $\mA_\mL(G)$, to check that $\tilde{\mu}_S$ is finite, it suffices to check that $\tilde{\mu}_S$ is quasi-finite by \cite[Proposition B.5]{FHZ}.

We choose a Katsylo slice $\mathfrak{K}$ and decomposition data $(L,\mathcal O)$ for $S$ (see Remark \ref{Decompdatarmk}) and identify $\mathfrak{K}$ with a finite quotient of $\mathfrak{z} = \Lie(Z(L))$ as in Corollary \ref{KatsyloAffinecor}. The map $\tilde{\mu}_S$ is induced by the composition
\begin{equation}\label{SHitchinBaseNormeqn}
\begin{tikzcd}
\mathcal B \arrow[r, "C"] & \mathfrak{z}/W_L \arrow[r, "\nu_L"] & \mathfrak{t}/W
\end{tikzcd}
\end{equation}
where $C$ is the coarsification map $C:[\mathfrak{K}/F] \rightarrow \mathfrak{K}/F \cong \mathfrak{z}/W_L$ (identifying $\mathcal B$ with $[\mathfrak{K}/F]$ as in the proof of the previous proposition), and $\nu_L$ is a normalisation onto its image. We will denote $\mathfrak{c}_{S,\mathcal L} = \mathfrak{z} \otimes\mathcal L/W_L$ and $\mathfrak{c}_\mathcal L = \mathfrak{t} \otimes \mathcal L/W$.

Suppose $\sigma:\Sigma \rightarrow \mathfrak{c}_\mathcal L$ represents a $\mathbb{C}$-point of $\mathcal A(G)$ in the image of $\tilde{\mu}_S$. Any lift $\hat{\sigma}:\Sigma \rightarrow \mathfrak{c}_{S,\mathcal L}$ of $\sigma$ is equivalent to a section $\hat{\sigma}:\Sigma \rightarrow \Sigma \times_{\mathfrak{c}_\mathcal L}\mathfrak{c}_{S,\mathcal L}$ of the finite surjective map $\Sigma \times_{\mathfrak{c}_\mathcal L}\mathfrak{c}_{S,\mathcal L} \rightarrow \Sigma$. Since $\Sigma$ is non-singular and connected, and every irreducible component of $\Sigma \times_{\mathfrak{c}_\mathcal L}\mathfrak{c}_{S,\mathcal L}$ has dimension less than or equal to that of $\Sigma$, such a section $\hat{\sigma}$ is an identification of $\Sigma$ with the reduced subscheme of an irreducible component of $\Sigma \times_{\mathfrak{c}_\mathcal L}\mathfrak{c}_{S,\mathcal L}$; as there are finitely many of these, there are finitely many such $\hat{\sigma}$.

Meanwhile, any map $\hat{\sigma}':\Sigma \rightarrow \mathcal B_\mathcal L:=\mathcal B \times^{\mathbb{G}_m}\mathcal L \cong [\mathfrak{K}_\mathcal L/F]$ lifting the choice of map $\hat{\sigma}$ corresponds to an isomorphism class of pairs $(\tilde{\Sigma},\tilde{\sigma})$ where $\tilde{\Sigma}$ is an $F$-torsor on $\Sigma$ and $\tilde{\sigma}$ is a section of $\tilde{\Sigma} \times^F\mathfrak{K}_\mathcal L$ lifting the map $\hat{\sigma}:\Sigma \rightarrow \mathfrak{c}_{S,\mathcal L} \cong  \mathfrak{K}_\mathcal L/F$. There are finitely many choices of $F$-torsor $\tilde{\Sigma}$ up to isomorphism, and for a fixed $\tilde{\Sigma}$ there are finitely many choices of $\tilde{\sigma}$ by the same argument as the previous paragraph. Hence, there are finitely many $\mathbb{C}$-points of $\mathcal A_\mathcal L(G;S)$ mapping to $\sigma$. This proves finiteness for $\tilde{\mu}_S$.

Consider the $\mathbb{G}_m$-stable open subscheme $\mathcal B^{rs}$ of $\mathcal B$ given by the image of the open subscheme $\mathfrak{z}^{rs} \subseteq \mathfrak{z}$ (defined in (\ref{WeylLeviActioneqn})) under the map $p:\mathfrak{z}\rightarrow \mathcal B$ of Lemma \ref{StackyZquotientlem}. This defines an open substack $\mathcal A_{\mathcal L}^\heartsuit(G;S)$ of $\mathcal A_\mathcal L(G;S)$ with $\mathbb{C}$-points corresponding to maps $\tau:\Sigma \rightarrow \mathcal B_{\mathcal L}$ whose image has non-trivial intersection with $\mathcal B^{rs}_\mathcal L := \mathcal B^{rs} \times^{\mathbb{G}_m} \mathcal L$. Moreover, $\mathcal A_{\mathcal L}^\heartsuit(G;S)$ is representable by a scheme: $\mathcal A_{\mathcal L}^\heartsuit(G;S)$ has trivial inertia stack by \cite[Proposition A.1]{FMN} since $\mathcal B^{rs}_{\mathcal L}$ is a scheme, so is an algebraic space; and by \cite[Théorème A.2]{LMB} it is a scheme since it admits a quasi-finite map $\tilde{\mu}_S$ to a scheme.

Suppose now $\sigma:\Sigma \rightarrow \mathfrak{c}_\mathcal L$ represents a $\mathbb{C}$-point of $\mathcal A_\mathcal L(G)$ in the image of $\mathcal A_{\mathcal L}^\heartsuit(G;S)$ under $\tilde{\mu}_S$. The map $\tilde{\nu}_S$ (defined by (\ref{SHitchinBaseNormeqn})) is an isomorphism over $\mathcal B^{rs}_{\mathcal L}$ and so any two lifts $\tau_1:\Sigma \rightarrow \mathcal B_\mathcal L$ and $\tau_2:\Sigma \rightarrow \mathcal B_\mathcal L$ of $\hat{\sigma}:\Sigma \rightarrow \mathcal \mathfrak{c}_{S,\mathcal L}$ agree over $\mathcal B^{rs}_{\mathcal L}$. But then by \cite[Proposition A.1]{FMN}, the maps $\tau_1$ and $\tau_2$ coincide. Hence $\tilde{\mu}_S$ is injective on $\mathcal A_{\mathcal L}^\heartsuit(G;S)$. If $\mathcal L$ admits a global section, then $\mathcal A_{\mathcal L}^\heartsuit(G;S)$ has non-empty intersection with $\mathcal A^0_{\mathcal L}(G;S)$, since any section of $\mathfrak{z} \otimes \mathcal L$ which is not contained in one of a finite number of root hyperplanes determines a point of $\mathcal A^0_{\mathcal L}(G;S)$ which is contained in $\mathcal A_{\mathcal L}^\heartsuit(G;S)$.
\end{proof}

\begin{remark}\label{GenHFAheartrmk}
    We note that there is a similar construction in \cite[Section 4.3]{Ngo3} applying in a different context to other generalised Hitchin fibrations.
\end{remark}

We can use the results of Section \ref{AdjointQGerbesbn} to describe the fibres of $h_S$. An $S$-valued Higgs bundle $(E,\mathcal L, \Phi)$ on $\Sigma$ corresponds to a map $f_{(E, \mathcal L, \Phi)}:\Sigma \rightarrow [S/G\times \mathbb{G}_m]$, and thus we can define a group scheme $\mathcal I^{sm}_{(E, \mathcal L, \Phi)} = f^*_{(E, \mathcal L, \Phi)}\mathcal I^{sm}_{S,G\times \mathbb{G}_m}$ over $\Sigma$ (in the notation of Corollary \ref{GmEquivStackyCBcor}).

\begin{theorem}\label{HiggsNonAbnthm}
         Given an $S$-valued Higgs bundle $(E, \mathcal L, \Phi)$ on $\Sigma$ mapping to a $\mathbb{C}$-point $\tau$ of $\mathcal A(G;S)$, the fibre $h_S^{-1}(\tau)$ can be identified with the stack ${\bf B}_{\Sigma} \mathcal I^{sm}_{(E, \mathcal L, \Phi)}$ of $\mathcal I^{sm}_{(E, \mathcal L, \Phi)}$-torsors on $\Sigma$.
\end{theorem}

\begin{proof}
This follows from Corollary \ref{GmEquivStackyCBcor} since the mapping stack construction is functorial. 
\end{proof}

\begin{remark}\label{HiggsNonAbnrmk}
    For a different choice of Higgs bundle $(F, \mathcal L', \Psi)$ mapping to the same point $\tau$, the group scheme $\mathcal I^{sm}_{(F, \mathcal L', \Psi)}$ is inner equivalent to $\mathcal I^{sm}_{(E, \mathcal L, \Phi)}$, but not isomorphic in general. Unless $S$ is the regular sheet, the group scheme $\mathcal I^{sm}_{(E, \mathcal L, \Phi)}$ is not commutative. 

    By Proposition \ref{SHitchinBaseNormprp}, the intersection of the usual Hitchin fibre $h_G^{-1}(\sigma)$ with $\mathcal M(G;S)$ is a disjoint finite union of stacks of torsors on $\Sigma$. 
\end{remark}

We can construct an analogue to the Hitchin section over the distinguished component $\mathcal A^0_\mathcal L(G;S)$, as in \cite[Proposition 2.5]{Ngo1}. We suppose $\mathcal L$ is a line bundle on $\Sigma$ which admits a square-root. Let $\mathcal M_\mathcal L^0(G;S)$ be the restriction of $\mathcal M_\mathcal L(G;S)$ to $\mathcal A_\mathcal L^0(G;S)$; $\mathcal M_\mathcal L^0(G;S)$ is a union of irreducible components of $\mathcal M_{\mathcal L}(G;S)$.

\begin{proposition}\label{HKMSectionprp}
     A choice of Katsylo slice $\mathfrak{K}$ for $S$ and a choice of square root $\mathcal L^{1/2}$ of $\mathcal L$ determines a map $\varepsilon: H^0(\Sigma,\mathfrak{K}_\mathcal L) \rightarrow \mathcal M_\mathcal L^0(G;S)$ making the diagram
\begin{equation} \label{HKMSectioneqn}    
\begin{tikzcd}
                   & \mathcal M^0_\mathcal L(G;S) \arrow[d, "h_S"] \\
{H^0(\Sigma,\mathfrak{K}_\mathcal L)} \arrow[r, "q"] \arrow[ru, "\varepsilon"] & \mathcal A^0_\mathcal L(G;S)                 
\end{tikzcd}
\end{equation}
     commute. In particular, $\mathcal A_{\mathcal L}^0(G;S)$ is in the image of $h_S$.
\end{proposition}

\begin{proof} 
Proposition \ref{StackyHKSectionprp} shows that the choice of square-root of $\mathcal L$ and Katsylo slice defines a map $\mathfrak{K}_\mathcal L \rightarrow [S_{\mathcal L}/G]$ which induces the map $\varepsilon$, and commutativity of (\ref{HKMSectioneqn}) follows from commutativity of (\ref{StackyHKSection1eqn}).
\end{proof}

\begin{remark}\label{HKMSectionrmk}
    We refer to $\varepsilon$ as a \emph{Hitchin-Katsylo multisection}, regarding it as a multi-valued section for the $S$-Hitchin map defined on an étale cover of $\mathcal A_\mathcal L^0(G;S)$. If $F$ is trivial, then $q$ is an isomorphism and $\varepsilon$ defines a genuine global section to $h_S$.
\end{remark}

The Hitchin-Katsylo multisection gives a global version of Theorem \ref{HiggsNonAbnthm}. In particular, it defines a map $\delta: H^0(\Sigma,\mathfrak{K}_\mathcal L) \times \Sigma \rightarrow [S_{\mathcal L}/G]$ via the evaluation map, and thus defines a group scheme $\mathcal I_{\delta}^{sm} = \delta^*(\mathcal I^{sm}_{S,G} \times^{\mathbb{G}_m}\mathcal L)$ on $H^0(\Sigma,\mathfrak{K}_\mathcal L) \times \Sigma$.

\begin{theorem}\label{GlobalNonabnthm}
  A choice of Katsylo slice $\mathfrak{K}$ for $S$ and square-root $\mathcal L^{1/2}$ of $\mathcal L$ determines a Cartesian diagram
  \begin{equation}\label{GlobalNonabn1eqn}
\begin{tikzcd}
{\bf B}_\Sigma(\mathcal I_{\delta}^{sm}) \arrow[d] \arrow[r] & \mathcal M^0_\mathcal L(G;S) \arrow[d, "h_S"] \\
{H^0(\Sigma,\mathfrak{K}_\mathcal L)} \arrow[r, "q"]         & \mathcal A^0_\mathcal L(G;S)                 
\end{tikzcd}
  \end{equation}
  
  Here, ${\bf B}_\Sigma(\mathcal I_{\delta}^{sm})$ denotes the mapping stack
  \begin{equation}\label{GlobalNonabn2eqn}
      {\bf B}_\Sigma(\mathcal I_{\delta}^{sm}) = \Sec(\Sigma, {\bf B}(\mathcal I_{\delta}^{sm})),
  \end{equation}
  where ${\bf B}(\mathcal I_{\mathcal \delta}^{sm})$ is the classifying stack for the group scheme $\mathcal I_{\mathcal \delta}^{sm}$ defined above.
\end{theorem}

\begin{proof}
The diagram (\ref{GmEquivStackyCB1eqn}) induces the diagram (\ref{GlobalNonabn1eqn}), noting that there is an isomorphism $\Maps(\Sigma,H^0(\Sigma,\mathfrak{K}_\mathcal L)) \cong H^0(\Sigma,\mathfrak{K}_L)$ since $\Sigma$ is projective.
\end{proof}

\subsection{The abelianised fibrations}\label{IntAbFibrnsbn}

We now assume that $S$ is a non-singular Dixmier sheet and that the cameral homomorphism $\kappa_S$ (constructed in Proposition \ref{CameralHomprp}) is smooth, e.g. when $G$ is a classical group by Proposition \ref{CameralSmoothprp}. We assume that $S$ corresponds to some fixed Levi subgroup $L\leq G$, i.e. every semisimple element in $S$ has centraliser conjugate to $L$. We can use the construction of Section \ref{AbQuotsbn} to factor the $S$-Hitchin map through an ``abelianised $S$-Hitchin map'', whose fibres are commutative group stacks which can be described up to isogeny in terms of cameral data.

\begin{definition}\label{AbnSHiggsdef}
    The \emph{stack $\mathcal M^{ab}(G;S)$ of abelianised $S$-valued Higgs bundles on $\Sigma$} is the mapping stack
    \begin{equation}\label{AbnSHiggseqn}
     \mathcal M^{ab}(G;S) = \Maps(\Sigma,[S/G]^{ab}/\mathbb{G}_m)   
    \end{equation}
    where $[S/G]^{ab}$ is the abelianised quotient stack of Definition \ref{AbelianisedStackdef}.
\end{definition}

\begin{remark}\label{AbnSHiggsrmk}
    We will regard the $\mathbb{C}$-points of $\mathcal M^{ab}(G;S)$, i.e., maps $f:\Sigma \rightarrow [S/G]^{ab}/\mathbb{G}_m$, as abelianised Higgs bundles on $\Sigma$; we will interpret these more explicitly in the examples in Section \ref{Hitchinegscn} and Section \ref{NonQSRealFormsbn}, but it is unclear if these have a more natural modular interpretation in general.
\end{remark}

\begin{proposition}\label{AbnSHiggsprp}
    The stack $\mathcal M^{ab}(G;S)$ is a quasi-separated algebraic stack, locally of finite type over $\mathbb{C}$.
\end{proposition}

\begin{proof}
This follows from \cite[Theorem 1.2]{HR} provided we justify that $[S/G]^{ab}$ has affine stabilisers. Fix a field $k$, and a $k$-point $x$ of $[S/G]^{ab}$, and define a $k$-algebraic group $J :=(\rho_S^{ab}\circ x)^*\mathcal J_{S, \mathbb{G}_m}$, where $\mathcal J_{S, \mathbb{G}_m}$ is as in Corollary \ref{GmEquivAbnGerbecor}. The stabiliser group $\Aut(x)$ is an extension of a subgroup of the Katsylo group $F$ by the group $J$. The group $J$ is a closed subgroup of a Weil restriction from a finite flat cover of $\text{Spec}(k)$, so is affine (e.g. by the construction in \cite[Section 7.6, Theorem 4]{BLR}); meanwhile $F$ is a finite group, so any extension of a subgroup of $F$ by $J$ is also affine. So $\Aut(x)$ is affine.
\end{proof}

The factorisation (\ref{AbnGerbe1eqn}) induces a factorisation
\begin{equation}\label{AbnSHitchinFacteqn}
   \begin{tikzcd}
\mathcal M(G;S) \arrow[r, "\Ab_S"] & \mathcal M^{ab}(G;S) \arrow[r, "h_S^{ab}"] & \mathcal A(G;S)
\end{tikzcd} 
\end{equation}
of the $S$-Hitchin map $h_S$. We have the following analogue of Theorem \ref{HiggsNonAbnthm}, which follows from Corollary \ref{GmEquivAbnGerbecor}. Let $\tau$ be a $\mathbb{C}$-point of $\mathcal A(G;S)$ such that the fibre $(h_S^{ab})^{-1}(\tau)$ is non-empty.

\begin{proposition}\label{AbnisedSHitchinFibresprp}
    The fibre $(h_S^{ab})^{-1}(\tau)$ can be identified with the stack ${\bf B}_\Sigma\mathcal J_\tau$ of $\mathcal J_\tau$-torsors on $\Sigma$, where $\mathcal J_\tau = \tau^*\mathcal J_{S, \mathbb{G}_m}$ (for $\mathcal J_{S, \mathbb{G}_m}$ as in Corollary \ref{GmEquivAbnGerbecor}).
\end{proposition}

\begin{remark}\label{AbnisedSHitchinFibresrmk}
    The fibres of $\Ab_S$ over $\mathcal M^{ab}(G;S)$ have a similar description in terms of stacks of non-abelian torsors on $\Sigma$. 
\end{remark}

Since $\mathcal J_S$ is commutative, the fibres of $h_S^{ab}$ inherit the structure of commutative group stacks. We can describe these fibres over the locus $\mathcal A^\heartsuit(G;S)$ (defined as in Proposition \ref{SHitchinBaseNormprp}) up to isogeny - in this context we shall say that a homomorphism of group stacks is an isogeny if it is finite and essentially surjective.

\begin{lemma}\label{PicardStackIsogenylem}
   For any $\tau \in \mathcal A^\heartsuit(G;S)$, the inclusion $\mathcal J_\tau \hookrightarrow \hat{\mathcal J}_\tau$ (where $\hat{\mathcal J}_\tau = \tau^*\hat{\mathcal J}_{S, \mathbb{G}_m}$ for $\hat{\mathcal J}_{S, \mathbb{G}_m}$ as in Corollary \ref{GmEquivAbnGerbecor}) induces an isogeny ${\bf B}_\Sigma \mathcal J_\tau \rightarrow {\bf B}_\Sigma \hat{\mathcal J}_\tau$ of group stacks. 
\end{lemma}

\begin{proof}
The proof is the same as that of \cite[Lemme 4.10.1]{Ngo2}.
\end{proof}

We can describe the stack ${\bf B}_\Sigma\hat{\mathcal J}_\tau$ in terms of bundles over a cover of $\Sigma$; this is the analogue of the cameral curve of \cite{Donagi} in our set-up.

\begin{definition}\label{SCameralcurvedef}
    For a $\mathbb{C}$-point $\tau$ of $\mathcal A(G;S)$, we define the \emph{$S$-cameral curve} to be the finite flat cover $p_\tau:\hat{\Sigma}_{\tau} \rightarrow \Sigma$ determined by the Cartesian diagram:
    \begin{equation}\label{SCameralcurveeqn}
\begin{tikzcd}
\hat{\Sigma}_\tau \arrow[d,"p_\tau"] \arrow[r, "\hat{\tau}"] & {[\mathfrak{z}/\mathbb{G}_m]} \arrow[d, "p_{\mathbb{G}_m}"] \\
\Sigma \arrow[r, "\tau"']       & \mathcal B/\mathbb{G}_m.           
\end{tikzcd}
    \end{equation}
    Here $\mathfrak{z}$ is the centre of the Lie algebra $\Lie(L)$, and $p_{\mathbb{G}_m}$ is the map induced by the $\mathbb{G}_m$-equivariant map $p$ defined in Lemma \ref{StackyZquotientlem}.
\end{definition}
\begin{remark}\label{SCameralcurvermk}
    Since the map $p$ is representable, $\hat{\Sigma}_{\tau}$ is a scheme. It inherits an action of $W_L=N_G(L)/L$ over $\Sigma$ from the $W_L$-action on $\mathfrak{z}$ by Lemma \ref{WLInvariancelem}.
    
    If $\tau$ maps to a $\mathbb{C}$-point $\sigma$ of $\mathcal A(G)$ then we can consider the usual cameral curve $\check{\Sigma}_{\sigma}$ defined by the Cartesian diagram:
    \begin{equation}\label{Cameralcurveeqn}
\begin{tikzcd}
\check{\Sigma}_\sigma \arrow[d] \arrow[r] & {[\mathfrak{t}/\mathbb{G}_m]} \arrow[d] \\
\Sigma \arrow[r, "\sigma"']       & {[\mathfrak{c}/\mathbb{G}_m]}.       
\end{tikzcd}
    \end{equation}
    The diagrams (\ref{SCameralcurveeqn}) and (\ref{Cameralcurveeqn}), together with the inclusion $\mathfrak{z} \hookrightarrow \mathfrak{t}$ and the map $\tilde{\nu}_S:\mathcal B \rightarrow \mathfrak{c}$, induce a map $n_\tau:\hat{\Sigma}_\tau \rightarrow \check{\Sigma}_\sigma$, which is finite since $p_\tau$ is finite. We suppose that $\tau \in \mathcal A^\heartsuit(G;S)$; then $\hat{\Sigma}_\tau$ is reduced (as in \cite[Lemme 4.1.5]{Ngo2}), and so $n_\tau$ factors through the reduced subscheme $\check{\Sigma}_\sigma^{red}$ of $\check{\Sigma}_\sigma$. One can also observe directly from the definitions that $n_\tau:\hat{\Sigma}_\tau \rightarrow \check{\Sigma}_\sigma^{red}$ is injective on the open subset of $\hat{\Sigma}^\tau$ which maps to $[\mathfrak{z}^{rs}/\mathbb{G}_m]$ under $\hat{\tau}$ (where $\mathfrak{z}^{rs}$ is defined as in (\ref{WeylLeviActioneqn})). In particular if $\hat{\Sigma}_\tau$ is non-singular, it is the reduced normalisation of a union of irreducible components of the usual cameral curve.

 As in \cite[Section 4.6]{Ngo2}, one can define a non-empty open substack $\mathcal A^\diamondsuit(G;S)$ of $\mathcal A^\heartsuit(G;S)$ such that the $\mathbb{C}$-points of $\mathcal A^\diamondsuit(G;S)$ correspond to non-singular cameral curves. If the degree of $\mathcal L$ is suitably high, $\mathcal A^\diamondsuit(G;S)$ has non-empty intersection with $\mathcal A_\mathcal L^0(G;S).$
    
    Unlike the usual cameral cover, the cover $p_\tau$ may be unramified; indeed, we can factor the map $p_\tau$ through an $F$-torsor $\tilde{\Sa}$ over $\Sigma$; $\hat{\Sigma}_\tau \rightarrow \tilde{\Sigma}$ is a $W_S$-cover which is locally pulled back from the quotient map $\mathfrak{z}\rightarrow \mathfrak{z}/W_S$ (where $W_S$ is the subgroup of $W_L$ defined in Corollary \ref{KatsyloGpWeylGpcor}).
\end{remark}

Let $\bar{Z} = L/L^{der}$ as in Section \ref{Camgpsbn}, and suppose $X$ is a scheme with a $W_L$-action. We make a definition analogous to \cite[Definition 5.7]{DG} (although we use a different notational convention). For a $\bar{Z}$-torsor $\mathcal Z$ on $X$, we denote by $w^*\mathcal Z$ the pullback of $\mathcal Z$ by the automorphism of $X$ defined by $w$; and we denote by $\mathcal Z^w$ the $\bar{Z}$-torsor with the same underlying scheme as $\mathcal Z$ but with $\bar{Z}$-action given by
\begin{equation}\label{WTwistTorsoreqn}
\begin{tikzcd}
\mathcal Z \times \bar{Z} \arrow[r, "{(\text{id},w^{-1})}"] & \mathcal Z \times \bar{Z} \arrow[r, "Act"] & \mathcal{Z}.
\end{tikzcd}
\end{equation}

\begin{definition}\label{WLEquivTorsordef}
     We say that a $\bar{Z}$-torsor $\mathcal Z$ on $X$ is \emph{(strongly) $W_L$-equivariant} if for every $w \in W_L$ there is an isomorphism $\widetilde{w}:w^*\mathcal Z^w \rightarrow \mathcal Z$ of $\bar{Z}$-torsors such that $\widetilde{\mathrm{id}}_{W_L} = \mathrm{id}_\mathcal Z$ and $\widetilde{w_1w_2} = \widetilde{w}_1 \circ \widetilde{w}_2$ (after making the usual canonical identifications).
\end{definition}

Let $\tau$ be a $\mathbb{C}$-point of $\mathcal A^\heartsuit(G;S)$. The diagonal action of $W_L$ on $\bar{Z} \times \hat{\Sigma}_\tau$ induces a strict $W_L$-action on the stack ${\bf B}_{\hat\Sigma_\tau}\bar{Z}$ of $\bar{Z}$-torsors on $\hat\Sigma_\tau$. The fixed point stack $({\bf B}_{\hat\Sigma_\tau}\bar{Z})^{W_L}$ is the stack of strongly $W_L$-equivariant torsors on $\hat\Sigma_\tau$; morphisms between objects of this stack are given by $W_L$-equivariant isomorphisms of $\bar{Z}$-torsors.

By the definition of $\mathcal A^\heartsuit(G;S)$, there are only finitely many ramification points of $p_\tau$ and these are exactly the points of $\hat\Sigma_\tau$ which have non-trivial stabiliser under the $W_L$-action. For each ramification point $x \in \hat\Sigma_\tau$, with stabiliser $W_{x}=\Stab_{W_L}(x)$, there is a map of stacks $r_x:({\bf B}_{\hat\Sigma_\tau} \bar{Z})^{W_L} \rightarrow (\textbf{B}\bar{Z})^{W_x}$ given by restriction of torsors to the point $x \in \hat\Sigma_\tau$. Thus we can define a map
\begin{equation}\label{WLEquivRestrictioneqn}
    r:=\prod_{x \in R} r_x:({\bf B}_{\hat\Sigma_\tau} \bar{Z})^{W_L} \rightarrow \prod_{x\in R}({\bf B}\bar{Z})^{W_x}
\end{equation}
where $R$ is the set of ramification points of $p_\tau$.

For each $x \in R$, the stack $({\bf B}\bar{Z})^{W_x}$ is a union of copies of ${\bf B}(\bar{Z}^{W_x})$ indexed by the group cohomology $H^1(W_x;\bar{Z})$; in particular, there is a component ${\bf B}^0(\bar{Z}^{W_x})$ whose image consists of the substack of $W_x$-equivariant $\bar{Z}$-torsors which are fppf-locally isomorphic to the trivial torsor with its natural equivariant structure.

\begin{proposition}\label{SCameralBunprp}
    For $\tau \in \mathcal A^\heartsuit(G;S)$, the stack ${\bf B}_{\hat\Sigma}\hat{\mathcal J}_\tau$ is isomorphic to the fibre product 
    \begin{equation}\label{SCameralBun1eqn}
        ({\bf B}_{\hat\Sigma_\tau} \bar{Z})^{W_L} \times_r \left(\prod_{x\in R} {\bf B}^0(\bar{Z}^{W_x})\right).
    \end{equation}
\end{proposition}

\begin{proof}
The proof is analogous to that of \cite[Proposition 16.4]{DG}.
\end{proof}

We can identify the stack (\ref{SCameralBun1eqn}) with the stack $\hat{\mathcal P}_{S,\tau}$ of $W_L$-equivariant $\bar{Z}$-torsors $\mathcal Z$ on $\hat\Sigma_\tau$, satsifying the following additional condition:
\begin{itemize}
    \item[$(*)$] for each ramification point $x \in R$, there is an isomorphism $\mathcal Z_x\cong \bar{Z}$ which is equivariant with respect to the actions of both $\bar{Z}$ and $W_x$.
\end{itemize}
Then Lemma \ref{PicardStackIsogenylem} and Proposition \ref{SCameralBunprp} combine to give a cameral description for the fibres of $h_S^{ab}$.

\begin{theorem}\label{SCameralDatathm}
    For any $\tau \in \mathcal A^\heartsuit(G;S)$ such that $(h_S^{ab})^{-1}(\tau)$ is non-empty, there is a finite essentially surjective map $(h_S^{ab})^{-1}(\tau) \rightarrow \hat{\mathcal P}_{S,\tau}$.
\end{theorem}

\begin{remark}\label{SCameralBunrmk}
    In the case of regular Higgs bundles this map determines the variant of the homogenised cameral data of \cite[16.7]{DG}, which does not include the data of the trivialisations $\beta_{i,0}$ (in the notation of \cite[16.3]{DG}). This latter discrepancy is due to the difference in general between $\mathcal J_S$ and $\hat{\mathcal J_S}$, e.g. see Remark \ref{SmallCameralgprmk}.
\end{remark}

We conclude this section by giving the abelian analogues of Theorem \ref{GlobalNonabnthm}. We let $\mathcal L$ be a line bundle on $\Sigma$ which admits a square root, and we denote by $\mathcal M_\mathcal L^{ab,0}(G;S)$ the restriction of $\mathcal M_\mathcal L^{ab}(G;S)$ to $\mathcal A_\mathcal L^0(G;S)$. The pullback $v^*\mathcal J_{S, \mathbb{G}_m}$ under the evaluation map $v:\mathcal A_\mathcal L^0(G;S) \times \Sigma \rightarrow \mathcal B/\mathbb{G}_m$ defines a smooth commutative group stack on $\mathcal A_\mathcal L^0(G;S)$ representable by group schemes. We set
\begin{equation}\label{SPicardeqn}
    \mathcal P^0_S = \Sec(\Sigma, {\bf B}(v^*\mathcal J_{S,\mathbb{G}_m}))
\end{equation}
where ${\bf B}(v^*\mathcal J_{S,\mathbb{G}_m})$ is the classifying stack of $v^*\mathcal J_{S,\mathbb{G}_m}$-torsors; $\mathcal P_S^0$ is a commutative group stack over $\mathcal A_\mathcal L^0(G;S)$.

Let $\tau$ be a $\mathbb{C}$-point of $\mathcal A_\mathcal L^0(G;S)$. For any abelianised Higgs bundle $f:\Sigma \rightarrow [S/G]^{ab}/\mathbb{G}_m$ representing a point in the fibre $(h_S^{ab})^{-1}(\tau)$, its group scheme of local automorphisms over $\Sigma$ is $\Aut_\Sigma(f) =\mathcal J_\tau$. Thus, the fibre $\mathcal P^0_{S,\tau} = {\bf B}_\Sigma \mathcal J_\tau$ of $\mathcal P^0_S$ over $\tau$ acts canonically on $(h_S^{ab})^{-1}(\tau)$ by twisting by $\mathcal J_\tau$-torsors. This defines a global action of the group stack $\mathcal P_S^0$ on $\mathcal M_\mathcal L^{ab,0}(G;S)$ over $\mathcal A_\mathcal L^0(G;S)$.

\begin{theorem}\label{SPicardTorsthm}
    The stack $\mathcal M_\mathcal L^{ab,0}(G;S)$ is a torsor for the action of $\mathcal P_S^0$.
\end{theorem}

\begin{proof}
The proof is the same as that of \cite[Proposition 4.3.3]{Ngo2}, noting that the fibres of $h_S^{ab}$ are non-empty over $\mathcal A^0(G;S)$ by Proposition \ref{HKMSectionprp}.
\end{proof}

Given a choice of Katsylo slice $\mathfrak{K}$ for $S$, this torsor structure can be trivialised over $H^0(\Sigma,\mathfrak{K}_\mathcal L)$. We have a map $\eta:H^0(\Sigma,\mathfrak{K}_\mathcal L) \times \Sigma \rightarrow \mathcal B/\mathbb{G}_m$ induced by the evaluation map, and we can define a group scheme $\mathcal J_\eta = \eta^*\mathcal J_{S, \mathbb{G}_m}$ on $H^0(\Sigma,\mathfrak{K}_\mathcal L) \times \Sigma$.

\begin{theorem}\label{GlobalSHiggsAbnthm}
    A choice of Katsylo slice $\mathfrak{K}$ for $S$ and square-root $\mathcal L^{1/2}$ of $\mathcal L$ determines a Cartesian diagram
  \begin{equation}\label{GlobalSHiggsAbn1eqn}
\begin{tikzcd}
{\bf B}_\Sigma(\mathcal J_{\eta}) \arrow[d] \arrow[r] & \mathcal M^{ab,0}_\mathcal L(G;S) \arrow[d, "h_S^{ab}"] \\
{H^0(\Sigma,\mathfrak{K}_\mathcal L)} \arrow[r, "q"]         & \mathcal A^0_\mathcal L(G;S)                 
\end{tikzcd}
  \end{equation}
  
  Here, ${\bf B}_\Sigma(\mathcal J_{\eta})$ denotes the mapping stack
  \begin{equation}\label{GlobalSHiggsAbn2eqn}
      {\bf B}_\Sigma(\mathcal J_{\eta}) = \Sec(\Sigma, {\bf B}(\mathcal J_{\eta})),
  \end{equation}
  which is a commutative group stack over $H^0(\Sigma,\mathfrak{K}_\mathcal L)$.
\end{theorem}

\begin{remark}\label{AbSurjectivermk}
    Theorems \ref{GlobalNonabnthm} and \ref{GlobalSHiggsAbnthm} ensure that the fibres of $h_S$ and $h_S^{ab}$ are always non-empty over $\mathcal A_\mathcal L^0(G;S)$. On the other hand, there is no guarantee that the map $\Ab_S:\mathcal M^0_\mathcal L(G;S) \rightarrow \mathcal M^{ab, 0}_\mathcal L(G;S)$ is necessarily essentially surjective (although it is in the examples we calculate in Sections \ref{Hitchinegscn} and \ref{NonQSRealFormsbn}).
\end{remark}

\section{Examples}\label{Hitchinegscn}
We make explicit the constructions of the previous section for the two main examples of sheets which were discussed in Section \ref{Sheetsegsbn}, i.e., the sheets in $\mathfrak{gl}_n$ for $n$ arbitrary, and the sheets in $\mathfrak{sp}_4$. We consider these using the spectral correspondence of \cite{Hitchin2} and \cite{BNR} (and its extensions to non-integral curves by \cite{Schaub} and \cite{Carbone}). The abelianised Hitchin fibrations in these cases involve spectral data on the normalisation of the reduced subscheme of the spectral curve, providing a non-abelian analogue to similar constructions for ``semi-abelian" spectral data in the singular locus of the Hitchin fibration \cite{Hitchin4}, \cite{Horn1}, \cite{Horn2}, \cite{FraPN}.

For concreteness, we fix the twisting line bundle $\mathcal L$ on $\Sigma$ to be the canonical bundle $K$.

\subsection{Sheet-valued \texorpdfstring{$\GL_n$}{GLn}-Higgs bundles and spectral data}\label{GLnHitchinsbn}

We consider the case where $G = \GL_n$, giving the details of the results sketched in Section \ref{IntroGLnscn}. We use the notation of Example \ref{GLnSheetexm}. 

We recall that every sheet $S$ in $\mathfrak{gl}_n$ is Dixmier, i.e. contains a semisimple element $x \in S$, and thus the sheets are in bijective correspondence with conjugacy classes of Levi subgroups $L \leq \GL_n$ (for example by taking $L=C_G(x)$); these in turn correspond to partitions ${\bf m} = (m_1 \geq ...\geq m_r)$ of $n$. We will denote by $S_{\bf m}$ the sheet corresponding to the partition ${\bf m}$, and $L_{\bf m}$ its associated Levi subgroup.

For a sheet $S_{\bf m}$, we can characterise the $S_{\bf m}$-valued Higgs bundles as follows. 

\begin{proposition}\label{SmValuedHiggsGLnprp}
    For a Higgs bundle $(V,\Phi)$ on $\Sigma$, the following are equivalent:
    \begin{itemize}
        \item[(i)] $(V,\Phi)$ is an $S_{\bf m}$-valued Higgs bundle;
        \item[(ii)] for every $x \in \Sigma$, $\Phi_x$ can be identified with an endomorphism of $\mathbb{C}^n$ with Jordan decomposition $\Phi_x = \Phi_x^{ss}+\Phi_x^{nil}$, such that $\Phi_x^{ss}$ decomposes as
        \begin{equation}\label{SmValuedHiggsGLn1eqn}
            \Phi_x^{ss} = \oplus_{i=1}^r\lambda_iI_{m_i} \in \bigoplus_{i=1}^r \End(\mathbb{C}^{m_i})
        \end{equation}
        for some scalars $\lambda_i \in \mathbb{C}$, and $\Phi_x^{nil}$ satisfies 
        \begin{equation}\label{SmValuedHiggsGln2eqn}
            \Phi_x^{nil} \in \bigoplus_{1\leq i<j \leq r}\Hom(\mathbb{C}^{m_j},\mathbb{C}^{m_i})
        \end{equation}
        with minimal possible centraliser dimension subject to this condition.

    \end{itemize}
\end{proposition}

\begin{proof}
This follows from \cite[Satz 4.8]{Borho2} and the definition of the induced orbit in \cite{LS}.
\end{proof}

This description is somewhat unwieldy, but it is substantially more straightforward to describe the $S_{\bf m}$-Hitchin base of Definition \ref{SHitchinmapdef}. We recall that the Katsylo group $F$, defined in Definition \ref{KatsyloGpdef}, is always trivial for any sheet in $G=\GL_n$. We use the notation of Example \ref{GLnSheetexm}. By Remark \ref{SHitchinBasePropsrmk}, we have
\begin{equation}\label{SmHitchinBaseGLneqn}
\begin{aligned}
    \mathcal A_K(\GL_n;S_{\bf m}) &= \bigoplus_{i=1}^sH^0(\Sigma,\mathfrak{z}_i \otimes  K/W_i) 
    \\
    &= \bigoplus_{i=1}^s\left(\bigoplus_{j=1}^{l_i}H^0(\Sigma,K^j) \right),
\end{aligned}
\end{equation}
noting that $\mathfrak{z}_i$ is a vector space of dimension $l_i$ and $W_i$ is the symmetric group on $l_i$ elements acting by permutation of the coordinates of $\mathfrak{z}_i$. For any semisimple element in $x \in S_{\textbf{m}}$, $l_i$ is the number of distinct eigenvalues of $x$ with multiplicity $i$.

We recall the interpretation of the usual Hitchin base $\mathcal A_K(\GL_n)$ as a space of spectral covers over $\Sigma$ \cite{Hitchin2}. If $\pi:T^*\Sigma\rightarrow \Sigma$ is the total space of the line bundle $K$ on $\Sigma$, a point $(a_1,...,a_n)$ of
\begin{equation}\label{HitchinBaseGLneqn}
\begin{aligned}
    \mathcal A_K(\GL_n) &= \bigoplus_{i=1}^nH^0(\Sigma,K^i)
\end{aligned}
\end{equation}
corresponds to a section $\xi$ of $\pi^*K^n$ over $T^*\Sigma$ of the form $$\xi=\lambda^n+a_1\lambda^{n-1}+ ...+a_n,$$
where $\lambda$ is the canonical section of $\pi^*K$ over $T^*\Sigma$. This section in turn corresponds to a cover ${\Sigma}_\xi \rightarrow \Sigma$ together with a closed immersion ${\Sigma}_\xi \hookrightarrow T^*\Sigma$; this is the \emph{spectral curve} corresponding to $\xi$.

The points $((b_i)_j)$ of $\mathcal A_K(\GL_n;S_{\bf m})$ correspond analogously to tuples $(\xi_1,...,\xi_s)$, where $\xi_i$ is a section of $\pi^*K^{l_i}$ defined by
$$\xi_i=\lambda^{l_i}+(b_i)_1\lambda^{l_i-1}+ ...+(b_i)_{l_i},$$
and thus to tuples $({\Sigma}_{\xi_i} \rightarrow \Sigma)_i$ of covers together with their closed immersions ${\Sigma}_{\xi_i} \hookrightarrow T^*\Sigma$. We could simply omit any $i$ with $l_i=0$, but it is notationally convenient to instead take $\xi_i=1$ so that ${\Sigma}_{\xi_i}$ is the empty scheme. The points of the open subscheme $\mathcal A_K^\heartsuit(\GL_n;S_{\bf m})$ of $\mathcal A_K(\GL_n;S_{\bf m})$ defined as in Proposition \ref{SHitchinBaseNormprp} correspond exactly to the tuples $({\Sigma}_{\xi_i} \rightarrow \Sigma)_i$ such that each ${\Sigma}_{\xi_i}$ is reduced and no pair of curves $\Sigma_{\xi_i}$ and $\Sigma_{\xi_j}$ share an irreducible component as subvarieties of $T^*\Sa$.

\begin{proposition}\label{GLnHitchinBaseNormprp}
    The map $\tilde{\mu}_S:\mathcal A_K(\GL_n;S_{\bf m}) \rightarrow \mathcal A_K(\GL_n)$ in Remark \ref{SHitchinBasermk} sends a tuple $(\xi_1,...,\xi_s)$ to the section
    \begin{equation}\label{GLnHitchinBaseNorm1eqn}
        \xi=\prod_{i=1}^s(\xi_i)^i
    \end{equation}
    of $\pi^*K^n$.
\end{proposition}

\begin{proof}
The map $\tilde{\mu}_S$ is induced by the map
$$\nu_L:\prod_{i=1}^s\mathfrak{z}_i/W_i \rightarrow \mathfrak{t}/W.$$ As in Example \ref{GLnSheetexm}, points of the left hand side are of the form $(\boldsymbol{\lambda}_1,\, ...,\, \boldsymbol{\lambda}_s)$, where each $\boldsymbol{\lambda}_i$ is an unordered $l_i$-tuple $((\lambda_i)_1,\, ..., \, (\lambda_i)_{l_i})$. The map $\nu_L$ sends $(\boldsymbol{\lambda}_1,\,...,\,\boldsymbol{\lambda}_s)$ to the unordered $n$-tuple of all of the $(\lambda_i)_j$, where each $(\lambda_i)_j$ appears $i$ times (for each occurence of this value in $\boldsymbol{\lambda}_i$). If a section $(s_i)_x$ corresponds to the tuple $((\lambda_i)_1, \, ..., \, (\lambda_i)_{l_j})$, then
$$(s_i)_x = (\lambda_x - (\lambda_i)_1)\,\cdots\,(\lambda_x -(\lambda_i)_{l_i});$$
a similar expression holds for $s_x$. Thus the statement of the proposition follows.
\end{proof}

\begin{remark}\label{GLnHitchinBaseNormrmk}
    For a tuple $(\xi_1, ...,\xi_s)$ corresponding to a point of $\mathcal A^\heartsuit_K(\GL_n;S_{\bf m})$, each of the $\xi_i$ have distinct irreducible factors occuring with multiplicity one; hence the expression (\ref{GLnHitchinBaseNorm1eqn}) is unique, i.e. the map $\tilde{\mu}_S$ in Remark \ref{SHitchinBasermk} is injective on this open subset as predicted by Proposition \ref{SHitchinBaseNormprp}. On the other hand, if we take $n=4$ and ${\bf m} = (2, 1^2)$, then for any $a \in H^0(\Sigma,K^2)$, the points in $\mathcal A(\GL_4;S_{\bf m})$ corresponding to $((\lambda-a)^2,(\lambda+a))$ and $((\lambda+a)^2,(\lambda-a))$ both map to the same section $(\lambda-a)^2(\lambda+a)^2$ in $\mathcal A(\GL_4)$. This is the global version of the failure of injectivity of $\nu_L$ noted at the end of Example \ref{GLnSheetexm}.
\end{remark}

We recall that a Higgs bundle $(V,\Phi)$ maps to $\xi(\lambda;\Phi)$ under the Hitchin map $h_{\GL_n}$, where $\xi(T;\Phi) = \det(T -\Phi)$ is the characteristic polynomial of $\Phi$ whose coefficients are given by sections of powers of $\pi^*K$. In particular, for an $S_{\bf m}$-valued Higgs bundle $(V,\Phi)$, the characteristic equation $\xi(T;\Phi)$ has a decomposition
\begin{equation}\label{GLnHitchinNorm2eqn}
    \xi(T;\Phi) = \prod_{i=1}^s\xi_i(T;\Phi)^i
\end{equation}
where each $\xi_i(T;\Phi)$ is a polynomial of degree $l_i$; this decomposition is unique if $(V,\Phi)$ is a point of $\mathcal M_K^\heartsuit(\GL_n;S_{\bf m})$. In that case, by taking each factor $\xi_i(T;\Phi)$ only once, we can define a polynomial by
\begin{equation}\label{MinPolyHiggsFibreeqn}
    m(T;\Phi) = \prod_{i=1}^s\xi_i(T;\Phi).
\end{equation}

\begin{lemma}\label{MinPolyHiggsFibrelem}
    For a $\mathbb{C}$-point $(V,\Phi)$ of $\mathcal M_K^\heartsuit(\GL_n;S_{\bf m})$, the $K^r$-twisted endomorphism $m=m(\Phi;\Phi)$ of $V$ is $0$.
\end{lemma}

\begin{proof}
The condition $m_x=0$ is a closed condition on $x\in \Sigma$, thus it suffices to check this on the open subset $\Sigma^{ss}$ over which $\Phi_x$ is semisimple; this is clear from Proposition \ref{SmValuedHiggsGLnprp}.
\end{proof}

We recall the spectral correspondence in the framework of \cite{BNR}.

\begin{theorem}\label{SpectralDatathm}\cite{Hitchin2}, \cite{BNR}, \cite{Schaub}, \cite{deCataldo}
There is a functorial correspondence between rank 1 torsion-free sheaves on ${\Sigma}_{\xi(\lambda;\Phi)}$ and Higgs bundles on $\Sigma$ with characteristic polynomial $\xi(T;\Phi)$. The correspondence sends a sheaf $\mathcal E$ on ${\Sigma}_{\xi(\lambda;\Phi)}$ (regarded as a compactly supported sheaf on $T^*\Sigma$) to the Higgs bundle $(V,\Phi) = (\pi_*\mathcal E,\pi_*\lambda)$, where $\pi_*\lambda$ denotes the pushforward of the ``multiplication by $\lambda$" map determined by the $\mathcal O_{T^*\Sigma}$-structure on $\mathcal E$.
\end{theorem}

\begin{remark}\label{SpectralDatarmk}
    For our purposes, a sheaf $\mathcal E$ on ${\Sigma}_\xi$ has \emph{rank 1} if whenever $x\in {\Sigma}_\xi$ is the generic point of an irreducible component of $\Sigma_\xi$, $\mathcal E_x$ has the same length as $\mathcal O_{{\Sigma}_\xi,x}$ as an $\mathcal O_{{\Sigma}_\xi,x}$-module.
\end{remark}

There are a particular class of rank 1 sheaves on the spectral curve corresponding to $S_{\bf m}$-valued Higgs bundles. We make the following definition.

\begin{definition}\label{ReducedSheafdef}
    Let $X$ be a scheme, and $\mathcal E$ be a sheaf on $X$. We say $\mathcal E$ is \emph{reduced} if there is a sheaf $\mathcal E_{X^r}$ on the reduced subscheme $\iota^r:X^r \hookrightarrow X$ such that $\mathcal E = \iota^r_*\mathcal E_{X^r}$.
\end{definition}

\begin{proposition}\label{GLnSmSpectralDataprp}
    Let $\xi \in \mathcal A_K(\GL_n)$ be a point in the image of $\mathcal A^\heartsuit_K(\GL_n;S_{\bf m})$ under $\tilde{\mu}_{S_{\bf m}}$. The sheaves $\mathcal E$ on ${\Sigma}_\xi$ corresponding to $S_{\bf m}$-valued Higgs bundles under Theorem \ref{SpectralDatathm} are reduced. 

    Conversely, any reduced rank 1 torsion-free sheaf $\mathcal E$ corresponds to a generically $S_{\bf m}$-valued Higgs bundle $(V,\Phi)$. 
\end{proposition}

\begin{proof} 
Let $(\xi_1, \, ..., \, \xi_s)$ be the point in $\mathcal A_K^\heartsuit(\GL_n;S_{\bf m})$ such that $\xi = \tilde{\mu}_S(\xi_1, \, ..., \, \xi_s)$. Viewing the spectral curve as the vanishing of the ideal sheaf on $T^*\Sigma$ generated by $\xi = \xi_1(\xi_2)^2...(\xi_s)^s$, we see that the reduced subscheme of ${\Sigma}_\xi$ is defined by the vanishing of the ideal sheaf generated by $m = \xi_1\xi_2...\xi_s$. Hence, a sheaf $\mathcal E$ on ${\Sigma}_\xi$ being reduced corresponds exactly to it being annihilated under multiplication by $m$. Hence the first statement follows by Lemma \ref{MinPolyHiggsFibrelem}.

The second statement follows by observing that, away from the ramification of the cover ${\Sigma}^r_\xi \rightarrow \Sigma$ defined by the reduced subscheme ${\Sigma}^r_\xi \subseteq\hat{\Sigma}_\xi$, the reducedness condition on the sheaf $\mathcal E$ corresponds to the Higgs field $\Phi$ being diagonalisable, with eigenvalues whose multiplicities are prescribed by ${\bf m}$.
\end{proof}

\begin{remark}\label{GlnSmSpectralDatarmk}
   Suppose $l_i \neq 0$ for only one value of $i$, i.e. $\xi = (\xi_i)^i$, and suppose the reduced subscheme of ${\Sigma}_\xi$ is non-singular. In this case, every reduced rank 1 torsion-free sheaf on ${\Sigma}_\xi$ corresponds to an $S_{\bf m}$-valued Higgs bundle, and Proposition \ref{GLnSmSpectralDataprp} gives the generalised spectral correspondence of \cite[Theorem 6.7]{BR}.

    In general, choosing a rank $i$ vector bundle $\mathcal E_i$ on each ${\Sigma}_{\xi_i}$ determines a sheaf $$\mathcal E = \iota^r_*\bigoplus_{i=1}^s\mathcal E_i,$$ which is rank 1, torsion-free and reduced. But if at least two of the values of $l_i$ are non-zero, the resulting Higgs bundle under the spectral correspondence is not an $S_{\bf m}$-valued Higgs bundle.
\end{remark}

We note an interesting corollary to Proposition \ref{GLnSmSpectralDataprp}. We consider the subscheme $\mathcal A^\diamondsuit_K(\GL_n;S_{\bf m})$ of $\mathcal A_K(\GL_n;S_{\bf m})$ whose points correspond to tuples $({\Sigma}_{\xi_i} \rightarrow \Sigma)_i$ such that each ${\Sigma}_{\xi_i}$ is non-singular and the curves $\Sigma_{\xi_i}$ have pairwise transverse intersections inside $T^*\Sigma$; this is a non-empty open subscheme by a Bertini-type argument as in \cite[Section 5.1]{Hitchin2}.

\begin{corollary}\label{GLnSHFibredimcor}
    The $S_{\bf m}$-Hitchin fibre $h_{S_{\bf m}}^{-1}(\boldsymbol{\xi})$ has the same dimension as the usual Hitchin fibre $h^{-1}(\tilde{\mu}_{S_{\bf m}}(\boldsymbol{\xi}))$, i.e. $\dim(h_{S_{\bf m}}^{-1}(\boldsymbol{\xi})) = n^2(g-1)$.
\end{corollary}

\begin{proof}
    For the dimension of the usual Hitchin fibre, see \cite[Proposition 3.12 (ii)]{dCFFHM}.

    If $l_i \neq 0$ for only one value of $i$ as in Remark \ref{GlnSmSpectralDatarmk}, i.e. $\xi = \tilde{\mu}_S(\boldsymbol{\xi}) = (\xi_i)^i$, then as in \cite[Theorem 6.7]{BR}, reduced rank 1 torsion free sheaves on $\Sigma_{\xi}$ correspond to rank $i$ vector bundles on $\Sigma_{\xi_i}$. By the calculations of \cite[5.1]{Hitchin2}, $\Sigma_{\xi_i}$ has genus $(l_i)^2(g-1)+1$, so the stack of rank $i$ vector bundles on $\Sigma_{\xi_i}$ has dimension $n^2(g-1)$ as required.

    In general, a reduced rank 1 torsion free sheaf $\mathcal E$ on $\Sigma_{\xi}$ has a filtration
    \begin{equation}\label{ReducedSheafFiltreqn}
        0 \subseteq \mathcal E^1 \subseteq ... \subseteq \mathcal E^s = \mathcal E
    \end{equation}
    where $\mathcal E^i = \mbox{Ker}(\xi_1\xi_2 ...\xi_i)$. Each quotient $\mathcal F^i =\mathcal E^i/\mathcal E^{i-1}$ defines a rank $i$ vector bundle on $\Sigma_{\xi_i}$; and to reconstruct the sheaf $\mathcal E$ from these vector bundles, it suffices to describe the $\mathcal O_{\Sigma_{\xi}}$-module structure about each pairwise intersection point between the irreducible components. This is determined by a collection of linear maps $\phi^{ij}_x:\mathcal F^j_x \rightarrow \mathcal F^i_x$ for every pair $i < j$ and every intersection point $x \in \Sigma_{\xi_i} \cap \Sigma_{\xi_j}$. 
    
    For each pair $i < j$, there are $2l_il_j(g-1)$ intersection points (since the resultant of the polynomials $\xi_i$ and $\xi_j$ is a section of $K^{l_il_j}$), and for a given intersection point the space of linear maps $\phi^{ij}_x$ has dimension $ij$. The condition for a reduced sheaf to correspond to an $S_{\bf m}$-Higgs bundle under Theorem \ref{SpectralDatathm} is a non-empty open condition on each of the spaces of linear maps; thus, the dimension of $h_S^{-1}(\boldsymbol{\xi})$ is
    \begin{equation}\label{GLnSHFibredimeqn}
        \sum_{i=1}^si^2(l_i)^2(g-1) +\sum_{1\leq i < j\leq s} 2(ij)(l_il_j)(g-1) = (g-1)\left( \sum_{i=1}^sil_i \right)^2 = n^2(g-1).
    \end{equation}
\end{proof}

We will now describe the abelianised fibration $h_{S_{\bf m}}^{ab}$ as a Hitchin fibration; we use the notation of Section \ref{IntAbFibrnsbn}. We can describe the map explicitly over $\mathcal A^\diamondsuit_K(\GL_n;S_{\bf m})$; We denote by $\mathcal M^\diamondsuit_K(\GL_n;S_{\bf m})$ the restriction of $\mathcal M_K(\GL_n;S_{\bf m})$ to $\mathcal A^\diamondsuit_K(\GL_n;S_{\bf m})$ under $h_{S_{\bf m}}$. 

We construct a map
\begin{equation}\label{GLnAbHitchin1eqn}
    \mathcal M^\diamondsuit_K(\GL_n;S_{\bf m})\rightarrow \prod_{i=1}^s\mathcal M_K^\diamondsuit(\GL_{l_i}).
\end{equation}
 Suppose $(V,\Phi)$ is a $\mathbb{C}$-point of $\mathcal M^\diamondsuit_K(\GL_n;S_{\bf m})$, which maps to $(\xi_1,\xi_2,...,\xi_s)$ under $h_{S_{\bf m}}$. The spectral curve ${\Sigma}_\xi \rightarrow \Sigma$ for $(V,\Phi)$ has reduced subscheme ${\Sigma}_\xi^r$ with irreducible components given by ${\Sigma}_{\xi_i}$. The sheaf $\mathcal E$ on ${\Sigma}_\xi$ corresponding to $(V,\Phi)$ under Theorem \ref{SpectralDatathm} defines rank $i$ vector bundles $\mathcal E_i$ on each ${\Sigma}_{\xi_i}$ via the kernel filtration of (\ref{ReducedSheafFiltreqn}). The determinants $\wedge^i \mathcal E_i$ thus define line bundles on each ${\Sigma}_{\xi_i}$, each of which can be viewed as the spectral cover of $\Sigma$ for the point $\xi_i \in \mathcal A^\diamondsuit_K(\GL_{l_i})$. Alternatively, we can consider the collection $(\wedge^i\mathcal E_i)$ as defining a sheaf $\mathcal F$ on the normalisation $\tilde{\Sigma}_\xi$ of ${\Sigma}_\xi^r$.

Under Theorem \ref{SpectralDatathm}, each $\wedge^i\mathcal E_i$ corresponds to a $\GL_{l_i}$-Higgs bundle $(V_i^{ab},\Phi_i^{ab})$, and this defines an assignment
\begin{equation}\label{GLnAbHitchin2eqn}
    (V,\Phi) \mapsto ((V_1^{ab},\Phi_1^{ab}),...,(V_s^{ab},\Phi_s^{ab}))
\end{equation}
which constructs (\ref{GLnAbHitchin1eqn}) at the level of $\mathbb{C}$-points.

\begin{proposition}\label{GLnAbHitchinprp}
    The assignment (\ref{GLnAbHitchin2eqn}) defines a morphism which extends to a map
    \begin{equation}\label{GLnAbHitchinFulleqn}
        \mathcal M_K(\GL_n;S_{\bf m})\rightarrow \prod_{i=1}^s\mathcal M_K^{reg}(\GL_{l_i})
    \end{equation}
    realising $\Ab_{S_{\bf m}}$ and identifying $h_{S_{\bf m}}^{ab}$ with the product of the Hitchin maps $h_{\GL_{l_i}}$.
\end{proposition}

\begin{proof}
Theorem \ref{GlobalSHiggsAbnthm} implies that $\mathcal M^{ab}_K(\GL_n;S_{\bf m})$ is isomorphic over $\mathcal A_K(\GL_n;S_{\bf m})$ to the commutative group stack ${\bf B}_\Sigma(\mathcal J_\eta)$, where $\mathcal J_\eta$ is the group scheme over $\mathcal A_K(\GL_n;S_{\bf m}) \times \Sigma$ which is pulled back from the representable group stack $\mathcal J_{S, \mathbb{G}_m}$ on $[\mathfrak{c}_{S_{\bf m}}/\mathbb{G}_m]$.

There is a decomposition of the $W_L$-action on $\mathfrak{z}$ into the permutation actions of $W_i = \Sym_{l_i}$ on $\mathfrak{z}_i = \mathbb{C}^{l_i}$ and a corresponding decomposition of the $W_L$-action on $\bar{Z} = (\mathbb{G}_m)^r$ into permutation actions of $\Sym_{l_i}$ on $\bar{Z}_i := (\mathbb{G}_m)^{l_i}$. If we denote the quotient maps by $\pi_i:\mathfrak{z}_i \rightarrow \mathfrak{z}_i/W_i$, we can define smooth group schemes $\mathcal J^i = (\pi_i)_*(\bar{Z}_i\times\mathfrak{z}_i)^{W_i}$, which descend to representable group stacks $\mathcal J^i_{\mathbb{G}_m}$ on each $[\mathfrak{c}_i/\mathbb{G}_m]$. The map $\eta$ decomposes into maps $\eta_i:\mathcal A_K(\GL_{l_i}) \times \Sigma \rightarrow [\mathfrak{c}_i/\mathbb{G}_m]$, and there is a decomposition $$\mathcal J_\eta = \prod_{i=1}^s\mathcal J^i_{\eta_i},$$ where $\mathcal J^i_{\eta_i}$ is the pullback of $\mathcal J^i_{\mathbb{G}_m}$ under $\eta_i$. The group schemes $\mathcal J^i_{\mathbb{G}_m}$ can also be regarded as the cameral groups for the regular sheets for $\GL_{l_i}$, and so by Theorem \ref{GlobalSHiggsAbnthm} (which in this case follows from \cite[Theorem 4.4]{DG}), ${\bf B}_\Sigma(\mathcal J_\eta)$ is isomorphic over $$\mathcal A_K(\GL_n;S_{\bf m}) = \prod_{i=1}^s \mathcal A_K(\GL_{l_i})$$ to $\prod_{i=1}^s\mathcal M_K^{reg}(\GL_{l_i})$.

Thus, abstractly, we can identify $\Ab_{S_{\bf m}}$ with a map (\ref{GLnAbHitchinFulleqn}) such that $h_{S_{\bf m}}^{ab}$ is identified with the product of the Hitchin maps. We need only check that on $\mathbb{C}$-points this map agrees with the assignment (\ref{GLnAbHitchin2eqn}). 

We first fix an identification of $\mathcal M^{ab}_K(\GL_n;S_{\bf m})$ with ${\bf B}_{\Sigma}(\mathcal J_\eta)$ via a choice of Hitchin-Katsylo section $\varepsilon:\mathcal A_K(\GL_n;S_{\bf m}) \rightarrow \mathcal M_K(\GL_n;S_{\bf m})$. We denote by $(V_{\boldsymbol{\xi}},\Phi_{\boldsymbol{\xi}})$ the image of a point $\boldsymbol{\xi} \in \mathcal A_K(\GL_n;S_{\bf m})$ under $\varepsilon$.

Fix an $S_{\bf m}$-valued Higgs bundle $(V,\Phi)$ mapping to $\boldsymbol{\xi} =(\xi_1,...,\xi_s)\in \mathcal A_K(\GL_n;S_{\bf m})$. Let $\hat\Sigma_{\boldsymbol{\xi}} \rightarrow \Sigma$ be the $S_{\bf m}$-cameral cover of $\Sigma$ as defined in Definition \ref{SCameralcurvedef}; this decomposes as 
\begin{equation}\label{GLnAbHitchin3eqn}
    \hat\Sigma_{\boldsymbol{\xi}}=\check\Sigma_{\xi_1} \times_\Sigma \, ... \, \times_\Sigma \check\Sigma_{\xi_s}
\end{equation}
where $\check\Sigma_{\xi_i} \rightarrow \Sigma$ is the usual $\GL_{l_i}$-cameral cover. Let $P$ be a parabolic subgroup of $\GL_n$ with Levi factor $L$ corresponding to ${\bf m}$, and let $\mathfrak{r}$ be the solvable radical of $\mathfrak{p}=\Lie(P)$. The map $\Ab_{S_{\bf m}}$ is defined as follows. By Proposition \ref{StackyGSprp}, the pullback of $(V,\Phi)$ to the $S_{\bf m}$-cameral curve defines a reduction of the structure group $V_P$ of the bundle $V$ to $P$ such that the Higgs field is valued in $\mathfrak{r}^{reg}$. The map $\Ab_{S_{\bf m}}$ outputs the $\bar{Z}$-bundle $V_{\bar{Z}}$ induced by the abelianisation of $P$. This does not have a $W_L$-equivariant structure, but the twist $V_{\bar{Z}} \otimes (V_{\boldsymbol{\xi}})_{\bar{Z}}^{-1}$ does. The resulting assignment $(V,\Phi) \mapsto V_{\bar{Z}} \otimes (V_{\boldsymbol{\xi}})_{\bar{Z}}^{-1}$ determines the abelianisation map fibrewise over the Hitchin base as a map onto a substack of the stack of $W_L$-equivariant $\bar{Z}$-torsors on $\hat{\Sigma}_{\boldsymbol{\xi}}$ as in Theorem \ref{SCameralDatathm}. However it will be more convenient to work with the map $(V,\Phi) \mapsto V_{\bar{Z}}$; this causes no issues since the images of both maps inside the relative stack over $\mathcal A_K(\GL_n;S_{\bf m})$ of $\bar{Z}$-torsors on $S$-cameral curves are related by a global isomorphism.

The data of the $W_L$-equivariant $\bar{Z}$-bundle $V_{\bar{Z}}\otimes (V_{\boldsymbol{\xi}})_{\bar{Z}}^{-1}$ is equivalent to the data of the $W_i$-equivariant $\bar{Z}_i$-bundles $U_{\bar{Z}_i}\otimes (U_{\boldsymbol{\xi}})_{\bar{Z}_i}^{-1}$ on each factor $\check\Sigma_{\xi_i}$, where $U_{\bar{Z}_i}$ and $(U_{\boldsymbol{\xi}})_{\bar{Z}_i}$ descend (under the natural projection maps on $\hat{\Sigma}_{\boldsymbol{\xi}}$, see (\ref{GLnAbHitchin3eqn})) from the $\bar{Z}_i$-bundles $V_{\bar{Z}_i}$ and $(V_{\boldsymbol{\xi}})_{\bar{Z}_i}$ induced by the projection maps $\bar{Z} \rightarrow \bar{Z}_i$; so $V_{\bar{Z}}$ is determined by the tuple $(U_{\bar{Z}_1}, ..., U_{\bar{Z}_s})$, and the image of the abelianisation map can be identified with the space of such tuples which correspond to some $(V,\Phi) \in \mathcal M_K(\GL_n;S_{\bf m})$.

Let $L_i$ be the factor of $L$ consisting of all of the simple factors of $L$ of the form $\GL_i$; the $\bar{Z}_i$-bundles $U_{\bar{Z}_i}$ can be regarded as the abelianisations of $L_i$-bundles $U_{L_i}$ on $\check{\Sigma}_{\xi_i}$, which descend from the $L_i$-bundles $V_{L_i}$ on $\hat{\Sigma}_{\boldsymbol{\xi}}$ induced by the projection maps $P \rightarrow L_i$. The abelianisation is given by taking the determinant line bundles of the vector bundles corresponding to the $\GL_i$-factors.

For $x \in \Sigma$, we can regard the points in the fibre $\check{p}_{\xi_i}^{-1}(x)$ of $\check{p}_{\xi_i}:\check\Sigma_{\xi_i}\rightarrow \Sigma$ as corresponding to orderings of the $l_i$-tuple $\boldsymbol{\lambda}_i$ of eigenvalues of the Higgs field $\Phi_x$; meanwhile, the points of ${\Sigma}_{\xi_i}$ correspond to choices of an eigenvalue $\lambda$ from the tuple $\boldsymbol{\lambda}_i$. As in \cite[Section 9]{DG}, the cover $\check\Sigma_{\xi_i} \rightarrow \Sigma$ factors through a map $\check\Sigma_{\xi_i} \rightarrow {\Sigma}_{\xi_i}$ sending an ordering of $\boldsymbol{\lambda}_i$ to its first coordinate. Under this map, the $\GL_i$-bundle factor of $V_{L_i}$ corresponding to the first coordinate descends to the bundle $\mathcal E_i$ on ${\Sigma}_{\xi_i}$. Thus the line bundle factor of $V_{\bar{Z}_i}$ corresponding to the first coordinate descends to $\wedge^i \mathcal E_i$, and by the arguments of \cite[Section 9]{DG} it corresponds to the $\GL_{l_i}$-Higgs bundle $(V^{ab}_i,\Phi^{ab}_i)$ via Theorem \ref{SpectralDatathm}.
\end{proof}

We note the similarity with the constructions in Sections 3 and 7 of \cite{FraPN}, which describe the Hitchin fibre for a reducible spectral curve via its normalisation. In that case, the pullback and pushforward under the normalisation map can be used to define maps between the Hitchin fibres for $\GL_n$ and for an associated Levi subgroup. 

\subsection{Sheet-valued Higgs bundles for \texorpdfstring{$\Sp_4$}{Sp4}}\label{Sp4Hitchinsbn}
We consider the case of $G= \Sp_4$ as in Example \ref{Sp4Sheetexm}. There are particular differences that arise with the existence of a non-Dixmier sheet in Example \ref{Sp4MinOrbexm}, and a Dixmier sheet with non-trivial Katsylo group in Example \ref{Sp4SDix2exm}. We show in the latter case how the non-triviality of the Katsylo group affects the geometry of the corresponding locus of sheet-valued Higgs bundles.

We recall from \cite[Section 5.10]{Hitchin2} that $\Sp_4$-Higgs bundles on $\Sigma$ correspond to triples $(V,\omega,\Phi)$, where $V$ is a rank $4$ vector bundle on $\Sigma$, $\omega$ is a symplectic form on $V$, and $\Phi \in H^0(\Sigma,\End(V)\otimes K)$ satisfies $\Phi = -\Phi^* $, where $\Phi^*$ is the adjoint of $\Phi$ under $\omega$. 

The Hitchin base for $\Sp_4$ is given by
\begin{equation}\label{Sp4Hitchinbaseeqn}
    \mathcal A_K(\Sp_4) = H^0(\Sigma,K^2) \oplus H^0(\Sigma,K^4).
\end{equation}
Indeed, taking the $\GL_4$-Higgs bundle $(V,\Phi)$ corresponding to an $\Sp_4$-Higgs bundle, the characteristic equation $\xi(T;\Phi)$ is of the form $\xi(T;\Phi) = T^4+a_2T^2+a_4$, and the Hitchin map $h_{\Sp_4}$ coincides with the pullback of the $\GL_4$-Hitchin map $h_{\GL_4}$ under the map
\begin{equation}\label{Sp4toGL4Higgseqn}
    \mathcal M_K(\Sp_4) \rightarrow \mathcal M_K(\GL_4).
\end{equation}
In particular, we can associate to $(V,\omega,\Phi)$ the spectral curve ${\Sigma}_{\xi(\lambda;\Phi)}$ for the $\GL_4$-Higgs bundle $(V,\Phi)$. 

There is an involution $\theta$ on ${\Sigma}_{\xi(\lambda;\Phi)}$ given by sending $\lambda$ to $-\lambda$. There is a spectral correspondence for $\Sp_4$ which gives an isomorphism between the fibres of $h_{\Sp_4}$ and a ``Prym stack" for the involution $\theta$ on ${\Sigma}_{\xi(\lambda;\Phi)}$. See \cite[Section 4.3]{Carbone} for the extension of \cite{Hitchin2} to the case where the spectral curve is not smooth.

We recall that there are five sheets of $\Sp_4$. The $0$ sheet is trivial, while the Hitchin fibration for the regular sheet is dealt with by the specifics of \cite{Hitchin2} and the general results of \cite{DG} and \cite{Ngo1}. We consider the remaining three cases individually.

\begin{example}\label{Sp4MinOrbexm}
    If $S$ is the sheet consisting only of the minimal orbit $\mathcal O_{min}$, the $S$-Hitchin base is a single point, and the stack $\mathcal M_K(\Sp_4;\mathcal O_{min})$ is entirely contained in the nilpotent cone, and is isomorphic to the stack of $C_G(e)$-torsors for any representative $e \in \mathcal O_{min}$ e.g. by a degenerate application of Theorem \ref{GlobalNonabnthm}. The Zariski closure of $\mathcal M_K(\Sp_4;\mathcal O_{min})$ determines a class of components in the reduced substack of the nilpotent cone described in \cite[Section 2.10.3]{BD}. This is the general picture for sheets consisting of a single rigid orbit.

    This case gives some counterexamples to some closure properties of the centraliser stratification which might be expected. In particular, $\mM_K(\Sp_4;\mO_{min})$ does not contain $\mM_K^{\gen}(\Sp_4;\mO_{min})$ in its closure; and also $\mM_K^{\gen}(\Sp_4;\mO_{min})=\mM^{\gen,6}_K(\Sp_4)$ does not contain $\mM_K^{\gen, 10}(\Sp_4)$ in its closure. Both of these properties follow from \cite[Section 2.10.3]{BD}. 
\end{example}

\begin{example}\label{Sp4SDix1exm}
    If $S$ is the sheet $S_{\Dix}'$ corresponding to the Levi subgroup $L = \GL_2$, then an $\Sp_4$-Higgs bundle $(V,\omega,\Phi)$ is $S_{\Dix}'$-valued if and only if $(V,\Phi)$ is an $S_{(2^2)}$-valued $\GL_4$-Higgs bundle. The analysis of the previous subsection can be made compatible with the involution $\theta$, and it can be shown as in Proposition \ref{GLnAbHitchinprp} that the abelianised fibration can be realised as the Hitchin fibration for $\Sp_2$. We omit further details.
\end{example}

The final example differs from the others we have considered in that it has non-trivial Katsylo group. We consider this in greater detail.

\begin{example}\label{Sp4SDix2exm}
    Let $S$ be the Dixmier sheet $S_{\Dix}$ corresponding to the Levi subgroup $L=\mathbb{G}_m \times \Sp_2$. Explicitly, $\mathcal M_K(\Sp_4;S_{\Dix})$ consists of the $\Sp_4$-Higgs bundles $(V,\omega,\Phi)$ such that for every $y \in \Sigma$, $\Phi_y$ can be represented in some trivialisation by the matrix
    \begin{equation}\label{Sp4Katsylo2Copyeqn}
        x_t = \frac{1}{4}\begin{pmatrix}
            2t & 0 & 4 & 0 \\
            0 & -2t & 0 & 4 \\
            t^2 & 0 & 2t & 0 \\
            0 & t^2 & 0 & -2t
        \end{pmatrix}.
    \end{equation}
    
     for some $t \in \mathbb{C}$, using the notation of (\ref{Sp4Katsylo2eqn}). We observe from the form of $x_t$ that $\Phi$ always has a kernel of rank 2; this can be most easily seen by taking simpler representatives for the conjugacy classes of $x_t$ as in Example \ref{Sp4Sheetexm}. Thus we can define the quotient $\GL_2$-Higgs bundle $(\overline{V},\overline{\Phi}) = (V/\Ker(\Phi),\overline{\Phi})$, noting that $\overline{\Phi}_y=0$ at every branch point $y \in \Sigma$ of the spectral cover for $(\overline{V},\overline{\Phi})$.

In this case, the Katsylo group $F:=F_{S_{\Dix}}$, defined in Definition \ref{KatsyloGpdef}, is non-trivial; indeed, $F$ can be identified with the group $W_L \cong \mathbb{Z}/2$, and its non-trivial element acts by $-1$ on the $1$-dimensional space $\mathfrak{z} = \Lie(Z(L))$. 

Proposition \ref{SHitchinBasePropsprp} implies the following form for the $S_{\Dix}$-Hitchin base.

\begin{proposition}
    The $S_{\Dix}$-Hitchin base $\mA_K(\Sp_4;S_{\Dix})$ has components indexed by the $\mathbb{Z}/2$-torsors on $\Sigma$; the component corresponding to the trivial torsor can be explicitly described by
\begin{equation}\label{Sp4DixHitchinbaseeqn}
    \mathcal A^0_K(\Sp_4;S_{\Dix}) = [H^0(\Sigma,K)/(\mathbb{Z}/2)]
\end{equation}
\end{proposition}

Any component of $\mA_K(\Sp_4;S_{\Dix})$ which is indexed by a non-trivial torsor $\pi:\tilde{\Sigma} \rightarrow \Sigma$ can be described as a quotient by $\mathbb{Z}/2$ of a subvariety of $H^0(\tilde{\Sigma},\pi^*K)$. 

The characteristic equation of an $S_{\Dix}$-valued $\Sp_4$-Higgs bundle $(V,\omega,\Phi)$ has the form $\xi(T;\Phi) = T^4-aT^2$ for some $a\in H^0(\Sigma,K^2)$. If $a \neq 0$, the corresponding spectral curve ${\Sigma}_{\xi(\lambda;\Phi)}$ decomposes into a non-reduced copy of $\Sigma$ and the spectral curve $\hat{\Sigma}_{\GL_2}$ of the quotient $\GL_2$-Higgs bundle $(\overline{V},\overline{\Phi})$. Since $\overline{\Phi}_x = 0$ at every branch point $x \in \Sigma$, the curve $\hat{\Sigma}_{\GL_2}$ has a node at every ramification point; the normalisation $\tilde{\Sigma} \rightarrow \hat{\Sigma}_{\GL_2}$ defines the $F$-torsor associated to the corresponding component of $\mathcal A_K(\Sp_4;S_{\Dix})$.

If we denote by ${\bf A}(\tilde{S}_{\Dix,K})$ the generalised Hitchin base for the normalisation of the closure of $S_{\Dix}$ as in Remark \ref{SHitchinBasermk}, and we consider the map $\gamma_S:\mA_K(\Sp_4;S_{\Dix}) \ra {\bf A}(\tilde{S}_{\Dix,K})$ of (\ref{SHitchinBase2eqn}), we have that ${\bf A}(\tilde{S}_{\Dix,K}) = H^0(\Sigma,K^2)$, and $\gamma_S$ has image in the sections of $K^2$ with doubled zeroes. In particular, since there exist sections of $K^2$ with simple zeroes (e.g. by Bertini), the map $\gamma_S$ is not surjective. Moreover, we see that if $(W, \omega)$ is a rank $2$ symplectic vector bundle on $\Sigma$ and $(W',\omega',\Phi')$ is an $\Sp_2$-Higgs bundle with smooth spectral curve, the $\Sp_4$-Higgs bundle $(W \oplus W',\omega\oplus \omega',0\oplus\Phi')$ is generically $S_{\Dix}$-valued but maps outside the image of $\gamma_S$ in the Hitchin base. This implies the following statement.

\begin{proposition}\label{Sp4ClosureFailureprp}
    The support of $\mM_K^{\gen}(\Sp_4;S_{\Dix})$ in $\mA_K(\Sp_4)$ is strictly larger than the support of $\mM_K(\Sp_4;S_{\Dix})$. In particular, $\mM^{\gen}_K(\Sp_4;S_{\Dix})$ is not contained in the closure of $\mM_K(\Sp_4;S_{\Dix})$.
\end{proposition}

We expect this to be a general feature of sheet-valued Higgs bundles with non-trivial Katsylo group.

We restrict to the component $\mathcal A^0_K(\Sp_4;S_{\Dix})$ as in Theorems \ref{GlobalNonabnthm} and \ref{GlobalSHiggsAbnthm}. We observe that the image of $\mathcal A^0_K(\Sp_4;S_{\Dix})$ under $\tilde{\mu}_S$ is exactly given by the sections $a \in H^0(\Sigma,K^2)$ of the form $a = b^2$ for some $b \in H^0(\Sigma,K)$. In particular, for $b \neq 0$ the curve ${\hat\Sigma}_{\GL_2}$ decomposes into two irreducible components both isomorphic to $\Sigma$, embedded as the vanishing of the sections $\lambda + b$ and $\lambda - b$ in $T^*\Sigma$. Since $\overline{\Phi}$ vanishes at the zeroes of $b$, $(\overline{V},\overline{\Phi})$ is of the form $(L_1 \oplus L_2, b \oplus(-b))$. We regard the pairs $(L_1,b)$ and $(L_2,-b)$ as $\mathbb{G}_m$-Higgs bundles on $\Sigma$.

By the calculations in Example \ref{Sp4Sheetexm}, the symplectic form $\omega$ on $V$ has (simple) zeroes on the kernel of $\Phi$ exactly at the zeroes of $b$ whenever $b \neq 0$; hence $\omega$ defines an isomorphism $\wedge^2\overline{V} \cong K$, i.e. an isomorphism $L_2 \cong L_1^{-1} \otimes K$. We observe that there is an involution $\Theta$ on $\mathcal M_K(\mathbb{G}_m)$ exchanging $(L,b)$ and $(L^{-1} \otimes K,-b)$. Away from $0 \in \mathcal A^0_K(\Sp_4;S_{\Dix})$, the above construction determines a map
\begin{equation}\label{Sp4DixAbneqn}
   \mathcal M_K^0(\Sp_4;S_{\Dix}) \rightarrow [\mathcal M_K(\mathbb{G}_m)/\Theta] 
\end{equation}
realising $\Ab_{S_{\Dix}}$ and identifying $h_{S_{\Dix}}^{ab}$ with the quotient of the Hitchin fibration for $\mathbb{G}_m$.
\end{example}

\section{The Hitchin fibration for real forms}\label{RealHitchinscn}
We now apply the above constructions for sheet-valued Higgs bundles to the Hitchin fibration for regular Higgs bundles for a real form $G_{\mathbb{R}}$ of $G$, or equivalently  regular Higgs bundles associated to a symmetric pair $(G,G^\theta)$ for some holomorphic involution $\theta$ on $G$. We show that the existing gerbe descriptions of \cite{GPPN} and \cite{HM} for real regular Higgs bundles can be viewed as $\theta$-equivariant versions of the gerbe descriptions for the $S$-Hitchin fibration associated to a suitable Dixmier sheet $S$. We adapt the abelianised fibration of Section \ref{IntAbFibrnsbn} to the Hitchin fibration for an arbitrary real form, and make it explicit in the non-quasi-split cases $G_{\mathbb{R}}=\SU(p,q)$ and $G_{\mathbb{R}}=\SO^*(4m+2)$. 

The general results in this section depend upon work in preparation of Bulois on smoothness of certain sheets in exceptional Lie algebras \cite{Bulois}.

\subsection{Regular \texorpdfstring{$G_{\mathbb{R}}$}{GR}-Higgs bundles as sheet-valued \texorpdfstring{$G$}{G}-Higgs bundles}\label{RealHitchinsbn}

We first recall the definitions for Higgs bundles associated to real forms; further details can be found in Sections 2 and 4 of \cite{GPPN}.
 
A \emph{real form} $G_{\mathbb{R}}$ is a real Lie subgroup of $G$ which is the fixed point group for an anti-holomorphic involution $\sigma$ on $G$. There is a correspondence between real forms $G_\mathbb{R}$ of $G$ and holomorphic involutions $\theta$ on $G$, up to conjugacy \cite[Section VI.3]{Knapp}. The real forms of $G$ are thus classified by Satake diagrams  \cite{Araki}.

Let $H$ be the fixed-point group for the involution $\theta$ and consider the eigenspace decomposition of $\theta$ on $\mathfrak{g}$ given by
\begin{equation}\label{CartanDecompositioneqn}
        \mathfrak{g} = \mathfrak{h}\oplus \mathfrak{m},
    \end{equation}
    where $\mathfrak{h} = \Lie(H)$ is the $+1$-eigenspace and $\mathfrak{m}$ is the $-1$-eigenspace. The group $H$ acts on $\mathfrak{m}$ through the adjoint action of $G$; this is the \emph{isotropy representation} of $H$ on $\mathfrak{m}$. 
    
    We define the moduli stack of $G_\mathbb{R}$-Higgs bundles as in \cite[Section 4.1]{Peon-Nieto1}.

\begin{definition}\label{RealHiggsdef}
    The \emph{moduli stack of (twisted) $G_{\mathbb{R}}$-Higgs bundles on $\Sigma$} is the mapping stack 
    \begin{equation}\label{RealHiggseqn}
        \mathcal M(G_{\mathbb{R}}) = \Maps(\Sigma,[\mathfrak{m} /H\times \mathbb{G}_m]).
    \end{equation}
\end{definition}

\begin{remark}\label{RealHiggsrmk}
    A \emph{$G_\mathbb{R}$-Higgs bundle} is then a triple $(E, \mathcal L, \Phi)$ for $E$ an $H$-bundle on $\Sigma$, $\mathcal L$ a line bundle on $\Sigma$, and $\Phi$ a global section of $(E \times^H\mathfrak{m}) \otimes \mathcal L$. This definition originates from Section 6 of \cite{Simpson1} (though specific cases had already been considered in \cite{Hitchin1} and \cite{Hitchin3}) and can be motivated by non-abelian Hodge theory; see e.g. \cite[Corollary 6.16]{Simpson1} and \cite[Theorem 3.32]{GPGMIR}.
    
    As for the complex group case, $\mathcal M(G_\mathbb{R})$ is a quasi-separated algebraic stack, locally of finite presentation over $\mathbb{C}$. We will use the same notational convention as in Section \ref{NonabnHitchinscn} when fixing a twist $\mathcal L$.
\end{remark}

We give the constructions for the Hitchin base and Hitchin map for real Higgs bundles, following \cite[Sections 3.1 and 4.1]{Peon-Nieto1}. We may assume that the torus $T$ of $G$ is stable under $\theta$ and the fixed point subgroup $T^\theta$ has minimal possible dimenison; $T$ is said to be \emph{maximally split} for the involution. This induces a decomposition of $\mathfrak{t}$ as
\begin{equation}\label{TorusCartandecompeqn}
    \mathfrak{t} = \mathfrak{d}\oplus \mathfrak{a}
\end{equation}
for $\mathfrak{d} = \mathfrak{t} \cap \mathfrak{h}$ and $\mathfrak{a} = \mathfrak{t} \cap \mathfrak{m}$. 

Let $W_{\mathfrak{a}} = N_H(\mathfrak{a})/C_H(\mathfrak{a})$; this is a finite group which acts by reflections on $\mathfrak{a}$. The Chevalley restriction theorem gives an isomorphism of GIT quotients $\mathfrak{m}//H \cong \mathfrak{a}/W_{\mathfrak{a}}$; we will denote $\mathfrak{c}_\mathbb{R} = \mathfrak{a}/W_{\mathfrak{a}}$. This defines a Chevalley map $\chi_\mathbb{R}:\mathfrak{m} \rightarrow \mathfrak{c}_\mathbb{R}$, which descends to a map $[\mathfrak{m}/H \times\mathbb{G}_m]\rightarrow [\mathfrak{c}_\mathbb{R}/\mathbb{G}_m]$ (which we also denote by $\chi_\mathbb{R}$).

\begin{definition}\label{RealHitchinmapdef}
   The \emph{$G_{\mathbb{R}}$-Hitchin base} is the mapping stack
   \begin{equation}\label{RealHitchinBasedef}
    \mathcal A(G_{\mathbb{R}})=\Maps(\Sigma,[\mathfrak{c}_\mathbb{R}/\mathbb{G}_m]).
    \end{equation}
    The \emph{$G_\mathbb{R}$-Hitchin map} is the morphism $h_{G_\mathbb{R}}:\mathcal M(G_\mathbb{R}) \rightarrow \mathcal A(G_\mathbb{R})$ induced by the Chevalley map $\chi_\mathbb{R}:[\mathfrak{m}/H \times\mathbb{G}_m]\rightarrow [\mathfrak{c}_\mathbb{R}/\mathbb{G}_m]$.
\end{definition}

\begin{remark}\label{RealHitchinBasermk}
    For a line bundle $\mathcal L$, the $\mathcal L$-twisted $G_\mathbb{R}$-Hitchin base $\mathcal A_\mathcal L(G_\mathbb{R})$ is representable by the vector space $H^0(\Sigma,\mathfrak{a}\otimes\mathcal L/W_\mathfrak{a})$. 
\end{remark}

 There is a map $[\mathfrak{m}/H] \rightarrow [\mathfrak{g}/G]$ induced by the inclusions $\mathfrak{m} \hookrightarrow \mathfrak{g}$ and $H \hookrightarrow G$. Similarly, there is a map $\mathfrak{c}_\mathbb{R} \rightarrow \mathfrak{c}$ and a commutative diagram
\begin{equation}\label{RealtoComplexCheveqn}
\begin{tikzcd}
{[\mathfrak{m}/H]} \arrow[r] \arrow[d, "\chi_{\mathbb{R}}"'] & {[\mathfrak{g}/G]} \arrow[d, "\chi"] \\
\mathfrak{c}_\mathbb{R} \arrow[r]                      & \mathfrak{c}.                      
\end{tikzcd}
\end{equation}
This induces a commutative diagram
\begin{equation}\label{RealtoComplexHitchineqn}
\begin{tikzcd}
{\mathcal M(G_{\mathbb{R}})} \arrow[r] \arrow[d, "h_{G_{\mathbb{R}}}"'] & {\mathcal M(G)} \arrow[d, "h_G"] \\
{\mathcal A( G_{\mathbb{R}})} \arrow[r]                             & {\mathcal A(G)}.              
\end{tikzcd}
\end{equation}

We will now restrict our attention to regular $G_\mathbb{R}$-Higgs bundles. A point $x \in \mathfrak{m}$ is \emph{regular} if $C_H(x)$ has minimal possible dimension; we denote the locus of regular points by $\mathfrak{m}^{reg}$, which is a dense open subset of $\mathfrak{m}$. Importantly, this notion of regularity does not necessarily coincide with regularity in $\mathfrak{g}$ under the adjoint action.

\begin{definition}\label{QuasiSplitdef}
    A real form $G_\mathbb{R}$ of $G$ is \emph{quasi-split} if $\mathfrak{g}^{reg} \cap \mathfrak{m} \neq \emptyset$.
\end{definition}

Let $L = C_G(\mathfrak{a})$; this is a Levi subgroup of $G$ since $\mathfrak{a}$ is a toral subalgebra of $\mathfrak{g}$, and there is an associated Dixmier sheet $S_H$ (see Definition \ref{DixmierSheetdef} and Remark \ref{DixmierSheetrmk}).

\begin{lemma}\label{RealSheetlem}
    The Dixmier sheet $S_H$ associated to $L$ is the unique sheet containing $\mathfrak{m}^{reg}$.
\end{lemma}
\begin{proof}
By \cite[Proposition 5]{KR}, the elements of $
\mathfrak{m}^{reg}$ all have $G$-centraliser of the same dimension. Hence, since $\mathfrak{m}^{reg}$ is irreducible it must be contained in a sheet of $\mathfrak{g}$. Moreover, by \cite[Remark 3]{KR} and \cite[Proposition 8]{KR}, there exist elements in $\mathfrak{a}$ with centraliser $L$, and these are contained in $\mathfrak{m}^{reg}$. The sheet $S_H$ is the unique sheet containing such elements, so $\mathfrak{m}^{reg}$ must be contained in $S_H$.
\end{proof}

\begin{remark}\label{RealSheetrmk}
     The real form $G_{\mathbb{R}}$ is quasi-split exactly when $S_H$ is the regular sheet.
\end{remark}

In order to apply the constructions from Section \ref{NonabnHitchinscn} to $G_\mathbb{R}$-Hitchin fibrations for arbitrary $G_\mathbb{R}$, we require that the sheets $S_H$ which can appear in Lemma \ref{RealSheetlem} are non-singular. If $G$ is a classical group, this is automatic by Theorem \ref{ClassicalSheetsthm}, and it is implied in general by work in preparation of \cite{Bulois}.

There is an open substack of $\mathcal M(G_{\mathbb{R}})$ defined by
\begin{equation}\label{RegularRealHiggsrmk}
    \mathcal M^{reg}(G_{\mathbb{R}}) = \Maps(\Sigma,[\mathfrak{m}^{reg}/H \times \mathbb{G}_m]),
\end{equation}
the stack of \emph{regular $G_{\mathbb{R}}$-Higgs bundles}.

\begin{proposition}\label{RealSheetprp}
    There is a commutative diagram
\begin{equation}\label{RealSheet1eqn}
\begin{tikzcd}
{\mathcal M^{reg}(G_{\mathbb{R}})} \arrow[r] \arrow[d, "h_{G_{\mathbb{R}}}"'] & {\mathcal M(G;S_H)} \arrow[r, hook] \arrow[d, "h_{S_H}"'] & {\mathcal M(G)} \arrow[d, "h_G"'] \\
{\mathcal A(G_{\mathbb{R}})} \arrow[r]                             & {\mathcal A( G;{S_H})} \arrow[r, "\tilde{\mu}_{S_H}"]                  & {\mathcal A(G)}               
\end{tikzcd}
\end{equation}
compatible with the diagrams (\ref{SHitchinBase2eqn}) and (\ref{RealtoComplexHitchineqn}).
\end{proposition}

\begin{proof} 
Let $\rho_{S_H}:S_H \rightarrow \mathcal B$ be the $S$-Chevalley map for $S_H$. By Proposition \ref{StackyCBKQprp}, it suffices to construct a commutative diagram
\begin{equation}\label{RealSheet2eqn}
    \begin{tikzcd}
{[\mathfrak{m}^{reg}/H]} \arrow[r] \arrow[d, "\chi_\mathbb{R}"] & {[S/G]} \arrow[d, "\rho_{S_H}"] \\
\mathfrak{c}_\mathbb{R} \arrow[r]           & \mathcal B 
\end{tikzcd}
\end{equation}
compatible with the $\mathbb{G}_m$-actions and the diagrams (\ref{StackyCBDiagrams2eqn}) and (\ref{RealtoComplexCheveqn}). The upper horizontal arrow in (\ref{RealSheet2eqn}) is induced by the inclusions $\mathfrak{m}^{reg} \subseteq S$ and $H \leq G$. Choosing a Kostant-Rallis section of the map $\chi_{\mathbb{R}}$ \cite[Theorems 11, 12 and 13]{KR} constructs the lower horizontal map. To see that this does not depend on the choice of Kostant-Rallis section, we observe that any two choices for the lower horizontal map will agree over the dense open locus $\mathcal B^{rs}$ in $\mathcal B$ (defined as in the proof of Propositon \ref{SHitchinBaseNormprp}); thus by \cite[Proposition A.1]{FMN}, any two such choices define the same map. The $\mathbb{G}_m$-equivariance follows by a similar argument (as in the proof of Proposition \ref{GmEquivStackyCBKAprp}).
\end{proof}

The map $\chi_\mathbb{R}:[\mathfrak{m}/H] \rightarrow \mathfrak{c}_\mathbb{R}$ is not always a gerbe; unlike the case for sheets, $\mathfrak{c}_\mathbb{R}$ is not in general a geometric quotient for the $H$-action on $\mathfrak{m}^{reg}$. However, by \cite[Section 4.2]{GPPN} and \cite[Theorem 3.16]{HM}, there is a scheme $\mathfrak{C}_{\mathbb{R}}$ with a $\mathbb{G}_m$-action and a $\mathbb{G}_m$-equivariant factorisation
\begin{equation}\label{RealChevGerbeeqn}
\begin{tikzcd}
{[\mathfrak{m}^{reg}/H]} \arrow[r, "{\boldsymbol{ \chi}}_{\mathbb{R}}"] & {\mathfrak{C}_{\mathbb{R}}} \arrow[r] & {\mathfrak{c}_\mathbb{R}}
\end{tikzcd}
\end{equation}
such that $\boldsymbol{\chi}_{\mathbb{R}}$ is a gerbe and the map $\mathfrak{C}_{\mathbb{R}} \rightarrow \mathfrak{c}_\mathbb{R}$ is surjective, quasi-finite and generically injective, but not in general separated. As is discussed in \cite[Section 3]{HM}, the scheme $\mathfrak{C}_\mathbb{R}$ is the regular quotient for the $H$-action on $\mathfrak{m}$ in the sense of \cite[Section 4.2]{Ngo3}.

This induces a factorisation of the real Hitchin map
\begin{equation}\label{RealHitchinGerbeeqn}
\begin{tikzcd}
{\mathcal M_{reg}(G_{\mathbb{R}})} \arrow[r, " \tilde{h}_{\mathbb{R}}"] & {\mathfrak{A}(G_{\mathbb{R}})} \arrow[r, "\xi"] & {\mathcal A(G_{\mathbb{R}})}
\end{tikzcd}
\end{equation}
 where $ \mathfrak{A}(G_\mathbb{R}) = \Maps(\Sigma,[\mathfrak{C}_{\mathbb{R}}/\mathbb{G}_m])$ has the following properties.

\begin{proposition}\label{RealHitchinGerbeprp}\cite[Section 4.2]{GPPN}, \cite[Theorems 5.2 and 5.3]{HM}
    The map $\xi$ is generically \'{e}tale; and for each $G_{\mathbb{R}}$-Higgs bundle $(E,\mathcal L,\Phi)$ mapping to a $\mathbb{C}$-point $\tau' \in  \mathfrak{A}(G_{\mathbb{R}})$ under ${\tilde{h}}_{\mathbb{R}}$, there is a group scheme $\mathcal I^H_{(E,\mathcal L, \Phi)}$ on $\Sigma$ such that the fibre $h^{-1}(\tau')$ can be identified with the stack ${\bf B}_{\Sigma} \mathcal I_{(E,\mathcal L,\Phi)}^H$ of $\mathcal I_{(E,\mathcal L, \Phi)}^H$-torsors on $\Sigma$.
\end{proposition}

\begin{remark}\label{RealHitchinGerbermk}
    More specifically, the $H$-centraliser group scheme $\mathcal I^H$ on $\mathfrak{m}^{reg}$ descends to the inertia stack $\mathcal I^H_{H\times \mathbb{G}_m}$ on $[\mathfrak{m}^{reg}/H\times \mathbb{G}_m]$, and $\mathcal I_{(E,\mathcal L,\Phi)}^H = f_{(E,\mathcal L,\Phi)}^*\mathcal I^H_{H\times \mathbb{G}_m}$ for the map $f_{(E,\mathcal L,\Phi)}:\Sigma \rightarrow [\mathfrak{m}^{reg}/H \times \mathbb{G}_m]$ defined by $(E,\mathcal L, \Phi)$.
\end{remark}

We observe that this description is compatible with our description in Theorem \ref{HiggsNonAbnthm} for the sheet $S_H$. The involution $\theta$ on $G$ defines an involution on $\mathcal I_S|_{\mathfrak{m}^{reg}}$, which we also denote by $\theta$ and which is equivariant with respect to the actions by $H$ and $\mathbb{G}_m$. Moreover, by definition, $\mathcal I^H$ is the fixed point subgroup scheme of $\mathcal I_S|_{\mathfrak{m}^{reg}}$ under the $\theta$-action.

\begin{lemma}\label{SmCentraliserInvlem}
   The involution $\theta$ restricts to an involution on $\mathcal I_S^{sm}|_{\mathfrak{m}^{reg}}$, and $\mathcal I^H$ is the fixed point subgroup scheme $\mathcal I^{sm}_S|_{\mathfrak{m}^{reg}}$ under $\theta$.
\end{lemma}

\begin{proof} 
We observe that on the dense open subset $\mathfrak{m}^{rs} \subseteq \mathfrak{m}^{reg}$ of semisimple elements, $\mathcal I_S^{sm}|_{\mathfrak{m}^{rs}} = \mathcal I_S|_{\mathfrak{m}^{rs}}$ (e.g. by Corollary \ref{KatsyloAffinecor}), so certainly $\theta$ sends $\mathcal I_S^{sm}|_{\mathfrak{m}^{rs}}$ to itself; moreover, since $\mathcal I_S^{sm}|_{\mathfrak{m}^{reg}}$ is smooth over $\mathfrak{m}^{reg}$, $\mathcal I_S^{sm}|_{\mathfrak{m}^{rs}}$ is dense in $\mathcal I_S^{sm}|_{\mathfrak{m}^{reg}}$. Then since $\mathcal I_S^{sm}|_{\mathfrak{m}^{reg}}$ is a closed subgroup of $\mathcal I_S|_{\mathfrak{m}^{reg}}$, $\theta$ must send $\mathcal I_S^{sm}|_{\mathfrak{m}^{reg}}$ to itself. The second statement follows from Proposition \ref{SmCentraliserPropsprp} and \cite[Theorem 3.10]{HM}.
\end{proof}

The following corollary is then straightforward. Let $(E,\mathcal L, \Phi)$ be a regular $G_\mathbb{R}$-Higgs bundle, which we also view as an $S_H$-valued Higgs bundle. Let $\tau' \in \mathfrak{A}(G_\mathbb{R})$ be the image of $(E,\mathcal L, \Phi)$ under the map $\tilde{h}_\mathbb{R}$, and let $\tau \in \mathcal A(G;S_H)$ be the image of $(E,\mathcal L, \Phi)$ under $h_{S_H}$. We use the notation of Theorem \ref{HiggsNonAbnthm} and Proposition \ref{RealHitchinGerbeprp}, and note that the involution $\theta$ on $\mathcal I_S^{sm}|_{\mathfrak{m}^{reg}}$ defines an involution on $\mathcal I^{sm}_{(E,\mathcal L,\Phi)}$.

\begin{corollary}\label{RealHitchinGerbecor}
    The group scheme $\mathcal I^H_{(E,\mathcal L,\Phi)}$ can be identified with the fixed point group scheme $(\mathcal I^{sm}_{(E,\mathcal L,\Phi)})^{\theta}$. Moreover, the map of Hitchin fibres $\tilde{h}_{\mathbb{R}}^{-1}(\tau') \rightarrow h_{S_H}^{-1}(\tau)$ induced by the diagrams (\ref{RealSheet1eqn}) and (\ref{RealHitchinGerbeeqn}) can be identified with the morphism ${\bf B}_{\Sigma}(\mathcal I_{(E,\mathcal L,\Phi)}^{sm})^\theta \rightarrow {\bf B}_{\Sigma}\mathcal I^{sm}_{(E,\mathcal L,\Phi)}$.
\end{corollary}

We consider also how to adapt the abelianised Hitchin fibration of Section \ref{IntAbFibrnsbn} to the context of $G_{\mathbb{R}}$-Higgs bundles. We state the following important lemma. We let $G_\mathbb{R}$ be any real form, and let $S_H$ be the corresponding Dixmier sheet of Lemma \ref{RealSheetlem}. Recall the cameral homomorphism $\kappa_{S_H}:\mathcal I_{S_H}^{sm} \rightarrow \rho_{S_H}^*\hat{\mathcal J}_{S_H}$ of Proposition \ref{CameralHomprp}.

\begin{lemma}\label{CameralSmoothReallem}
    The cameral homomorphism $\kappa_{S_H}:\mathcal I_{S_H}^{sm} \rightarrow \rho_{S_H}^*\hat{\mathcal J}_{S_H}$ is smooth.
\end{lemma}

\begin{proof}
It is straightforward to reduce to checking this in the case that group $G$ is simple. We can check directly from the classification in \cite[5.11]{Araki} that unless $G_\mathbb{R}$ is the non-quasi-split real form $\F_{4,(-20)}$, the sheet $S_H$ has classical reduction type, and so the statement follows from Proposition \ref{CameralSmoothprp}. For $\F_{4,(-20)}$, the sheet $S_H$ is the Dixmier sheet of $\F_4$ corresponding to the Levi subgroup of type $B_3$, and so the statement is Proposition \ref{F4CameralSmoothprp}.
\end{proof}

Thus all of the constructions of Sections \ref{Abelianisationscn} and \ref{NonabnHitchinscn} can be applied to the sheet $S_H$. Let $\mathcal J_{S_H}$ be the cameral group for the sheet $S_H$, defined by Definition \ref{CameralDescentprp}, and denote by $\mathcal J_{\mathfrak{c_\mathbb{R}}}$ its pullback to $\mathfrak{c}_\mathbb{R}$ under the map $\mathfrak{c}_\mathbb{R} \rightarrow \mathcal B$.

\begin{lemma}\label{CameralHomInvollem}
    There is an involution $\Theta$ on $\mathcal J_{\mathfrak{c}_{\mathbb{R}}}$ such that the cameral homomorphism $\kappa_S|_{\mathfrak{m}^{reg}}: \mathcal I^{sm}_{S_H}|_{\mathfrak{m}^{reg}} \rightarrow \chi_\mathbb{R}^*\mathcal J_{\mathfrak{c}_\mathbb{R}}$ (defined by Proposition \ref{CameralHomprp}) is equivariant with respect to the $\mathbb{Z}/2$-actions defined by $\theta$ and $\Theta$.
\end{lemma}

\begin{proof} 
The construction is similar to that of \cite[Proposition 20]{GPPN}. Since $L = C_G(\mathfrak{a})$, the involution $\theta$ on $G$ sends $L$ to itself, and thus defines an involution on its abelianisation $\bar{Z}$; moreover, $\theta$ also defines an involution on $W_L = N_G(L)/L$. Similarly, the involution $\theta$ on $\mathfrak{g}$ restricts to an involution on $\mathfrak{z}$. These maps are compatible with the $W_L$-actions, in the sense that for any $w \in W_L$, $\bar{g} \in \bar{Z}$ and $x \in \mathfrak{z}$, we have $\theta(w\bar{g})=\theta(w)\theta(\bar{g})$ and $\theta(wx) = \theta(w)\theta(x)$. In particular, if $y \in \mathfrak{z}$ lies in the $W_L$-orbit of $x \in \mathfrak{a}$, $-\theta(y)$ also lies in this $W_L$-orbit. Thus the map
\begin{equation}\label{CameralHomInvol0eqn}
\begin{tikzcd}
\mathfrak{c}_\mathbb{R} \times_\mathcal B \mathfrak{z} \arrow[r] & \mathfrak{z} \arrow[r, "-\theta"] & \mathfrak{z} \arrow[r] & \mathcal B
\end{tikzcd}
\end{equation}
agrees with the structure map for $\mathfrak{c}_\mathbb{R} \times_\mathcal B \mathfrak{z}$ over the schematic locus of $\mathcal B$; and so as in the proof of Lemma \ref{StackyZquotientlem}, (\ref{CameralHomInvol0eqn}) coincides with the structure map by \cite[Proposition A.1]{FMN}. This defines an involution $-\theta$ on $\mathfrak{c}_\mathbb{R} \times_\mathcal B \mathfrak{z}$.

We denote by $\hat{\mathcal J}_{\mathfrak{c}_\mathbb{R}}$ the pullback of the pseudo-cameral group $\hat{\mathcal{J}}$, defined in Definition \ref{PseudoCameralgpdef}, under the map $\mathfrak{c}_\mathbb{R} \rightarrow \mathcal B$. We define an involution $\Theta$ on $\hat{\mathcal J}_{\mathfrak{c}_\mathbb{R}}$ as follows. Suppose we have a morphism of schemes $X \rightarrow \mathfrak{c}_\mathbb{R}$; the $X$-points of $\hat{\mathcal J}_{\mathfrak{c}_\mathbb{R}}$ over this morphism correspond to $W_L$-equivariant morphisms $f:X \times_{\mathcal B}\mathfrak{z} \rightarrow \bar{Z}$. We define $\Theta(f):X \times_{\mathcal B}\mathfrak{z} \rightarrow \bar{Z}$ by
\begin{equation}\label{CameralHomInvol2eqn}
\begin{tikzcd}
X \times_{\mathcal B}\mathfrak{z} \arrow[r, "-\theta"] & X \times_{\mathcal B}\mathfrak{z} \arrow[r, "f"] & \bar{Z} \arrow[r, "\theta"] & \bar{Z}.
\end{tikzcd}
\end{equation}
The morphism $\Theta(f)$ is $W_L$-equivariant, and thus defines an $X$-point of $\hat{\mathcal J}_{\mathfrak{c}_\mathbb{R}}$ over $\mathfrak{c}_\mathbb{R}$. This defines $\Theta$ on $\hat{\mathcal J}_{\mathfrak{c}_\mathbb{R}}$.

We now observe that $\kappa_S|_{\mathfrak{m}^{reg}}: \mathcal I^{sm}_S|_{\mathfrak{m}^{reg}} \rightarrow \chi_\mathbb{R}^*\hat{\mathcal J}_{\mathfrak{c}_\mathbb{R}}$ is $\mathbb{Z}/2$-equivariant. This is straightforward on the locus $\mathfrak{m}^{rs}$, since $\kappa_S$ can be identified fibrewise with the abelianisation map $L \rightarrow \bar{Z}$, and this is certainly equivariant with respect to the $\mathbb{Z}/2$-actions; so by continuity, $\kappa_S|_{\mathfrak{m}^{reg}}$ is $\mathbb{Z}/2$-equivariant on all of $\mathfrak{m}^{reg}$. But this proves the lemma, since $\mathcal J_{\mathfrak{c}_\mathbb{R}}$ descends from the image of $\kappa_S|_{\mathfrak{m}^{reg}}$ along the morphism $\chi_\mathbb{R}$ by definition.
\end{proof}

\begin{proposition}\label{HCameralgpprp}
    The image of $\mathcal I^H$ under $\kappa_S|_{\mathfrak{m}^{reg}}$ descends to a smooth subgroup scheme $\mathcal J^H$ of $\mathcal J_{\mathfrak{c}_\mathbb{R}}$ on $\mathfrak{c}_{\mathbb{R}}$. The group scheme $\mathcal J^H$ is an open subgroup scheme of the fixed point group scheme $(\mathcal J_{\mathfrak{c}_\mathbb{R}})^\Theta$.
\end{proposition}

\begin{proof} 
By \cite[Theorem 9]{KR}, the map $\chi_\mathbb{R}:\mathfrak{m}^{reg} \rightarrow \mathfrak{c}_{\mathbb{R}}$ is a geometric quotient for the action of the group $G_\theta$, where
\begin{equation}\label{HCameralgpeqn}
    G_\theta = \{ g \in G \, | \, g^{-1}\theta(g) \in Z(G)\}.
\end{equation}
Thus, to see that the image of $\mathcal I^H$ descends to a smooth subgroup scheme $\mathcal J^H$ of $\mathcal J_{\mathfrak{c}_{\mathbb{R}}}$, it suffices to note that $\theta$ commutes with the $G_\theta$-action on $\mathcal I_S^{sm}$. By the $\mathbb{Z}/2$-equivariance statement of Lemma \ref{CameralHomInvollem}, $\mathcal J^H$ is a subgroup scheme of $(\mathcal J_{\mathfrak{c}_\mathbb{R}})^\Theta$. Moreover, the restriction $\kappa_S^\theta:\mathcal I^H \rightarrow \chi_\mathbb{R}^*(\mathcal J_{\mathfrak{c}^{reg}})^\Theta$ of $\kappa_S$ is smooth, by the same argument as, e.g., \cite[Lemme 2.4.1]{Ngo2}; thus $\mathcal J^H$ is open in $(\mathcal J_{\mathfrak{c}_\mathbb{R}})^\Theta$.
\end{proof}

\begin{remark}\label{HCameralgprmk}
    In the case that $G_\mathbb{R}$ is quasi-split, stronger versions of this statement are known (see \cite[Theorem 4.7]{Leslie} and \cite[Theorem 21]{GPPN}, with the correction of \cite[Proposition 4.14]{HM}).
    
    By the proposition, the fibres of $\mathcal J^H$ have finite index in $(\mathcal J_{\mathfrak{c}_\mathbb{R}})^\Theta$, as in Remark \ref{SmallCameralgprmk}. It is possible in general that $\mathcal J^H_x$ is a proper subgroup of $(\mathcal J_{\mathfrak{c}_\mathbb{R},x})^\Theta$ for every $x \in \mathfrak{c}_{\mathbb{R}}$; this occurs, for example, if $G_{\mathbb{R}}$ is the quaternionic special linear group $SU^*(2m)$ for $m >1$. In this case $\mathcal J^H$ is trivial, while the fibres of $(\mathcal J_{\mathfrak{c}_\mathbb{R}})^\Theta$ generically have order $2^{m-1}$ (corresponding to the $2$-torsion points of a torus of rank $m-1$), and always contain a non-identity element. See e.g. \cite[Section 12.3.2, Type AII]{GW} for further details.
\end{remark}

The involution $\theta$ stabilises the kernel $\mathcal N_{\mathfrak{m}^{reg}}$ of $\kappa_S|_{\mathfrak{m}^{reg}}$, and the fixed point group scheme $(\mathcal N_{\mathfrak{m}^{reg}})^\theta$ is a smooth closed subgroup scheme of $\mathcal I^H$ (from the proof of Proposition \ref{HCameralgpprp}). Since $\theta$ commutes with the actions of $H$ and $\mathbb{G}_m$, $(\mathcal N_{\mathfrak{m}^{reg}})^\theta$ descends to a smooth closed subgroup stack $\mathcal N^H_{H\times \mathbb{G}_m}$ of the inertia stack of $[\mathfrak{m}^{reg}/H \times \mathbb{G}_m]$. Thus we can define an algebraic stack $[\mathfrak{m}^{reg}/H\times \mathbb{G}_m]^{ab}$ as the rigidification of $[\mathfrak{m}^{reg}/H\times \mathbb{G}_m]$ by $ \mathcal N^H_{H \times \mathbb{G}_m}$ as in Definition \ref{AbelianisedStackdef}.

As in Proposition \ref{AbnGerbeprp}, this is a gerbe over $[\mathfrak{C}_{\mathbb{R}}/\mathbb{G}_m]$, banded by a group stack $\mathcal J^H_{\mathbb{G}_m}$ which descends from the pullback of $\mathcal J^H$ to $\mathfrak{C}_\mathbb{R}$.

\begin{definition}\label{AbnsdGRHiggsdef}
    The \emph{stack of abelianised regular $G_\mathbb{R}$-Higgs bundles on $\Sigma$} is the mapping stack
    \begin{equation}\label{AbnsdGRHiggs1eqn}
        \mathcal M^{ab}(G_\mathbb{R}) = \Maps(\Sigma,[\mathfrak{m}^{reg}/H \times \mathbb{G}_m]^{ab}).
    \end{equation}
\end{definition}

\begin{remark}\label{AbnsdGRHiggsrmk}
    There is a commutative diagram
        \begin{equation}\label{AbnsdGRHiggs2eqn}
\begin{tikzcd}
{\mathcal M^{reg}(G_{\mathbb{R}})} \arrow[r, " \Ab_\mathbb{R}"] \arrow[d] & {\mathcal M^{ab}(G_{\mathbb{R}})} \arrow[r, "{h}^{ab}_{\mathbb{R}}"] \arrow[d] & {\mathfrak{ A}(G_{\mathbb{R}})} \arrow[d] \\
{\mathcal M(G;{S_H})} \arrow[r, "\Ab_{S_H}"']                                             & {\mathcal M^{ab}(G;{S_H})} \arrow[r, "h_{S_H}^{ab}"']                                                   & {\mathcal A(G;S_H)}                
\end{tikzcd}
        \end{equation}
        such that $\tilde{h}_{\mathbb{R}} = { h}_{\mathbb{R}}^{ab} \circ { \Ab}_{\mathbb{R}}$.
\end{remark}

We state the analogy of Proposition \ref{AbnisedSHitchinFibresprp} in this context. For $\tau' \in \mathfrak{ A}(G_\mathbb{R})$, we denote by $\mathcal J_{\tau'}^H$ the pullback of $\mathcal J^H_{\mathbb{G}_m}$ under $\tau':\Sigma \rightarrow [\mathfrak{C}_\mathbb{R}/\mathbb{G}_m]$.
\begin{proposition}\label{RealNonAbnAbnprp}
        For any $\mathbb{C}$-point $\tau'$ of $\mathfrak{ A}(G_{\mathbb{R}})$, the fibre $({h}_{\mathbb{R}}^{ab})^{-1}(\tau')$ can be identified with the commutative group stack ${\bf B}_{\Sigma}(\mathcal J^H_{\tau'})$ of $\mathcal J^H_{\tau'}$-torsors.
\end{proposition}

Using Proposition \ref{HCameralgpprp}, we can interpret this in terms of a $\mathbb{Z}/2$-equivariant version of the cameral data of Theorem \ref{SCameralDatathm} in the spirit of \cite[Section 5.2]{GPPN}. 

We use the constructions from the proof of Lemma \ref{CameralHomInvollem}. Let $\tau'$ be a $\mathbb{C}$-point of $\mathfrak{ A}(G_{\mathbb{R}})$ mapping to a $\mathbb{C}$-point $\tau$ of $\mathcal A^\heartsuit(G;S)$. The involution $-\theta$ on $\mathfrak{c}_\mathbb{R} \times_{\mathcal B} \mathfrak{z}$ induces an involution $-\theta$ on the $S_H$-cameral curve $\hat\Sigma_\tau$ (defined in Definition \ref{SCameralcurvedef}). Thus we can define an involution $\Theta$ on the stack $\hat{\mathcal P}_{S_H,\tau}$ (defined in Theorem \ref{SCameralDatathm}) which acts on $\mathbb{C}$-points by 
\begin{equation}\label{CameralDatInvoleqn}
    \Theta(\mathcal Z) = (-\theta)^*\mathcal Z^\theta
\end{equation}
for any $W_L$-equivariant $\bar{Z}$-torsor $\mathcal Z$ on $\hat\Sigma_\tau$; here, $\mathcal Z^\theta$ denotes the torsor twisted by the involution $\theta$ on $\bar{Z}$ as in (\ref{WTwistTorsoreqn}). We denote the fixed point stack by $(\hat{\mathcal P}_{S_H,\tau})^\Theta$.

\begin{theorem}\label{HCameralDatathm}
    There is a finite map $({ h}_{\mathbb{R}}^{ab})^{-1}(\tau') \rightarrow (\hat{\mathcal P}_{S_H,\tau})^\Theta$, which factors through an isogeny to an open subgroup stack of $(\hat{\mathcal P}_{S_H,\tau})^\Theta$.
\end{theorem}

\begin{proof} 
After identifying $({ h}_{\mathbb{R}}^{ab})^{-1}(\tau')$ with ${\bf B}_{\Sigma}(\mathcal J^H_{\tau'})$ by Proposition \ref{RealNonAbnAbnprp}, the morphism 
$$({ h}_{\mathbb{R}}^{ab})^{-1}(\tau') \rightarrow (\hat{\mathcal P}_{S_H,\tau})^\Theta$$
can be defined by
\begin{equation}\label{HCameralDataeqn}
\begin{tikzcd}
{\bf B}_{\Sigma}(\mathcal J^H_{\tau'}) \arrow[r] & {\bf B}_{\Sigma}((\mathcal J_{\tau})^\Theta) \arrow[r] & {\bf B}_{\Sigma}((\hat{\mathcal J}_{\tau})^\Theta) \arrow[r] & ({\bf B}_{\Sigma}(\hat{\mathcal J}_{\tau}))^\Theta,
\end{tikzcd}
\end{equation}
recalling that ${\bf B}_{\Sigma}(\hat{\mathcal J}_{\tau})$ can be identified with $\hat{\mathcal P}_{S_H,\tau}$ by Proposition \ref{SCameralBunprp}. The second arrow in (\ref{HCameralDataeqn}) is an isogeny by Lemma \ref{PicardStackIsogenylem}, and the first arrow is an isogeny to an open and closed subgroup stack of ${\bf B}_{\Sigma}((\mathcal J_{\tau})^\Theta)$ by a similar argument. The third arrow is an inclusion of an open and closed subgroup stack as in Proposition \ref{SCameralBunprp}. Thus (\ref{HCameralDataeqn}) is finite as required.
\end{proof}

\begin{remark}\label{HCameralDatarmk}
    A more precise generalisation of Theorem \ref{SCameralDatathm} could be obtained by also including a condition analogous to $(*)$ controlling the $\theta$-equivariant structure at $\theta$-fixed points on $\hat\Sigma_\tau$. However, the statement of Theorem \ref{HCameralDatathm} would not be made any stronger by such an alteration, since the first arrow in (\ref{HCameralDataeqn}) is not in general essentially surjective by Remark \ref{HCameralgprmk}. 
\end{remark}

In the quasi-split cases, this is covered more fully and explicitly in \cite{GPPN}.

\subsection{Abelianisation for non-quasi-split real forms}\label{NonQSRealFormsbn}

We will now consider the abelianised Hitchin fibration for non-quasi-split real forms in more detail.

We will first make a general observation which governs the form of the abelianised fibration. We recall the following definition.

\begin{definition}\label{InnerRealFormsdef}
  Two real forms $G^1_{\mathbb{R}}$ and $G^2_{\mathbb{R}}$ of a complex reductive group $G$ are \emph{inner equivalent} if there is a $g \in G$ such that $\theta_2 = \theta_1 \circ I_g$; here $\theta_1$ and $\theta_2$ are the involutions on $G$ corresponding to $G^1_{\mathbb{R}}$ and $G^2_{\mathbb{R}}$ respectively.
\end{definition}

\begin{lemma}\label{InnerRealFormslem}
    Suppose $G^1_{\mathbb{R}}$ and $G^2_{\mathbb{R}}$ are inner equivalent real forms of $G$, with corresponding involutions $\theta_1$ and $\theta_2$ respectively. Suppose we have $x_1, x_2 \in \mathfrak{g}$ with $\theta_j(x_j)=-x_j$, such that $x_1$ and $x_2$ are $G$-conjugate.

    Then each $\theta_j$ induces an involution on $C_G(x_j)^{ab}$, and the corresponding fixed point groups $(C_G(x_1)^{ab})^{\theta_1}$ and $(C_G(x_2)^{ab})^{\theta_2}$ are canonically isomorphic. 
\end{lemma}

\begin{proof}
The first statement is clear, since the involution $\theta_j$ preserves the derived subgroup $C_G(x_j)^{der}$.

By assumption, there exists $g \in G$ such that $\theta_2 = \theta_1 \circ I_g$. We also have some $k \in G$ such that $x_2 = Ad_k(x_1)$, and thus $I_k:C_G(x_1) \rightarrow C_G(x_2)$ is an isomorphism, which also determines an isomorphism on the abelianisations independent of the choice of $k$. To conclude, it suffices to show that the diagram
\begin{equation}\label{InnerRealFormseqn}
\begin{tikzcd}
C_G(x_1)^{ab} \arrow[r, "I_k"] \arrow[d, "\theta_1"'] & C_G(x_2)^{ab} \arrow[d, "\theta_2"] \\
C_G(x_1)^{ab} \arrow[r, "I_k"']                       & C_G(x_2)^{ab}                      
\end{tikzcd}
\end{equation}
commutes, i.e. $\theta_1 = I_{k^{-1}} \circ \theta_2 \circ I_k$. But $I_{k^{-1}} \circ \theta_2 \circ I_k = \theta_1 \circ I_l$, where $l = \theta_1(k^{-1})gk$. One can use the properties of $x_1$ and $x_2$ to check that $l \in C_G(x_1)$, so $I_l$ acts as the identity on $C_G(x_1)^{ab}$.
\end{proof}

As a result, if $G_{\mathbb{R}}$ is inner to a \emph{split form}, i.e. one such that $\mathfrak{d} = 0$ in the decomposition (\ref{TorusCartandecompeqn}), the resulting abelianised fibration can only have finite fibres. 

\begin{proposition}\label{InnerSplitFormprp}
    Suppose $G_\mathbb{R}$ is inner equivalent to a split form. Then the group scheme $\mathcal J^H$ of Proposition \ref{HCameralgpprp} is a quasi-finite group scheme on $\mathfrak{c}_\mathbb{R}$, and the abelianised Hitchin map $ h_\mathbb{R}^{ab}:\mathcal M^{ab}(G_{\mathbb{R}}) \rightarrow \mathfrak{ A}(G_{\mathbb{R}})$ is quasi-finite.
\end{proposition}

\begin{proof}
We let $\theta$ be the involution on $G$ corresponding to $G_\mathbb{R}$, and as in Section \ref{RealHitchinsbn}, we denote the $-1$-eigenspace of $\theta$ by $\mathfrak{m}$. We also let $\theta_s$ be an involution on $G$ corresponding to a split real form, and we assume that the torus $T$ is split for $\theta_s$, i.e. $\theta_s$ acts by inversion on $T$. We first show that for any $x \in \mathfrak{m}^{rs}$, $\mathcal J^H_x$ is finite. We note that $x$ is conjugate to some $y \in \mathfrak{t}$, and by assumption $\theta_s(y)=-y$. Hence, by Lemma \ref{InnerRealFormslem}, $(C_G(x)^{ab})^\theta = (C_G(y)^{ab})^{\theta_s}$. Since $y$ is semisimple, the centre $Z$ of the Levi subgroup $L = C_G(y)$ surjects onto the abelianisation $\bar{Z} = C_G(y)^{ab}$. Then since $y \in \mathfrak{t}$, $Z$ is a subgroup of $T$, so $\theta_s$ acts by inversion on $\bar{Z}$. Hence $(C_G(y)^{ab})^{\theta_s}$ is finite; and by Propositions \ref{CameralHomprp} and \ref{HCameralgpprp}, $\mathcal J^H_x$ is a subgroup of $(C_G(y)^{ab})^{\theta_s}$, so is also finite. Thus, we can deduce that ${h}_{\mathbb{R}}^{ab}$ is quasi-finite.
\end{proof}

\begin{remark}\label{InnerSplitFormrmk}
    In particular, in the cases $G_\mathbb{R} = \SU^*(2m)$, $\Sp(2m,2m)$ and $\SO^*(4m)$, the fibration $h^{ab}_{\mathbb{R}}$ is trivial. However, in these cases, the non-abelian spectral data of \cite{HS} indicates that a full (albeit entirely non-abelian) cameral description is possible for the $G_\mathbb{R}$-Hitchin fibres.
\end{remark}

We now sketch the constructions realising the abelianised fibration for $G_{\mathbb{R}} = \SU(p,q)$,  when $p - q>1$, and for $G_{\mathbb{R}}=\SO^*(4m+2)$. We choose these examples as, up to isogeny, they are the only non-compact non-quasi-split examples of simple classical real forms where the abelianised Hitchin map has positive-dimensional fibres.

As in Section \ref{Hitchinegscn} we will fix the twisting line bundle to be the canonical bundle $K$ on $\Sigma$. In both of these cases, by \cite[Theorem 1]{Sekiguchi}, $\mathfrak{ A}_K(G_\mathbb{R}) = \mathcal A_K(G_\mathbb{R})$.

\begin{example}\label{SUpqAbnexm}
    Let $G_{\mathbb{R}} = \SU(p,q)$, the special unitary group on $\mathbb{C}^n$ of signature $(p,q)$ for $p+q=n$, where $p-q>1$.
    
    We use the notation of Section \ref{RealHitchinsbn}. The group $H = S(\GL_p \times \GL_q)$, and $\mathfrak{m}$ is the vector subspace of $\mathfrak{gl}_n$ of $n \times n$ matrices with block-diagonal form
    \begin{equation}\label{SUpqMeqn}
        \begin{pmatrix}
        0_p & B \\
        C & 0_q
    \end{pmatrix}.
    \end{equation}
     The Levi subgroup $L$ associated to the sheet $S_H$ containing $\mathfrak{m}^{reg}$ corresponds to the partition $(r, 1^{2q})$ (as in Example \ref{GLnSheetexm}), where $r = p-q$.

     Thus, an $\SU(p,q)$-Higgs bundle on $\Sigma$ can be described by a pair $(V \oplus W,\Phi)$ where $V$ and $W$ are vector bundles on $\Sigma$ of rank $p$ and $q$ respectively such that $\wedge^p V \cong \wedge^q W^*$, and the Higgs field $\Phi$ has block diagonal form
        \begin{equation}
    \Phi = \begin{pmatrix}
        0 & \beta \\
        \gamma & 0
    \end{pmatrix}
\end{equation}
for $\beta \in H^0(\Sigma,\Hom(W,V) \otimes K)$ and $\gamma \in H^0(\Sigma,\Hom(V,W) \otimes K)$. We will also need to consider $\U(p,q)$-Higgs bundles below: these are described similarly, but without the condition on the determinants, and have an associated topological invariant, the \emph{Toledo invariant}, given by
\begin{equation}\label{ToledoInvteqn}
    \tau = 2\frac{q\deg(V)-p\deg(W)}{p+q}.
\end{equation}
The Toledo invariant plays an important role in the moduli theory of $\U(p,q)$-Higgs bundles \cite{BGPG}.

If $(V \oplus W, \Phi)$ is regular, then $N = \Ker(
\gamma)$ is a vector bundle of rank $p-q$. This determines a morphism of stacks
\begin{equation}\label{Decompositionmapeqn}
    \mathcal M^{reg}_K(\SU(p,q)) \rightarrow \mathcal M^{reg}_K(\U(q,q);\tau_{\max}),
\end{equation}
where $\mathcal M_K^{reg}(\U(q,q);\tau_{\max})$ is the locus of $\U(q,q)$-Higgs bundles with Toledo invariant $\tau_{\max} = 2q(g-1)$; this morphism sends $(V \oplus W,\Phi)$ to $(V/N \oplus W, \overline{\Phi})$, where $\overline{\Phi}$ is the induced Higgs field on the quotient. By considering this as a special case of Proposition \ref{GLnAbHitchinprp}, we see that the morphism (\ref{Decompositionmapeqn}) realises the map ${ \Ab}_{\mathbb{R}}$. Moreover the abelianised Hitchin map for $\SU(p,q)$ corresponds to the usual Hitchin map for $\U(q,q)$; spectral data has been calculated in the latter case in \cite{Schaposnik}. The fibres of ${ \Ab}_{\mathbb{R}}$ can be described in terms of the moduli stack of rank $p-q$ vector bundles over $\Sigma$ together with extension data.

There are two subtleties to note here. First, if $p-q=1$, (\ref{Decompositionmapeqn}) can still be defined, but is not the abelianised fibration, since $\SU(q+1,q)$ is quasi-split and the morphism (\ref{Decompositionmapeqn}) is nowhere injective. Secondly, while $\mathcal A_K(\U(q,q)) = \mathfrak{A}_K(\SU(p,q))$, the cover $\mathfrak{ A}_K(\U(q,q)) \rightarrow \mathcal A_K(\U(q,q))$ is non-trivial; however, this does not cause a problem, since fixing the maximal Toledo invariant induces a section $\mathcal A_K(\U(q,q)) \rightarrow \mathfrak{ A}_K(\U(q,q))$ \cite[Section 6]{HM}.

A more detailed consideration of spectral data in this case will appear in \cite{FruPN}.
\end{example}

\begin{example}\label{SOQuatAbnexm}
     Let $G_{\mathbb{R}} = \SO^*(4m+2)$, the quaternionic special orthogonal group on $\mathbb{H}^{2m+1}$. We take a specific matrix form for $G = \SO_{4m+2}$, namely the group of invertible matrices for the bilinear form $(x,y) = x^TJy$, where $J_{ij} = 1$ exactly when $|i-j| = 2m+1$.
    
    The group $H = \GL_{2m+1}$ and the space $\mathfrak{m}$ consists of all matrices of the form 
    \begin{equation}\label{isotropyrepcopy3eqn}
    \begin{pmatrix}
        0 & B\\
        C & 0
    \end{pmatrix}
\end{equation}
      where $B$ and $C$ are skew-symmetric $(2m+1) \tm (2m+1)$ matrices. The Levi $L$ associated to the sheet $S_H$ is isomorphic to a product of $m$ copies of $\GL_2$ together with a copy of $\mathbb{G}_m$. 

    An $\SO^*(4m+2)$-Higgs bundle on $\Sigma$ can be described by a pair $(V \oplus V^*, \Phi)$ where $V$ is a rank $2m+1$ vector bundle and $\Phi$ has block diagonal form
\begin{equation}
    \Phi = \begin{pmatrix}
        0 & \beta \\
        \gamma & 0
    \end{pmatrix}
\end{equation}
for $\beta \in H^0(\Sigma, \wedge^2 V \otimes K)$ and $\gamma \in H^0(\Sigma, \wedge^2 V^* \otimes K)$.

Let $L = \Ker(\gamma)$; since $\gamma$ is skew-symmetric, $L$ has rank at least 1, and if $(V \oplus V^*, \Phi)$ is regular, then $L$ must be a line bundle. Moreover, for any local section $s$ of $V$, $s \in L$ if and only if for all local sections $t$ of $V$
\begin{equation}
    \gamma(s)(t) = -\gamma(t)(s) = 0.
\end{equation}
So $L$ is the annihilator of $\im(\gamma) \otimes K^{-1}$ in $V$ and we have a canonical isomorphism $\im(\gamma) \otimes K^{-1} \cong \overline{V}^*$, where $\overline{V} = V/L$. Thus $\gamma$ descends to a skew-symmetric map $\overline{\gamma}:\overline{V} \rightarrow \overline{V}^* \otimes K$ and $\beta$ automatically determines a skew-symmetric map $\overline{\beta}:\overline{V}^* \rightarrow \overline{V} \otimes K$ by restriction. In this way, we can deconstruct $(V \oplus V^*,\Phi)$ into a line bundle $L$ and an $\SO^*(4m)$-Higgs bundle $(\overline{V} \oplus \overline{V}^*,\overline{\Phi})$. The $\SO^*(4m+2)$-Higgs bundle can be reconstructed from this pair by compatible data for the extension of $\overline{V}$ by $L$ and the extension of $\overline{\beta}$ to $\beta$. By construction, $\overline{\gamma}$ is an isomorphism, so $\overline{V}$ has fixed degree $2m(g-1)$. We note that the apparent breaking of symmetry by choosing $\gamma$ over $\beta$ is resolved by observing from the arguments above that $L$ is canonically isomorphic to $\wedge^{2m+1} V \otimes K^{m}$.

By Proposition \ref{CameralSmoothprp} and \ref{HCameralgpprp} we can calculate that $\mathcal J^H$ is the constant group $\mathbb{G}_m$, and moreover the map
\begin{equation}
    \mathcal M_K^{reg}(\SO^*(4m+2)) \rightarrow {\bf Pic}(\Sigma) \times \mathcal A_K( \SO^*(4m+2))
\end{equation}
which on the first factor sends $(V \oplus V^*, \Phi)$ to $L$ realises the map $ \Ab_{\R}$ and the abelianised fibration. The fibres of $\Ab_{\R}$ are connected components of the $\SO^*(4m)$-Hitchin fibres, for which spectral interpretations have been given by \cite{HS} and \cite{Branco}.
\end{example}

We conclude this section by noting that a version of the abelianised fibration for $\SO(p,q)$, where $p-q>2$, is implicit in the constructions of \cite[Section 3]{BS}; this describes a map from the regular fibres of the $\SO(p,q)$-Hitchin fibration to the regular fibres of the $\Sp(2q,\mathbb{R})$-Hitchin fibration. Comparing this construction with our abelianised fibration is somewhat complicated by the fact that $\mathfrak{A}_K(\SO(p,q))$ is a non-trivial cover of $\mathcal A_K(\SO(p,q))$.

\section{The geometry of sheets and representation theory}\label{AdjointQuotientKatsyloscn}
In this section, we prove auxiliary results on the geometry of sheets, and consider their relationship to the representation theory of $\mathfrak{g}$. We give two alternative descriptions of the Katsylo group, one as a subquotient of the Weyl group $W$, and one in terms of the Grothendieck-Springer theory for the sheet. This gives a way of studying multiplicities attached to the orbits in the sheet. We also calculate a different multiplicity attached to the primitive ideals of the universal enveloping algebra $\mathcal U(\mathfrak{g})$ in the image of Losev's orbit method map using the Katsylo group. We deduce an asymptotic relation between these two multiplicities.

\subsection{The Katsylo group and the Weyl group}\label{KatsyloWeylsbn}
Let $S$ be a non-singular sheet for the action of a reductive algebraic group $G$ on its Lie algebra $\mathfrak{g}$. We choose decomposition data $(L, \mathcal O)$ for $S$ (as in Remark \ref{Decompdatarmk}), and we consider $\mathfrak{z} =\Lie(Z(L))$ with its action of $W_L$.

We use the following construction of \cite{Katsylo}. Let $e \in S$ be nilpotent, and complete it to an $\mathfrak{sl}_2$-triple $(e,h,f)$, and let
$$\mathfrak{K} = (e+\mathfrak{c}_{\mathfrak{g}}(f)) \cap S$$
be the corresponding \emph{Katsylo slice}. The reductive centraliser $A = C_G(<e,h,f>)$ acts on $\fK$ with some cofinite kernel $N$, and as in Definition \ref{KatsyloGpdef} we refer to the group $F = A/N$ as the \emph{Katsylo group}. In this section, we will also describe $F$ as a quotient of $W_L$, and $\fK$ as an affine space.

\begin{lemma}\label{KatsyloMorphismlem}\cite[Lemma 5.1]{Katsylo}
    There is a morphism $\mathcal E:\mathfrak{z} \rightarrow \mathfrak{K}$ such that for all $x \in \mathfrak{z}$, $\mathcal E(x)$ is $G$-conjugate to $x+e$, and $\mathcal E$ is $\mathbb{G}_m$-equivariant with respect to the square of the scaling action on $\mathfrak{z}$ and the Kazhdan action on $\mathfrak{K}$ (see Definition \ref{KazhdanActiondef}).
\end{lemma}

\begin{remark}
    By the construction in \cite[Section 5]{Borho2}, there is a commutative diagram
    \begin{equation}\label{KMPropseqn}
        \begin{tikzcd}
            \mathfrak{z} \arrow[r, "\mathcal E"] \arrow[d, two heads] & \mathfrak{K} \arrow[d, two heads] \\
            \mathfrak{z}/W_L \arrow[r, "\cong"]                       & \mathfrak{K}/F                   
        \end{tikzcd}
    \end{equation}
    where the lower horizontal arrow is the identification noted in Remark \ref{KatsyloGQrmk}. One can deduce that $\mathcal E:\mathfrak{z} \rightarrow \mathfrak{K}$ is finite and surjective (which can also be seen from \cite[Section 6]{Katsylo}).
\end{remark}

This induces a quotient description of the Katsylo group.

\begin{proposition}\label{KatsyloGpWeylGpprp}
    There is a surjective homomorphism $f_{\mathcal E}:W_L \rightarrow F$ such that 
    \begin{equation}\label{KatsyloGpWeylGp1eqn}
        f_{\mathcal E}(w)\mathcal E(x) = \mathcal E(wx)
    \end{equation}
    for all $x \in \mathfrak{z}$ and $w \in W_L$.
\end{proposition}

\begin{proof}
We first define a morphism $g_{\mathcal E}:W_L \times \mathfrak{z} \rightarrow \mathfrak{K}$ by $g_{\mathcal E}(w,x) = \mathcal E(wx)$. Let $\mathfrak{z}^\circ$ be the open subset $\mathcal{E}^{-1}(\mathfrak{K}^\circ) \subseteq \mathfrak{z}$, where $\mathfrak{K}^\circ$ is the locus of $\mathfrak{K}$ on which $F$ acts freely (as in Remark \ref{FInertia1rmk}). Then for any $x \in \mathfrak{z}^\circ$ and $w \in W_L$, since $\mathcal E(wx)$ is in the $F$-orbit of $\mathcal E(x)$, $\mathcal E(wx) \in \mathfrak{K}^\circ$. Hence, $g_{\mathcal E}^{-1}(\mathfrak{K}^\circ) = W_L \times \mathfrak{z}^\circ$, and we consider the restriction of $g_{\mathcal E}$ to $W_L \times \mathfrak{z}^\circ$. For each $a \in F$, we define
\begin{equation}\label{KatsyloGpWeylGp2eqn}
    X_a = \{ (w,x) \in W_L \times \mathfrak{z}^\circ \, | \, g_{\mathcal E}(w,x) = a\mathcal E(x) \},
\end{equation}
which is a closed subset of $W_L \times \mathfrak{z}^\circ$. The $X_a$ are disjoint by definition of $\mathfrak{K}^\circ$, and each $X_a$ is non-empty, since $\mathcal E$ maps the $W_L$-orbit of $x$ surjectively to the $F$-orbit of $\mathcal E(x)$. Hence, the collection $\{X_a\}_{a \in F}$ defines a partition of the connected components of $W_L \times \mathfrak{z}^\circ$. As a result, we can partition $W_L$ into a collection $\{W_{L,a}\}_{a \in F}$ such that $X_a = W_{L,a} \times \mathfrak{z}^\circ$.

Define the map of sets $f_{\mathcal E}:W_L \rightarrow F$ as the map sending each set $W_{L,a}$ to $a$. By construction, (\ref{KatsyloGpWeylGp1eqn}) holds for each $x \in \mathfrak{z}^\circ$ and $w \in W_L$, and since $\mathfrak{z}^\circ$ is dense in $\mathfrak{z}$, (\ref{KatsyloGpWeylGp1eqn}) also holds by continuity for all $x \in \mathfrak{z}$. Since each $X_a$ is non-empty, $f_{\mathcal E}$ is surjective. So it only remains to show that $f_{\mathcal E}$ is a group homomorphism.

Fix $x \in \mathfrak{z}^\circ$, and take $w_1, w_2 \in W_L$. Then we have
$$
f_{\mathcal E}(w_1w_2)\mathcal E(x) = \mathcal E(w_1w_2x) = f_{\mathcal E}(w_1)\mathcal E(w_2x) = f_{\mathcal E}(w_1)f_{\mathcal E}(w_2)\mathcal E(x)
$$
by (\ref{KatsyloGpWeylGp1eqn}), and since $F$ acts freely on $\mathfrak{K}^\circ$, we must have $f_{\mathcal E}(w_1w_2)=f_{\mathcal E}(w_1)f_{\mathcal E}(w_2)$.
\end{proof}

\begin{corollary}\label{KatsyloGpWeylGpcor}
    Let $W_S$ be the kernel of $f_\mathcal E$. The morphism $\mathcal E$ determines an isomorphism $\mathfrak{K} \cong \mathfrak{z}/W_S$ identifying the $F$-action on $\mathfrak{K}$ with the $W_L/W_S$-action on $\mathfrak{z}/W_S$. 
\end{corollary}

\begin{proof}
By construction, $\mathcal E$ is $W_S$-invariant, so it induces a morphism $\mathfrak{z}/W_S \rightarrow \mathfrak{K}$. This morphism is an isomorphism over $\mathfrak{K}^\circ$, and is finite, since it factors through the finite morphism $\mathcal E$; hence by Zariski's main theorem, it is an isomorphism.
\end{proof}

\begin{proposition}\label{WSEIndprp}
    The group $W_S$ depends only on $S$ and its decomposition data $(L, \mathcal O)$ (and not on the choice of $\mathcal E$).
\end{proposition} 

\begin{proof} 
We use the map $p:\mathfrak{z} \rightarrow \mathcal B$ constructed in Lemma \ref{StackyZquotientlem}; we emphasise that this argument is not circular, as the construction of $p$ only uses the existence of the identification $\mathfrak{K} \cong \mathfrak{z}/W_S$ and does not rely on the uniqueness of the subgroup $W_S \leq W_L$.

The map $p$ is ramified exactly where the quotient map $\mathfrak{z} \rightarrow \mathfrak{z}/W_S$ is ramified. Since $W_S$ acts freely on $\mathfrak{z}$ and $\mathfrak{z}/W_S \cong \mathfrak{K}$ is smooth, $W_S$ is generated by reflections by the Shephard-Todd theorem; hence, the ramification locus of the quotient map determines the group $W_S$. Since $p$ is uniquely defined, and in particular does not depend on the choice of $\mathcal E$, the group $W_S$ is also independent of this choice.
\end{proof}

We will sometimes denote $\mathfrak{z}/W_S$ by $\tilde{\mathfrak{c}}_S$. We have the following implications for the Katsylo slice.

\begin{corollary}\label{KatsyloAffinecor}
    The Katsylo slice $\mathfrak{K}$ for a non-singular sheet is isomorphic to an affine space. Moreover, the Kazhdan action admits a square-root; in particular, there are coordinates on $\mathfrak{K}$ such that the Kazhdan action is given by
        \begin{equation}\label{KatsyloAffine1eqn}
        \lambda \cdot (t_1,...,t_r)=(\lambda^{2e_1}t_1, ...,\lambda^{2e_r}t_r)
        \end{equation}
        for some fixed weights $e_i \in \mathbb{N}$.
\end{corollary}

\begin{proof}
The statements follow from the Shephard-Todd theorem, since $\mathfrak{K}$ is isomorphic to a non-singular quotient of a vector space by a linear action of a finite group. The scalar action on $\mathfrak{z}$ induces a $\mathbb{G}_m$-action on $\tilde{\mathfrak{c}}_S$ which in suitable coordinates can be written as
    \begin{equation}\label{KatsyloAffine2eqn}
        \lambda \cdot (t_1,...,t_r)=(\lambda^{e_1}t_1, ...,\lambda^{e_r}t_r)
    \end{equation}
for weights $e_i \in \mathbb{N}$; this defines the required square-root of the Kazhdan action by the equivariance statement in Lemma \ref{KatsyloMorphismlem}.
\end{proof}

\begin{remark}\label{ClassicalKatsylormk}
    If the group $G$ is classical, the content of Corollaries \ref{KatsyloGpWeylGpcor} and \ref{KatsyloAffinecor} already appears in \cite{ImHof}. Indeed, Im Hof's strategy for proving that sheets are non-singular for classical groups is to show by explicit calculation that the morphism $\mathcal E$ is a quotient by a finite reflection group.

    For $G$ exceptional, the fact that $\mathfrak{K}$ is an affine space when $S$ is smooth follows from the calculations of \cite{Bulois}.
\end{remark}

\subsection{Grothendieck-Springer theory for sheets}\label{GSSheetssbn}
There is a generalisation of Grothendieck-Springer theory for sheets developed in \cite{Broer1}, building on work of \cite{BB}. We will use a somewhat different setup from that of \cite{Broer1}, in particular to match the setup of \cite{Losev4} for the applications in Section \ref{Multiplicitiessbn}, and so we prove some results analogous to those of \cite{Broer1} in this alternative setting. We also give a description of the Katsylo group as a quotient of a $G$-centraliser by a $P$-centraliser for a suitable parabolic subgroup $P \leq G$. 

Let $S$ be a sheet and choose decomposition data $(L,\mathcal O)$ for $S$ as in Remark \ref{Decompdatarmk}. Let $\mathfrak{z}$ be the centre of $\mathfrak{l}$ as usual, let $P$ be a parabolic subgroup of $G$ which contains $L$ as a Levi factor, and let $\mathfrak{n}$ be the nilradical of $\mathfrak{p}$. We consider the closed $P$-invariant subvariety $\mathfrak{z} + \overline{\mathcal O} +\mathfrak{n}$ of $\mathfrak{g}$, and the associated $G$-variety
\begin{equation}\label{GeneralisedGSeqn}
    \hat{S} = G \times^P (\mathfrak{z} + \overline{\mathcal O}+\mathfrak{n}).
\end{equation}
By \cite[Lemma 2.2]{Borho2} and \cite[Satz 3.1 (b)]{Borho2}, the $G$-action map
\begin{equation}\label{GeneralisedGSmapeqn}
    \hat{p}:\hat{S} \rightarrow \mathfrak{g}
\end{equation}
maps $\hat{S}$ to the Zariski closure $\overline{S}$ of $S$ in $\mathfrak{g}$. This is the \emph{(generalised) Grothendieck-Springer map} for the sheet $S$.

The following lemma is known (e.g. see \cite[Section 4.1]{Losev4}).

\begin{lemma}\label{BInductionlem}
    For every $z \in \mathfrak{z}$ there is a unique dense $P$-orbit in $z +\overline{\mathcal O} +\mathfrak{n}$, which we denote by $(z+ \overline{\mathcal O} +\mathfrak{n})^{reg}$. Moreover, for every $x \in (z+ \overline{\mathcal O} +\mathfrak{n})^{reg}$, the identity component $C_G^\circ(x)$ is contained in $P$. 
\end{lemma}

If we restrict $\hat{p}$ to the open subset 
\begin{equation}\label{GeneralisedGSRegeqn}
    \hat{S}^{reg} = G \times^P (\mathfrak{z} + \overline{\mathcal O}+\mathfrak{n})^{reg},
\end{equation}
where
\begin{equation}
    (\mathfrak{z} +\overline{\mathcal O} +\mathfrak{n})^{reg} = \bigcup_{z \in \mathfrak{z}}(z +\overline{\mathcal O} +\mathfrak{n})^{reg}
\end{equation}
then the image of $\hat{S}^{reg}$ is the sheet $S$. Moreover, by the construction of $\chi_S$ in \cite[5.1]{Borho2}, there is a commutative diagram
\begin{equation}\label{GSDiagrameqn}
        \begin{tikzcd}
         \hat{S}^{reg} \arrow[r, "\hat{\chi}_S"] \arrow[d, "\hat{p}"] & \mathfrak{z} \arrow[d] \\
            S \arrow[r, "\chi_S"]                                       & \mathfrak{c}_S.              
        \end{tikzcd}
    \end{equation}
\begin{remark}\label{DixmierGSrmk}
    In the case of primary interest for the main body of the work, $S$ is a Dixmier sheet, i.e. $\mathcal O = 0$; in this case $\hat{S} = G \times^P\mathfrak{r}$ and $\hat{S}^{reg} = G \times^P \mathfrak{r}^{reg}$, where $\mathfrak{r}$ is the solvable radical of $\mathfrak{p}$. 
\end{remark}

\begin{lemma}\label{GSDiagramlem}
    The restriction of the generalised Grothendieck-Springer morphism $\hat{p}$ to $\hat{S}^{reg}$ is finite.
\end{lemma}

\begin{proof}
The properness of $\hat{p}$ follows by \cite[7.9]{BK}. 

To see that $\hat{p}$ has finite fibres over $S$, it suffices to observe that, for any $x \in (\mathfrak{z}+\overline{\mathcal O}+\mathfrak{n})^{reg}$, the identity component of the $G$-centraliser $C_G^\circ(x)$ is contained in $P$ and $\Ad(G)(x) \cap (\mathfrak{z}+\overline{\mathcal O} +\mathfrak{n})$ consists of finitely many $P$-orbits by \cite[Lemmata 7.8 \& 7.10]{BK}. The first observation is contained in Lemma \ref{BInductionlem}, and the second can be deduced in the same way as \cite[Zusatz 5.5 (f)]{BK}.
\end{proof}

Unlike the case for the regular sheet, the diagram (\ref{GSDiagrameqn}) is no longer Cartesian in general, but it instead realises the normalisation of the fibre product. To avoid having to rule out the existence of embedded components in $S \times_{\mathfrak{c}_S} \mathfrak{z}$, we take the convention that the normalisation of an irreducible scheme is the normalisation of its reduced subscheme.

\begin{proposition}\label{GSCartesianprp}
    The morphism $\nu:\hat{S}^{reg} \rightarrow S \times_{\mathfrak{c}_S} \mathfrak{z}$ induced by the diagram (\ref{GSDiagrameqn}) is the normalisation map.
\end{proposition}

\begin{proof}
First, we show that the morphism is surjective on $\mathbb{C}$-points. Let $(y,z)$ be a $\mathbb{C}$-point of $S \times_{\mathfrak{c}_S} \mathfrak{z}$, i.e. $y \in S$ and $z \in \mathfrak{z}$ with $\chi_S(y) = W_Lz$. By the construction in \cite[5.1]{Borho2}, we have $y = \Ad_g(x)$ for some $g \in G$ and $x \in (z+\overline{\mathcal O}+\mathfrak{n})^{reg}$; so $(y,z) = \nu(P(g,x))$. Hence $\nu$ is surjective, and in particular $S \times_{\mathfrak{c}_S} \mathfrak{z}$ is irreducible.

Next, we show that $\nu$ is injective on the open set 
\begin{equation}\label{GSRegSemisimpeqn}
    \hat{S}^{rs}=G\times^P(\mathfrak{z}^{rs}+\overline{\mathcal O} +\mathfrak{n})^{reg} = G \times^P\Ad(P)(\mathfrak{z}^{rs}+\mathcal O),
\end{equation}
where $\mathfrak{z}^{rs}$ is the locus of points $z \in \mathfrak{z}$ such that $C_G(z) \leq L$, as in (\ref{WeylLeviActioneqn}). Suppose $(y,z)$ is a point in the image of $\hat{S}^{rs}$, with $\Ad_g(x) = y$ for $g \in G$ and $x \in (\mathfrak{z}^{rs}+\mathcal O)$. Every point in the fibre $\nu^{-1}(y,z)$ is of the form $P(g',x)$ for some $g' \in G$ such that $\Ad_{g'}(x) = y$ by the uniqueness statement in Lemma \ref{BInductionlem}. Then 
\begin{equation}\label{GSCartesian1eqn}
    g^{-1}g' \in C_G(x) \leq C_G(z) = L \leq P,
\end{equation}
so $P(g',x) = P(g,x)$.

The morphism is finite since it factors through the finite morphism $\hat{p}$. Hence, by Zariski's main theorem, $\nu$ is the normalisation provided that $\hat{S}^{reg}$ is normal. Let $\mathfrak{K}$ be a Katsylo slice for $S$ and consider
\begin{equation}\label{GSCartesian2eqn}
    \hat{\mathfrak{K}} =  \hat{S} \times_S \mathfrak{K}.
\end{equation}
By the proof of \cite[Lemma 4.1]{Losev4}, the restriction of $\hat{\chi}_S$ to $\hat{\mathfrak{K}}$ is \'{e}tale, so  $\hat{\mathfrak{K}}$ is non-singular since $\mathfrak{z}$ is. Moreover, $G \times \mathfrak{K} \rightarrow S$ pulls back to a map $G \times \hat{\mathfrak{K}} \rightarrow \hat{S}^{reg}$ which is smooth by Proposition \ref{KatsyloSliceprp}. Hence, $\hat{S}^{reg}$ is non-singular, and this proves the proposition.
\end{proof}

\begin{remark}\label{GSCartesianrmk}
    Proposition \ref{GSCartesianprp}, together with \cite[Corollaries 4.6 \& 5.3(i)]{Broer1}, show that this version of the generalised Grothendieck-Springer map is the same as that of \cite{Broer1} once restricted to $\hat{S}^{reg}$.
\end{remark}

Assume now that $S$ is a non-singular sheet. We can use Proposition \ref{GSCartesianprp} and an alternative description of the normalisation of $S \times_{\mathfrak{c}_S} \mathfrak{z}$ to give another description of the Katsylo group. Fix a nilpotent $e \in (\fz+\oO+\fn)^{reg}$, let $\mathfrak{K}$ be a Katsylo slice to $e$ for $S$, and let $F$ be the corresponding Katsylo group. We recall the $F$-inertia group scheme $\mathcal F$ on $\mathfrak{K}$ defined in Definition \ref{FInertiadef}, with fibres $\mathcal F_y = \Stab_F(y)$.

\begin{theorem}\label{KatsylogpPolarisationthm}
    For any $x \in (\mathfrak{z} +\overline{\mathcal O} +\mathfrak{n})^{reg}$, $C_P(x)$ is a normal subgroup of $C_G(x)$ which depends only on $x$ and $S$ (and not on the parabolic $P$). Moreover, for any $y \in \mathfrak{K}$ and $g \in G$ with $\Ad_g(y) \in (\mathfrak{z} +\overline{\mathcal O} +\mathfrak{n})^{reg}$, there is a canonical isomorphism $\mathcal F_y \cong C_G(y)/C_{g^{-1}Pg}(y)$. In particular, $F$ is isomorphic to $C_G(e)/C_P(e)$.
\end{theorem}

A weaker version of this statement can be deduced from \cite[Proposition 4.7 (3)]{Losev4}, but that in particular does not include the normality of $C_P(x)$ in $C_G(x)$.

The theorem is an immediate corollary of the following characterisation of the \emph{smooth centraliser} $\mathcal I_S^{sm}$ (the maximal smooth group subscheme of the centraliser group scheme over $S$) over the locus $(\mathfrak{z} +\overline{\mathcal O} +\mathfrak{n})^{reg} \subseteq S$ (see Proposition \ref{QuasiSteinHomprp}, Corollary \ref{SmoothCentralisercor} and Proposition \ref{SmCentraliserDescentprp} in Section \ref{SmoothCentralisersbn}).

\begin{proposition}\label{ParabolicCentraliserprp}
    For $x \in (\mathfrak{z} +\overline{\mathcal O} +\mathfrak{n})^{reg}$, $\mathcal I_{S,x}^{sm} = C_P(x)$, as a subgroup of $\mathcal I_x = C_G(x)$.
\end{proposition}

\begin{proof} 
To ease notation, we will denote $U = (\mathfrak{z} +\overline{\mathcal O} +\mathfrak{n})^{reg}$ and $U^{rs} = (\mathfrak{z}^{rs} +\overline{\mathcal O} +\mathfrak{n})^{reg}$, where $\mathfrak{z}^{rs}$ is defined as in (\ref{WeylLeviActioneqn}). We first observe that $\mathcal I_{S,x}^{sm} \leq C_P(x)$, i.e. that there is a factorisation
\begin{equation}\label{ParabolicCentraliser1eqn}
  \begin{tikzcd}
\mathcal I^{sm}_{U} \arrow[r, hook] \arrow[rd, hook] & P\times U \arrow[d, hook] \\
& G\times U                
\end{tikzcd}
\end{equation}
of group schemes; the argument is the same as that of \cite[Lemme 2.4.3]{Ngo2}.

Since $\mathcal I_{S,x}^{sm}$ is a finite index subgroup of $C_G(x)$, to complete the proof of the proposition, it suffices to show that $[C_G(x):C_P(x)] = [C_G(x): \mathcal I_{S,x}^{sm}]$, i.e.  $[C_G(x):C_P(x)] = |\Stab_F(y)|$ for any $y \in \mathfrak{K} \cap \Ad(G)(x)$ by Proposition \ref{QuasiSteinHomprp}. To do this, we consider the pullback of $\nu:\hat{S}^{reg} \rightarrow S \times_{\mathfrak{c}_S} \mathfrak{z}$ under the map $\Ad:G \times \mathfrak{K} \rightarrow S$; using the notation of (\ref{GSCartesian2eqn}), this is the map
\begin{equation}\label{ParabolicCentraliser2eqn}
    \hat{\nu}:G \times\hat{\mathfrak{K}} \rightarrow G \times (\mathfrak{K} \times_{\mathfrak{c}_S} \mathfrak{z}).
\end{equation}
This is a normalisation map, so the restriction $\hat{\nu}|_{\hat{\mathfrak{K}}} = \nu|_{\hat{\mathfrak{K}}}:\hat{\mathfrak{K}} \rightarrow \mathfrak{K} \times_{\mathfrak{c}_S} \mathfrak{z}$ is also a normalisation. We count the number of $\mathbb{C}$-points in the fibres of $\nu$ in two different ways. 

Specifically, let $x \in U$ and $z = \pi_P(x) \in \mathfrak{z}$, and let $y = \mathcal E(z) \in \mathfrak{K}$ where $\mathcal E$ is the morphism defined in Lemma \ref{KatsyloMorphismlem}. Let $g \in G$ be such that $\Ad_g(x) = y$, which exists since $x$ and $y$ map to the same point in $\mathfrak{c}_S$. Consider the $\mathbb{C}$-point $(y,z)$ of $ \mathfrak{K} \times_{\mathfrak{c}_S} \mathfrak{z}$; as in the proof of Proposition \ref{GSCartesianprp}, the $\mathbb{C}$-points in the fibre $\nu^{-1}(y,z)$ are exactly the $P$-orbits in $G\times U$ of the form $P(gk,x)$ for $k \in C_G(x)$. Moreover, two such orbits $P(gk_1,x)$ and $P(gk_2,x)$ coincide exactly when $k_1C_P(x)=k_2C_P(x)$; so $|\nu^{-1}(y,z)| = |C_G(x)/C_P(x)|$.

Now, we observe that $\mathfrak{K} \times_{\mathfrak{c}_S} \mathfrak{z}$ embeds into $\mathfrak{K} \times \mathfrak{z}$ as a closed subscheme, whose underlying reduced subscheme is
\begin{equation}\label{ParabolicCentraliser3eqn}
    ( \mathfrak{K} \times_{\mathfrak{c}_S} \mathfrak{z})_{red} = \bigcup_{a \in F}\mathfrak{z}_a
\end{equation}
where $\mathfrak{z}_a$ is the graph of the morphism $\mathfrak{z} \rightarrow \mathfrak{K}$ given by $z \mapsto a\mathcal E(z).$ But then by uniqueness of normalisation, there is an isomorphism
\begin{equation}\label{ParabolicCentraliser4eqn}
    \hat{\mathfrak{K}} \cong \coprod_{a \in F} \mathfrak{z}_a.
\end{equation}

Under the identifications (\ref{ParabolicCentraliser3eqn}) and (\ref{ParabolicCentraliser4eqn}), the $\mathbb{C}$-point $(y,z)$ of $\mathfrak{K} \times_{\mathfrak{c}_S} \mathfrak{z}$ corresponds to the $\mathbb{C}$-point $(y,z)$ of $\mathfrak{z}_{\text{id}_F}$, and the number of $\mathbb{C}$-points in the fibre $\nu^{-1}(y,z)$ is the number of components $\mathfrak{z}_a$ containing $(y,z)$. But this is exactly $|\Stab_F(y)|$, so we have the required equality $[C_G(x):C_P(x)] = |\Stab_F(x)|$.
\end{proof}

\begin{remark}\label{ClassicalPolrmk}
    If $\mathfrak{g} = \mathfrak{sp}_n$ or $\mathfrak{g}=\mathfrak{so}_n$, the values of $|F|$ for any Dixmier sheet $S$ in $\mathfrak{g}$ can be calculated using Theorem \ref{KatsylogpPolarisationthm} and \cite{Hesselink}. One can verify that, for $\mathfrak{g}= \mathfrak{sp}_{2m}$ or $\mathfrak{g} = \mathfrak{so}_{2m+1}$, $F$ is trivial exactly when $S = S_{\mathfrak{gl}_n} \cap \mathfrak{g}$ for some sheet $S_{\mathfrak{gl}_n}$ in $\mathfrak{gl}_n$. There are exceptions to this for $\mathfrak{g}=\mathfrak{so}_{2m}$, but these can be resolved by considering Dixmier sheets for $\O_{2m}$ instead of $\SO_{2m}$.

    For Dixmier sheets in Dynkin types B and C, the group $\bar{A}_P = C_G(e)/C_P(e)$ (isomorphic to $F$ by Theorem \ref{KatsylogpPolarisationthm}) was considered in \cite{FRW} for its role in mirror symmetry for closures of Richardson orbits.
\end{remark}

We give a first corollary. We observe that the Grothendieck-Springer morphism associates a distinguished nilpotent orbit cover to any sheet $S$.

\begin{definition}\label{SNilpOrbitCoverdef}
    The \emph{$S$-nilpotent orbit cover} $\hO_S$ is the reduced subscheme of the fibre $\hp_S^{-1}(\mO_S)$ of the Grothendieck-Springer map $\hp_S:\hS\ra S$ over the nilpotent orbit $\mO_S$ in $S$.
\end{definition}

The following is then immediate from the discussion in \cite[Section 2.2]{LMM}.

\begin{corollary}\label{SNilpOrbitCovercor}
    The cover $\hO_S \ra \mO_S$ is Galois and the automorphism group $\Aut_{\mO_S}(\hO_S)$ is isomorphic to $F_S$.
\end{corollary}

\begin{remark}\label{SNilpOrbitCoverrmk}
    Nilpotent orbit covers and their automorphism groups play a central role in the theory of unipotent representations of \cite{LMM}. In particular, assuming a certain maximality condition defined in \cite[Section 6.5]{LMM} on the orbit cover $\hO_S$, the representations of $F$ are in bijective correspondence with the unipotent representations of $G$ attached to $\hO_S$ by \cite[Theorem 6.31]{LMM}. As such, it would be interesting to determine what conditions on $S$ are required for this maximality condition on $\hO_S$ to hold.
\end{remark}

The symplectic geometry of $\hO_S$ is also intimately connected to the invariant theory of $S$ by the following corollary which follows from Theorem \ref{KatsylogpPolarisationthm} and \cite[Proposition 4.7]{Losev4}.

Let $X = \Spec(\C[\hO_S])$. As $X$ is a conical symplectic singularity, there is an affine space $\fP_X$, the \emph{Namikawa space}, and a complex reflection group $W_X$ acting on $\fP_X$, the \emph{Namikawa-Weyl group} for $X$, such that $\fP_X/W_X$ parametrises both the filtered Poisson deformations and the filtered quantisations of $X$ (see \cite[Theorem 1.1]{Namikawa}, \cite[Corollary 2.13 \& Proposition 3.5]{Losev4}).

\begin{corollary}\label{PoissonDefcor}
   There are equalities $\fP_X = \fz$ and $W_X=W_S$, where $W_S$ is defined in Corollary \ref{KatsyloGpWeylGpcor}.

   Hence, the Katsylo slice $\fK_e$ parametrises both the filtered Poisson deformations and the filtered quantisations of $X$, and the Katsylo group $F_S$ is isomorphic to the group of $G$-equivariant Poisson automorphisms of $X$.
\end{corollary}

\subsection{Multiplicities of representations}\label{Multiplicitiessbn}
We give two connections between the Katsylo group and multiplicities arising in the representation theory of $G$ and its Lie algebra. As a by-product, we establish a link between two distinct notions of multiplicity. We suppose from now on that $G$ is a connected semisimple group, and we fix a $G$-equivariant isomorphism $\mathfrak{g}^* \cong \mathfrak{g}$ (e.g. via the Killing form). We let $S$ be a non-singular sheet in $\mathfrak{g}$, and $\mathfrak{K}$ be a Katsylo slice for $S$ corresponding to some $\mathfrak{sl}_2$-triple $(e,h,f)$ with $e \in S$ nilpotent.

Our first statement simply rewrites \cite[Theorem 7.2]{BK} using Theorem \ref{KatsylogpPolarisationthm}. We recall the relevant definitions for the statement. We fix a set of simple roots $\Delta \subseteq \mathfrak{t}^*$ for $G$.

 Let $\mathcal O$ be a $G$-orbit in $\mathfrak{g}$. The group $G$ acts on the coordinate ring $\mathbb{C}[\mathcal O]$, and by the reductivity of $G$, there is a direct sum decomposition
\begin{equation}\label{GRedActioneqn}
    \mathbb{C}[\mathcal O] = \bigoplus_{i \in \mathbb{N}} V_i
\end{equation}
as a $G$-module, where each $V_i$ is a finite-dimensional irreducible $G$-module.

\begin{definition}\label{VMultiplicitydef}
    For any finite-dimensional irreducible representation $V$ of $G$, the \emph{multiplicity of $V$ in $\mathbb{C}[\mathcal O]$} is the number of $V_i$ in the decomposition (\ref{GRedActioneqn}) isomorphic to $V$ as a $G$-module. We denote the multiplicity by $m_V(\mathcal O)$.
\end{definition}

By \cite[Proposition 8]{Kostant}, the multiplicity $m_V(\mathcal O)$ is finite for any $V$. Borho defined the following function in \cite{Borho1} which captures the asymptotic information of the multiplicities for a fixed orbit. We let $\Lambda^+$ be the set of dominant weights in the weight lattice $\Lambda \subset \mathfrak{t}^*$ for the choice of simple roots $\Delta$. For any $\lambda \in \Lambda^+$, we can uniquely write $\lambda = n_1\omega_1 + \cdots + n_r \omega_r$ for $n_i \in \mathbb{N}$, where $\omega_1$, ..., $\omega_r$ are the fundamental weights; we denote
$$|\lambda| = n_1 + \cdots + n_r.$$

\begin{definition}\label{Multiplicityfndef}
    The \emph{multiplicity function} $ M:\mathfrak{g}/G \times \mathbb{N} \rightarrow \mathbb{N}$ is defined as
    \begin{equation}\label{Multiplicityfneqn}
        M(\mathcal O;n) = \sum_{\lambda \in \Lambda^+, \,|\lambda|=n} m_{V_{\lambda}}(\mathcal O)
    \end{equation}
    where $V_\lambda$ is the irreducible representation of $G$ with highest weight $\lambda$.
\end{definition}

 The statement of \cite[Theorem 7.2]{BK} gives an asymptotic relationship between multiplicity functions for orbits of the same sheet. By Theorem \ref{KatsylogpPolarisationthm} we can rewrite this in terms of the Katsylo group.

\begin{proposition}\label{KatsyloMultiplicitiesprp}
    Let $\mathfrak{K}$ be a Katsylo slice for a non-singular sheet, and choose points $x,y \in \mathfrak{K}$ with corresponding $G$-orbits $\mathcal O(x)$, $\mathcal O(y)$. Then
    \begin{equation}\label{KatsyloMultiplicities1eqn}
       |\mathcal F_x|M(\mathcal O(x);n) \sim |\mathcal F_y|M(\mathcal O(y);n), 
    \end{equation}
    where $\mathcal F$ is the $F$-inertia group scheme of Definition \ref{FInertiadef}.
\end{proposition}

The second connection involves multiplicities for ideals of the universal enveloping algebra $\mathcal U(\mathfrak{g})$. Let $I \subset \mathcal U(\mathfrak{g})$ be a \emph{primitive ideal}, i.e. the annihilator of a simple module. By the Poincaré-Birkhoff-Witt Theorem, the associated graded algebra $\text{gr}\,\mathcal U(\mathfrak{g})$ is isomorphic to the coordinate algebra $\mathbb{C}[\mathfrak{g}]$. Thus $\text{gr}I$ is an ideal in $\mathbb{C}[\mathfrak{g}]$, and in fact, by \cite[Theorem 3.10]{Joseph}, the radical of $\text{gr}I$ is a prime ideal $Q$ associated to the closure of a nilpotent orbit $\mathcal O^{nil}$ in $\mathfrak{g}$.

\begin{definition}\label{Qmultiplicitydef}
    The \emph{multiplicity of $\overline{\mathcal O^{nil}}$ in $\mathcal U(\mathfrak{g})/I$} is the length of $(\mathbb{C}[\mathfrak{g}]/\text{gr}I)_Q$ as a module over $\mathbb{C}[\mathfrak{g}]_Q$, where the subscript denotes localisation by $Q$. We denote this multiplicity by $\mu(I)$.
\end{definition}

In \cite{Losev4}, Losev defined the \emph{orbit method map}
\begin{equation}\label{OrbitMethodeqn}
    \mathfrak{I}:\mathfrak{g}^*/G \rightarrow \mathscr X_\mathfrak{g}
\end{equation}
where $\mathfrak{g}^*/G$ is the space of coadjoint orbits and $\mathscr X_{\mathfrak{g}}$ is the space of primitive ideals of $\mathcal U(\mathfrak{g})$. Our second statement is the following multiplicity formula for ideals in the image of the orbit method map. We use the same set-up and notation as Proposition \ref{KatsyloMultiplicitiesprp}, and we identify coadjoint orbits with adjoint orbits via the identification $\mathfrak{g} \cong \mathfrak{g}^*$.

\begin{proposition}\label{OrbitMethodMultprp}
    For $x \in \mathfrak{K}$ with $G$-orbit $\mathcal O(x)$,
    \begin{equation}\label{OrbitMethodMulteqn}
        \mu(\mathfrak{I}(\mathcal O(x))) = |F|/|\mathcal F_x|.
    \end{equation}
\end{proposition}

In order to prove Proposition \ref{OrbitMethodMultprp}, we first recall the \emph{birational induction} of \cite[Definition 1.2]{Losev4}. Namely, for every $x \in \mathfrak{g}$, there is a $G$-conjugacy class of triples $(L',\mathcal O', z)$ with $L'$ a Levi subgroup of $G$, $\mathcal O'$ a nilpotent orbit in $\mathfrak{l}'$, and $z$ an element of $\mathfrak{z}' = \mathfrak{z}(\mathfrak{l'})$, satisfying the following property: for any parabolic $P'$ with Levi factor $L'$ (and with $\mathfrak{n}'$ denoting the nilradical of $\mathfrak{p}'$), the generalised Springer map
\begin{equation}\label{GenSpringereqn}
    G \times^{P'}(z+\overline{\mathcal O'} + \mathfrak{n}')^{reg} \rightarrow \mathfrak{g}
\end{equation}
defines an isomorphism onto the $G$-orbit $\mathcal O(x)$ in $\mathfrak{g}$ \cite[Theorem 4.4]{Losev4}. We say that $\mathcal O(x)$ is birationally induced from $(L',\mathcal O', z)$.

\begin{lemma}\label{BirIndlem}
  Suppose the orbit of $x \in \mathfrak{K}$ is birationally induced from $(L',\mathcal O',z)$, and denote $P'$ and $\mathfrak{n}'$ as in (\ref{GenSpringereqn}). Any non-empty fibre of the map
  \begin{equation}\label{BirIndeqn}
    G \times^{P'}(\overline{\mathcal O'} + \mathfrak{n}')^{reg} \rightarrow \mathfrak{g}
    \end{equation}
    is a finite reduced scheme of length $|F|/|\mathcal F_x|$.
\end{lemma}

\begin{proof}
The source of the generalised Springer map is a homogeneous space and the map is $G$-equivariant, so that the fibres are all reduced.

We first show that the statement of the lemma holds if the decomposition class $\mD(L',\mO')$ is contained in the sheet $S$ (in the notation of (\ref{CameralSmooth1eqn})). We denote
\begin{equation}\label{BirInd1eqn}
    Y = G \times^{P'}(\mathfrak{z}'+\overline{\mathcal O'}+\mathfrak{n}')^{reg},
\end{equation}
and consider the generalised Grothendieck-Springer map $\hat{\chi}_Y:Y \rightarrow \mathfrak{g}$.

 By \cite[3.7]{Borho2}, we may choose decomposition data $(L,\mathcal O)$ for $S$ such that $L$ is contained in $L'$; then $\mathfrak{z}'$ is contained in $\mathfrak{z} = \mathfrak{z}(\mathfrak{l})$, so has a natural map to the adjoint quotient space $\mathfrak{c}_S$. To prove the lemma, it will be enough to show that the number of $\mathbb{C}$-points in the fibre $\nu_Y^{-1}(e,0)$ of the induced map $\nu_Y:Y \rightarrow S \times_{\mathfrak{c}_S} \mathfrak{z}'$ is exactly $|F|/|\mathcal F_{x}|$.

The same argument as in Proposition \ref{GSCartesianprp} shows that $\nu_Y$ is a normalisation map of the reduced subscheme of $S \times_{\mathfrak{c}_S} \mathfrak{z}'$. As in the proof of Proposition \ref{ParabolicCentraliserprp}, we consider the restriction ${\nu}_Y :\mathfrak{K} \times_S Y \rightarrow \mathfrak{K} \times_{\mathfrak{c}_S} \mathfrak{z}'$. We have
\begin{equation}\label{BirInd2eqn}
    ( \mathfrak{K} \times_{\mathfrak{c}_S} \mathfrak{z}')_{red} = \bigcup_{a \in F}\mathfrak{z}'_a
\end{equation}
as in (\ref{ParabolicCentraliser3eqn}), but now some of the $\mathfrak{z}'_a$ may coincide. By the assumption that the orbit of $x$ is birationally induced from $(L',\mathcal O',z)$, the normalisation map ${\nu}_Y$ is a bijection over the point $z \in \mathfrak{z}'$, and we can deduce that $\mathfrak{z}'_{a_1} = \mathfrak{z}'_{a_2}$ exactly when $a_1 \mathcal F_{x}  = a_2\mathcal F_{x}$. Thus
\begin{equation}\label{BirInd3eqn}
    \mathfrak{K}_e \times_S Y\cong \coprod_{a \in F/\mathcal F_{x}} \mathfrak{z}'_a,
\end{equation}
and the required statement follows in this case.

In general, let $S'$ be the sheet of $\fg$ containing $\mD(L',\mO')$. Then, by \cite[Satz 3.2]{Borho2}, $x$ and $e$ are both contained in $S'$, and the choice of $\fsl_2$-triple $(e,h,f)$ defining the Katsylo slice for $S$ also defines a Katsylo slice $\fK'_e$ for $S'$ by construction (see Definition \ref{KatsyloSlicedef}); moreover $x\in \fK_e'$. Thus, we have shown already that the length of the generalised Springer fibres for $(L',\mO',0)$ is $|F_{S'}/\mathcal F_{S',x}|$. But note that if $x = x_{ss}+x_n$ is the Jordan decomposition for $x$, then if $M=C_G(x_{ss})$ and $\mO_{\fm} = \Ad_M(x_n)$, $(M,\mO_{\fm},x_{ss})$ birationally induces the orbit of $x$ by \cite[Lemma 2.4]{Borho2}. The decomposition class $\mD(M,\mO_{\fm})$ is contained in both $S$ and $S'$ by Proposition \cite[3.7]{Borho2}, so the length of generalised Springer fibres for $(M,\mO_{\fm}, 0)$ can be calculated as $|F_S/\mathcal F_{S,x}| = |F_{S'}/\mathcal F_{S',x}|$; this proves the lemma.
\end{proof}

\begin{proof}[Proof of Proposition \ref{OrbitMethodMultprp}]
Let $I = \mathfrak{I}(\mathcal O(x))$. From the construction of $\mathfrak{I}$, there is an associative algebra $\mathcal A$ and an inclusion $\mathcal U(\mathfrak{g})/I \hookrightarrow \mathcal A$. Moreover, by \cite[Proposition 3.4.1]{Losev3}, there are algebras $(\mathcal U(\mathfrak{g})/I)_\dagger \hookrightarrow \mathcal A_\dagger$ with compatible actions of the reductive centraliser $A$ of $e$ with respect to the $\mathfrak{sl}_2$-triple $(e,h,f)$, as defined in Definition \ref{ReductiveCentraliserdef}; in fact, these algebras are representations for a finite $\mW$-algebra, but we will not require this structure. These algebras have the following properties:
\begin{itemize}
    \item $\mu(I)$ is equal to the dimension of $(\mathcal U(\mathfrak{g})/I)_\dagger$ as a complex vector space;
    \item $\mathcal A_\dagger$ is $A$-equivariantly isomorphic to the coordinate ring of the fibre of the generalised Springer map (\ref{BirIndeqn}) for a triple $(L',\mathcal O', z)$ which birationally induces $\mathcal O(x)$, by \cite[Lemma 5.2 (1)]{Losev4} and \cite[Section 5.3]{Losev4}.
\end{itemize}
Then, the statement of the proposition follows from Lemma \ref{BirIndlem} once we justify that in fact $(\mathcal U(\mathfrak{g})/I)_\dagger$ coincides with $\mathcal A_\dagger$. By $A$-equivariance, and the structure of $\mathcal A_\dagger$ as the coordinate ring of a finite reduced scheme with a transitive $A$-action, $(\mathcal U(\mathfrak{g})/I)_\dagger$ is of the form $(\mathcal A_\dagger)^B$ for some finite index subgroup $B \leq A$. Thus there is an action of $B$ on $\mathcal A_\dagger$ which fixes $(\mathcal U(\mathfrak{g})/I)_\dagger$, and if the action is trivial, then we are done.

If this action is non-trivial, we can derive a contradiction in the same way as in the proof of \cite[Theorem 5.3]{Losev4}; we sketch the argument. The group $B$ defines a non-trivial action on $\mathcal A$ of $G$-equivariant automorphisms of filtered algebras, which corresponds to a non-trivial action on the coordinate ring of $G \times^{P'}(z+\overline{\mathcal O'} + \mathfrak{n}')^{reg}$ as a filtered Poisson $G$-algebra \cite[Remark 3.24]{Losev4}, \cite[Proposition 4.7]{Losev4}. But since $G$ is semisimple and the generalised Springer map (\ref{GenSpringereqn}) is the unique moment map for the $G$-action, $B$ acts by non-trivial automorphisms on $G \times^{P'}(z+\overline{\mathcal O'} + \mathfrak{n}')^{reg}$ which leave the map (\ref{GenSpringereqn}) invariant; this contradicts the birational induction assumption.
\end{proof}

\begin{remark}\label{Walgebrasrmk}
In the case that $\mathfrak{g}$ is classical, there is a more direct relationship between the Katsylo group and the space $\mathscr{E}_{\mathfrak{g},e}$ of 1-dimensional representations of the \emph{finite $\mW$-algebra} $\mathcal U(\mathfrak{g},e)$ for a nilpotent $e \in S$. The algebra $\mathcal U(\mathfrak{g},e)$ is an associative algebra which is a deformation of the coordinate ring of the Slodowy slice at $e$ \cite{Premet1}. If $\mathfrak{g}$ is classical, $\mathscr{E}_{\mathfrak{g},e}$ decomposes into components $\mathscr{E}_{S}$ labelled by the sheets containing $e$; moreover, for each such sheet, the component $\mathscr{E}_S$ can be degenerated to the Katsylo slice $\mathfrak{K}$ for $S$ at $e$ by \cite[Theorem 1.1]{Topley}. This induces an action of the Katsylo group $F$ on $\mathscr{E}_{S}$ which can be identified with an action on the representations of $\mathcal U(\mathfrak{g},e)$ defined in \cite[Section 4.9]{Premet3}.

There is a map $\mathfrak{J}_e:\mathscr E_{\mathfrak{g},e} \rightarrow \mathscr X_{\mathfrak{g}}$ constructed in Skryabin's appendix to \cite{Premet1}. A consequence of \cite[Theorem 1.2.2]{Losev3} is that the map $\mathfrak{J}_e$ defines a set-theoretic quotient onto its image for the action of $F$ on $\mathscr E_{S}$; moreover, for any $M \in \mathscr E_S$ the multiplicity $\mu(\mathfrak{J}_e(M))$ is equal to $[F:\Stab_F(M)]$ \cite[Theorem 3.1.1]{Losev3}. Thus, if $\mathfrak{g}$ is classical, Proposition \ref{OrbitMethodMultprp} can in this case also be deduced from the results of \cite[Section 9]{Topley}, which in particular state that the image of $\mathfrak{I}$ in $\mathscr X_{\mathfrak{g}}$ is the union of the images of the maps $\mathfrak{J}_e$, where $e$ ranges over a collection of representatives for the nilpotent orbits of $\mathfrak{g}$.
\end{remark}

We finish by observing that Proposition \ref{KatsyloMultiplicitiesprp} and Proposition \ref{OrbitMethodMultprp} together imply the following link between the two notions of multiplicity defined above.

\begin{corollary}\label{QCMultiplicitiescor}
    Let $\mathcal O \in \mathfrak{g}/G$ be any $G$-orbit contained in a non-singular sheet $S$, and let $\mathcal O^{nil}$ be the nilpotent orbit in $S$. Then 
    \begin{equation}\label{QCMultiplicitieseqn}
        \mu(\mathfrak{I}(\mathcal O)) = \lim_{n \rightarrow\infty} \frac{M(\mathcal O;n)}{M(\mathcal O^{nil};n)}.
    \end{equation}
\end{corollary}

\appendix
\section{Dixmier sheets for maximal Levi subgroups}\label{DixmierClassicalscn}
In this appendix, we prove Proposition \ref{CameralSmoothprp} in two special cases. The primary case is that of a Dixmier sheet $S$ associated to a proper Levi subgroup $L$ of a classical group $G$ such that $L$ is maximal, i.e. there is no proper Levi subgroup $M \leq G$ with $L \subsetneq M$. This is a key step in the proof of the general statement, and requires a case-by-case analysis of the maximal Levi subgroups in the groups $G = \PGL_n$, $G=\SO_n$ and $G=\Sp_{2m}$. The second case is a specific example of a Dixmier sheet in $F_4$ which is required for the application to Hitchin fibrations for real forms in Section \ref{RealHitchinsbn}.

Throughout the appendix, for $a,b \in \mathbb{Z}$ we will write $a \equiv b$ to denote $a \equiv b \, ( \text{mod} \, 2)$.

If $G = \PGL_n$, as in Example \ref{GLnSheetexm} we can label any Levi subgroup by a partition ${\bf m}$, which has exactly two parts when the Levi subgroup is maximal. 

If $G = \Sp_{2m}$, any maximal Levi subgroup $L$ is of the form $\GL_a \times \Sp_{2p}$ where $a$ and $p$ are non-negative integers with $a+p =m$; moreover, the Levi subgroups are conjugate exactly when they are isomorphic. Thus we can label the maximal Levi subgroups of $\Sp_{2m}$ by pairs $(a;p)$. Similarly, if $G = \SO_n$, a maximal Levi subgroup $L$ is of the form $\GL_a \times \SO_{q}$ where $2a+q =n$ and $q \neq 2$. These Levi subgroups are not necessarily conjugate under $\SO_n$ when they are isomorphic, but isomorphic Levi subgroups are conjugate under the larger group $\O_n$. We label the maximal Levi subgroups of $\SO_n$ by pairs $(a;q)$.

Let $L$ be a maximal Levi subgroup of $G = \PGL_n$, $\SO_n$ or $\Sp_{2m}$, and let $S$ be the associated Dixmier sheet as in Definition \ref{DixmierSheetdef} and Remark \ref{DixmierSheetrmk}. For $G=\PGL_n$,  the nilpotent orbit in $S$ has Jordan partition ${\bf n}$ conjugate to ${\bf m}$. For $G=\SO_n$ or $G=\Sp_{2m}$, by viewing $G$ as a subgroup of a general linear group under the standard representation, we can still assign a partition to a nilpotent orbit in the same way. In these cases, the relationship between a Levi subgroup and the nilpotent orbit in the corresponding Dixmier sheet is not as simple, but there is still a combinatorial algorithm for calculating it (e.g. see \cite[Lemma 7.3]{Hesselink}).

Table \ref{Levistbl} below outlines the possible cases which can occur, which have been grouped into nine classes according to the properties of the partition ${\bf m}$ (for $\PGL_n$), the label $(a;q)$ (for $G=\SO_n$) or the label $(a;p)$ (for $G=\Sp_{2m}$). The fourth column shows the order of the Katsylo group $F$ for the sheet, defined in Definition \ref{KatsyloGpdef}, which can be calculated using Theorem \ref{KatsylogpPolarisationthm} and \cite{Hesselink}. The order of $W_L$ can be calculated by inspection (noting in particular for Class VI that the usual Weyl group $W$ for type $D_n$, $n$ odd, contains no element acting by $-1$ on the centre $\mathfrak{z}$ of $\mathfrak{l}$).

\begin{table}[ht]
\centering
\caption{Summary of maximal Levi subgroups in classical groups}
\begin{tabular}{|c|c|c|c|c|}
\hline
$G$ & Levi subgroup & Nilpotent orbit in $S$ & $|F|$ & $|W_L|$ \\
\hline
\multirow{2}{*}{$\PGL_n$} & I. $m_1=m_2$ & $(2^{m_1})$ & 1 & 2 \\
& II. $m_1 > m_2$ & $(2^{m_2},1^{m_1-m_2})$ & 1 & 1 \\
\hline
\multirow{3}{*}{$\SO_n$} & III. $a \geq q>0$, $a \not \equiv q$ & $(3^q, 2^{a-q-1},1^2)$ & 2 & 2 \\
&IV. $a \geq q$, $a \equiv q$& $(3^{q},2^{a-q})$ & 1 & 2 \\
& V. $a < q$ & $(3^a,1^{q-a}) $& 1 & 2\\
\hline
$\SO_{2m}$&VI. $q = 0$, $a \not \equiv 0$& $(2^{a-1},1^2)$ & 1 & 1 \\
\hline
\multirow{3}{*}{$\Sp_{2m}$} &VII. $a \geq 2p$ & $(3^{2p},2^{a-2p})$  & 1 & 2\\
&VIII. $a < 2p$, $a \not \equiv 0$ & $(3^{a-1},2^2,1^{2p-a-1})$ & 2 & 2 \\
&IX. $a < 2p$, $a \equiv 0$ & $(3^a,1^{2p-a})$ & 1 & 2 \\
\hline
\end{tabular}
\label{Levistbl}
\end{table}

We group these classes into two broader types according to the ramification of the map $p:\mathfrak{z} \rightarrow \mathcal B$ of Lemma \ref{StackyZquotientlem}, which can be detected by the value of $|W_S| = |W_L|/|F|$, where $W_S$ is the group defined in Corollary \ref{KatsyloGpWeylGpcor}. If $|W_S| = 1$, then $p$ is unramified; we call this Type 1, and it includes Classes II, III, VI and VIII in Table \ref{Levistbl}. If $|W_S|=2$, then $p$ is ramified exactly at $0 \in \mathfrak{z}$; we call this Type 2, and this covers the remaining classes. We will need the following lemma.

\begin{lemma}\label{KTAblem}
For each of the classes in Table \ref{Levistbl}, there is a choice of Levi subgroup $L \leq G$, a parabolic subgroup $P \leq G$ with Levi factor $L$, and an $\mathfrak{sl}_2$-triple $(e,h,f)$ with the following properties.
\begin{itemize}
    \item[(i)] The semisimple element $h$ is contained in $\mathfrak{l}$ and the nilpotent $e$ is contained in $\mathfrak{n}^{reg} = \mathfrak{n} \cap S$, where $\mathfrak{n}$ is the nilradical of $\mathfrak{p}$.
    \item[(ii)] The semisimple element $h$ is not in the kernel of the abelianisation map $\Lie(L)=\mathfrak{l} \rightarrow \mathfrak{z}$.
    \item[(iii)] If $L$ is of Type 1, there is an element $h' \in \mathfrak{c}_{\mathfrak{g}}(e) \cap \mathfrak{l}$ such that $h'$ is not in the kernel of $\mathfrak{l} \rightarrow \mathfrak{z}$. Here, $\mathfrak{c}_{\mathfrak{g}}(e)$ denotes the Lie algebra centraliser of $e$.
\end{itemize}
\end{lemma}

\begin{proof}
We will deal with the $\PGL_n$ classes separately.

We take the vector space $\mathbb{C}^n$ with standard basis $v_1$, ..., $v_n$, and identify $\mathbb{C}^{m_1}$ with the subspace with basis $v_1$, ..., $v_{m_1}$ and $\mathbb{C}^{m_2}$ with the subspace with basis $v_{m_1+1}$, ..., $v_n$. We identify $\PGL_n$ with $\PGL(\mathbb{C}^n)$ and $\mathfrak{pgl}_n$ with $\End(\mathbb{C}^n)/\C I_n$, and take a representative $L = P(\GL(\mathbb{C}^{m_1}) \times \GL(\mathbb{C}^{m_2}))$ for the Levi subgroup $L$ of $\GL_n$ associated with $(m_1 \geq m_2)$. The space $\Hom(\mathbb{C}^{m_2},\mathbb{C}^{m_1})$ can be identified via the quotient map with the nilradical $\fn$ of a parabolic subalgebra $\mathfrak{p}$ corresponding to a parabolic subgroup $P$ with Levi factor $L$. The abelianisation map $\mathfrak{l} \rightarrow \mathfrak{z} \cong \mathbb{C}$ is given by
\begin{equation}
    x_1\oplus x_2 \oplus \C I_n \mapsto m_2\text{tr}(x_1) - m_1\text{tr}(x_2),
\end{equation}
where $x_1 \in \End(\mathbb{C}^{m_1})$ and $x_2 \in \End(\mathbb{C}^{m_2})$.

The nilpotent $e$ given by
\begin{equation}
    ev_j = 0 \text{ if } j \leq m_1, \, ev_j = v_{j-m_1} \text{ if } j >m_1
\end{equation}
has $e \in \Hom(\mathbb{C}^{m_2},\mathbb{C}^{m_1})$ and is of the correct Jordan type; hence $e \in \mathfrak{n}^{reg}$. We define $h \in \mathfrak{l}$ by $h = \tilde{h}+\C I_n$ for 
\begin{itemize}
    \item $\tilde{h}v_j = v_j$ for $1 \leq j \leq m_2$,
    \item $\tilde{h}v_j = 0$ for $m_2+1 \leq j \leq m_1$,
    \item and $\tilde{h}v_j = -v_j$ for $m_1+1 \leq j \leq n$.
\end{itemize}
We define the nilpotent $f \in \fn^{op}$ by
$$ fv_j = v_{j+m_1} \text{ if } j \leq m_2, \, fv_j = 0 \text{ if } j >m_2. $$
By inspection, $(e,h,f)$ is an $\mathfrak{sl}_2$-triple, and $h$ maps to $m_2n \neq 0$ under the abelianisation map. Hence, this proves statements (i) and (ii) in these cases.

If $L$ is of Type 1, i.e. $m_1 > m_2$, we can define $h' \in \mathfrak{c}_{\mathfrak{g}}(e) \cap \mathfrak{l}$ by $h' =\tilde{h}'+\C I_n$, where  $\tilde{h}'v_j = \delta_{m_2+1,j}v_j$; this is sent to $m_2$ under the abelianisation, so statement (iii) holds in this case.

We now assume $G = \SO_n$ or $G=\Sp_n$ (where $n=2m$); we will use the construction in \cite[Lemma 7.3]{Hesselink}. Specifically, let $V$ be a vector space of dimension $n$ with a non-degenerate bilinear form $(.,.)$, which is symmetric if $G=\SO_n$ and anti-symmetric if $G=\Sp_{2n}$, and identify $G$ with the group of linear automorphisms respecting the bilinear form. Fix a maximal Levi subgroup $L$ of $G$ corresponding to $(a;p)$ if $G=\Sp_{2m}$ or $(a;q)$ if $G=\SO_n$. There is a decomposition of $V$ as
\begin{equation}\label{MLCameralHom3eqn}
    V = V^+\oplus W \oplus V^-
\end{equation}
with the following properties:
\begin{itemize}
    \item The vector spaces $V^\pm$ have dimension $a$ and are isotropic with respect to $(.,.)$.
    \item The dimension of $W$ is $2p$ if $G=\Sp_{2n}$, $q$ if $G=\SO_n$.
    \item The restrictions of the bilinear form to $V^+ \oplus V^-$ and $W$ are non-degenerate.
    \item $L$ is the subgroup of $G$ of linear automorphisms respecting the decomposition (\ref{MLCameralHom3eqn}).
\end{itemize}

Let ${\bf n} = (n_1 \geq \,  ...\, \geq n_s)$ of $n$ be the partition for the nilpotent orbit in the sheet corresponding to $L$ (i.e. the partition in the third column of Table \ref{Levistbl}). We choose a basis $v_{i,j}$ of $V$, for $1 \leq j \leq s$ and $1\leq i \leq n_j$, and an involution $\beta$ on the set of indices $\{1,\, ..., s\}$ of the partition ${\bf n}$ with the following properties:
\begin{itemize}
    \item For all $1 \leq j \leq s$, $n_{\beta(j)} =n_j$.
    \item The basis vectors satisfy $(v_{i,j},v_{i',j'})=-(v_{i+1,j},v_{{i-1},j})$ for all $1\leq j, j'\leq s$ and $1\leq i,i' \leq n_j$; moreover, $(v_{i,j},v_{i',j'}) \neq 0$ exactly when $i+i' = n_j+1$ and $j = \beta(j')$.
    \item The vectors $v_{1,j}$ for $1 \leq j \leq a$ form a basis for $V^+$, the vectors $v_{n_j, \beta(j)}$ for $1 \leq j \leq a$ form a basis for $V^-$, and the remaining $v_{i,j}$ form a basis for $W$.
    \item If $L$ is of Type 1, $\beta(a) = a+1$.
\end{itemize}
This is possible by the properties of the decomposition (\ref{MLCameralHom3eqn}) and by the form of the partitions which occur in Table \ref{Levistbl}.

Then the endomorphism $e$ of $V$ which acts by $ev_{i,j} = v_{i-1,j}$ if $i >1$, and $ev_{1,j} = 0$, for all $1 \leq j \leq s$, defines a nilpotent element of $\mathfrak{g}$; indeed, the construction ensures that $v_{i,j}$ is a \emph{normalised Jordan basis} for $e$ \cite[5.1]{Hesselink}. Moreover, by the construction of $v_{i,j}$ and the form of the partition ${\bf n}$, $e$ fixes the flag
\begin{equation}\label{MLCameralHom4eqn}
    0 \subseteq V^+ \subseteq V^+ \oplus W \subseteq V;
\end{equation}
the subgroup of $G$ fixing the flag (\ref{MLCameralHom4eqn}) is a parabolic $P$ with Levi factor $L$. We see that $e \in \mathfrak{n}^{reg}$, where $\mathfrak{n}$ is the nilradical of $\mathfrak{p}$, as required.

Define the endomorphism $h$ by $hv_{i,j} = (n_j -2i+1)v_{i,j}$ for all $1 \leq j \leq s$ and $1 \leq i \leq n_j$; then $[h,e] = 2e$ by inspection. By the construction of the basis, $h \in \mathfrak{g}$; moreover it is clear that $h \in \mathfrak{l}$. The pair $(e,h)$ can be extended to an $\mathfrak{sl}_2$-triple e.g. by the calculations of \cite[Section 5.2]{CMcG}.

The abelianisation map for $\mathfrak{l}$ in this case sends an endomorphism $x \in \mathfrak{l}$ to $\text{tr}(x|_{V^+})$. Under this map, $h$ is sent to 
$$\sum_{j=1}^a (n_j-1);$$
each of the terms in the sum is greater than or equal to $0$, and since $n_1 > 1$, the sum must be strictly positive. Hence, $h$ does not map to $0$ under the abelianisation map; so statements (i) and (ii) are proven for these classes also.

If $L$ is of Type 1, then we define the endomorphism $h'$ by
\begin{itemize}
    \item $h'v_{i,a} = v_{i,a}$ for $1 \leq i \leq n_a$,
    \item $h'v_{i,a+1}=-v_{i,a+1}$ for $1 \leq i \leq n_{a+1}$,
    \item and $h'v_{i,j} = 0$ otherwise.
\end{itemize}
   This endomorphism centralises $e$ and is contained in the Levi subalgebra $\mathfrak{l}$. The abelianisation map sends $h'$ to $1$; hence this proves the statement (iii) in these classes.
\end{proof}

Now let $G$ be an arbitrary classical group, and let $S$ be the Dixmier sheet associated to a maximal Levi subgroup $L \leq G$. Let $\kappa_S:\mathcal I_S^{sm} \rightarrow \rho_S^*\hat{\mathcal J_S}$ be the cameral homomorphism as defined in Section \ref{Camgpsbn}.

\begin{proposition}\label{MLCameralSmoothprp}
    The cameral homomorphism $\kappa_S:\mathcal I_S^{sm} \rightarrow \rho_S^*\hat{\mathcal J_S}$ is smooth (where $S$ is a Dixmier sheet associated to a maximal proper Levi subgroup of $G$).
\end{proposition}

\begin{proof}
Since the statement of Proposition \ref{MLCameralSmoothprp} is equivalent to checking that the induced map on vector bundles $\kappa_{S, \Lie}:\Lie(\mathcal I_S^{sm}) \rightarrow \Lie(\rho_S^*\hat{\mathcal J_S})$ is surjective, it is not affected by altering the centre of $G$ (by extensions or quotients) or by splitting $G$ into direct factors. In particular, we can reduce to checking the statement for the cases $G = \PGL_n$, $G = \SO_n$ or $G=\Sp_{2m}$. Moreover, by Remark \ref{CameralHomrmk}, for the statement of Proposition \ref{MLCameralSmoothprp} we do not need to distinguish between Levi subgroups which are related by an automorphism of $G$. Thus, it suffices to prove the proposition for any representative of each of the classes in Table \ref{Levistbl}. We let $L$ and $P$ be the choices of Levi and parabolic subgroup of $G$ in Lemma \ref{KTAblem} and let $\mathfrak{r}$ be the solvable radical of $\mathfrak{p} = \Lie(P)$. We let $(e,h,f)$ be the $\mathfrak{sl}_2$-triple determined by Lemma \ref{KTAblem}.

In the notation of Section \ref{Camgpsbn}, we have $\bar{Z} \cong \mathbb{C}^*$ and $\mathfrak{z} \cong \Lie(\bar{Z}) \cong \mathbb{C}$. The group $W_L$ has order $1$ or $2$, and in the latter case its non-trivial element acts on $\bar{Z}$ by inversion. Any element of the Dixmier sheet associated to $L$ is either semisimple or nilpotent.

By $G$-equivariance it suffices to show that $(\kappa_{S,\Lie})_x$ is surjective for each $x \in \mathfrak{r}^{reg}$; moreover, this holds if $x$ is semisimple by Proposition \ref{CameralHomprp}, so we may assume that $x = e$.

We first assume that $L$ is of Type 1. Since $p$ is unramified, by the construction of $\kappa_S$ in Proposition \ref{CameralHomprp}, the pullback $\hat{p}$ is unramified over $e$, so the fibre $\rho_S^*\hat{\mathcal J}_{S,x}$ can be identified with $\bar{Z}$ such that the homomorphism $\kappa_{S,e}:C_P(e) \rightarrow \bar{Z}$ is given by
\begin{equation}\label{ML1CameralHomeqn}
\begin{tikzcd}
C_P(e) \arrow[r] & C_P(e)P^{der}/P^{der} \arrow[r, hook] & P/P^{der} = \bar{Z}. 
\end{tikzcd}
\end{equation}
Since $C_P(e)P^{der}/P^{der}$ is isomorphic to the image of $C_P(e) \cap L$ under the abelianisation of $L$, and $C_P(e)^\circ = C_G(e)^\circ$ \cite[Zusatz 5.5 (d)]{BK}, $(\kappa_{S,\Lie})_e$ is surjective by Lemma \ref{KTAblem} (iii). 

Now suppose $L$ is of Type 2. Let $\mathfrak{s}$ be the copy of $\mathfrak{sl}_2$ generated by $(e,h,f)$, with fixed Cartan subalgebra $\mathfrak{t} = \mathbb{C}h$ and Borel subalgebra $\mathfrak{b} = \mathbb{C}h \oplus \mathbb{C}e$. We consider it as the Lie algebra of $\SL_2$, with corresponding Cartan subgroup $T$ and Borel subgroup $B$, although the embedding $\mathfrak{s} \hookrightarrow\mathfrak{g}$ may not arise from any embedding $\SL_2 \hookrightarrow G$, and the maps of Lie algebra bundles we construct below may not arise from group homomorphisms. Let $W_{\mathfrak{s}}$ be the Weyl group on $\mathfrak{t}$, which is the group generated by $-\text{id}_{\mathfrak{t}}$. 

By Lemma \ref{KTAblem} (ii), the map $\zeta:\mathfrak{t} \rightarrow \mathfrak{z}$ defined by sending $h$ to its image under the abelianisation is an isomorphism of vector spaces; moreover this map identifies the $W_{\mathfrak{s}}$-action on $\mathfrak{t}$ with the $W_L$-action on $\mathfrak{z}$. Thus, $\zeta$ induces an isomorphism of vector bundles
\begin{equation}\label{MLCameralSMooth4eqn}
    \zeta_J:\Lie(\hat{\mathcal J}_{\mathfrak{s}}) \rightarrow \Lie(\hat{\mathcal J}_S)
\end{equation}
over $\mathfrak{t}/W_{\mathfrak{s}} \cong \mathfrak{z}/W_L \cong \mathcal B$; here $\hat{\mathcal J}_{\mathfrak{s}}$ is the pseudo-cameral group for the regular sheet $\mathfrak{s}^{reg}$ as defined in Definition \ref{PseudoCameralgpdef}.

We also have a map $\zeta_{\mathfrak{b}}:\mathfrak{b}^{reg} \hookrightarrow \mathfrak{r}^{reg} \subseteq S$ induced by $\zeta$, since $\mathfrak{b}^{reg} = (\mathfrak{t} \oplus \mathbb{C}e) \backslash\{0\}$. For any $y \in \mathfrak{b}^{reg}$, the centraliser $\mathfrak{c}_{\mathfrak{s}}(y)$ is generated by $y$ as a vector space, so there is a map of vector bundles
\begin{equation}\label{MLCameralSmooth5eqn}
    \zeta_I:\Lie(\mathcal I_{\mathfrak{s}}^{reg})|_{\mathfrak{b}^{reg}} \rightarrow \Lie(\zeta_{\mathfrak{b}}^*\mathcal I_S^{sm})
\end{equation}
which fibrewise is given by the map $\zeta_{I,y}:\mathfrak{c}_\mathfrak{s}(y) \rightarrow \mathfrak{c}_{\mathfrak{g}}(\zeta_{\mathfrak{b}}(y))$ sending $y$ to $\zeta_{\mathfrak{b}}(y)$.

We now show that there is a commutative diagram of vector bundles over $\mathfrak{b}^{reg}$ given by 
\begin{equation}\label{MLCameralSmooth6eqn}
\begin{tikzcd}
\Lie(\mathcal I_{\mathfrak{s}}^{reg}) \arrow[d, "{\kappa_{\mathfrak{s},\Lie}}"'] \arrow[r, "\zeta_I"] & \Lie(\zeta_{\mathfrak{b}}^*\mathcal I_S^{sm}) \arrow[d, "\zeta_{\mathfrak{b}}^*\kappa_{S, \Lie}"] \\
\Lie(\chi_{\mathfrak{s}}^*\hat{\mathcal J}_{\mathfrak{s}}) \arrow[r, "\chi_{\mathfrak{s}}^*\zeta_J"] & \Lie( \chi_{\mathfrak{s}}^*\hat{\mathcal J_S}),                                                 
\end{tikzcd}
\end{equation}
where $\kappa_{\mathfrak{s}}$ is the cameral homomorphism for the regular sheet $\mathfrak{s}^{reg}$, and $\chi_{\mathfrak{s}}$ is the Chevalley map for $\mathfrak{s}$. We first observe that the diagram makes sense since $\chi_{\mathfrak{s}} = \chi_S \circ \zeta_{\mathfrak{b}}$ by \cite[Satz 5.6]{Borho2}. To see that it commutes, it suffices to do so over the dense open subset $\mathfrak{b}^{rs}$ of regular semisimple elements of $\mathfrak{b}$.

Since the quotient map $p_{\mathfrak{s}}:\mathfrak{t} \rightarrow \mathfrak{t}/W_{\mathfrak{s}}$ is unramified on the regular semisimple locus $\mathfrak{t}^{rs}$, to check the diagram (\ref{MLCameralSmooth6eqn}) commutes over $\mathfrak{b}^{rs}$, it suffices to check that the diagram
\begin{equation}\label{MLCameralSmooth7eqn}
\begin{tikzcd}
\mathfrak{c}_{\mathfrak{s}}(y) \arrow[d, "{(\hat \kappa_{\mathfrak{s},\Lie})_y}"'] \arrow[r, "{\zeta_{I,y}}"] &  \mathfrak{c}_{\mathfrak{g}}(\zeta_{\mathfrak{b}}(y)) \arrow[d, "(\hat\kappa_{S, \Lie})_{\zeta_{\mathfrak{b}}(y)}"] \\
\mathfrak{t} \arrow[r, "\zeta"]                                                                                                        & \mathfrak{z}                                                                                                               
\end{tikzcd}
\end{equation}
commutes for each $y \in \mathfrak{b}^{rs}$, where $\hat{\kappa}_S$ and $\hat{\kappa}_{\mathfrak{s}}$ are defined as in the proof of Proposition \ref{CameralHomprp}.  By construction, both maps send the generator $y$ of $\mathfrak{c}_{\mathfrak{s}}(y)$ to $\zeta(\pi_{\mathfrak{b}}(y))$, where $\pi_{\mathfrak{b}}:\mathfrak{b} \rightarrow \mathfrak{t}$ is the quotient by the nilradical.

So the diagram (\ref{MLCameralSmooth6eqn}) commutes; and $\zeta_J$ is an isomorphism by construction, while $\kappa_{\mathfrak{s},\Lie}$ is an isomorphism by \cite[Proposition 12.5]{DG}. Hence $\zeta_{\mathfrak{b}}^*\kappa_{S, \Lie}$ must be surjective, so $(\kappa_{S,\Lie})_x$ is surjective for all $x \in \zeta_{\mathfrak{b}}(\mathfrak{b}^{reg}) =(\mathfrak{z} \oplus \mathbb{C}e)\backslash\{0\}$. In particular $(\kappa_{S,\Lie})_e$ is surjective.
\end{proof}

We can also give an ad hoc proof for the following special case, which is the only example of a sheet not of classical reduction type containing the regular locus of a simple symmetric space.

\begin{proposition}\label{F4CameralSmoothprp}
    The cameral homomorphism $\kappa_S:\mathcal I_S^{sm}\rightarrow \rho_S^*\hat{\mathcal J}$ is smooth when $S$ is the Dixmier sheet associated to the Levi subgroup $L$ of type $B_3$ in the exceptional group $\F_4$.
\end{proposition}

\begin{proof}
We note first that the sheet $S$ is non-singular, so that the cameral homomorphism is well-defined \cite{Bulois}. The nilpotent orbit in $S$ is the orbit with Bala-Carter label $\tilde{A}_2$, and has trivial component group in $\F_4$ (see e.g. \cite[Theorem 7.1.6 and Section 8.4]{CMcG}). In particular, the Katsylo group is trivial. 

On the other hand the group $W_L$ is of order 2 \cite{Howlett}. Since $S$ corresponds to a maximal Levi subgroup in $\F_4$, the same proof as in Proposition \ref{MLCameralSmoothprp} (for Type 2) applies provided we can find an $\mathfrak{sl}_2$-triple $(e,h,f)$ satisfying properties (i) and (ii) in Lemma \ref{KTAblem}. But the form of the weighted Dynkin diagram for $\tilde{A}_2$ (with labels $0$ at each node in the subdiagram $B_3$, and $2$ at the remaining node) determines an $\mathfrak{sl}_2$-triple $(e,h,f)$ such that $L = C_G(h)$ and $e$ is a representative for $\tilde{A}_2$. This implies the required properties in Lemma \ref{KTAblem}, and thus proves the statement.
\end{proof}

\bibliographystyle{alpha}
\bibliography{bibliography}

@article{AOV,
    author = {D. Abramovich and M. Olsson and A. Vistoli},
    title = {Tame stacks in positive characteristic},
    journal = {Ann. Inst. Fourier},
    volume = {58},
    number = {4},
    pages = {1057-1091},
    year = {2008}
}

@article{Araki,
    author = {S. Araki},
    title = {On root systems and an infinitesimal classification of irreducible symmetric spaces},
    journal = {J. Math. Osaka City Univ.},
    volume = {13},
    number = {1},
    pages = {1-34},
    year = {1962}
}

@misc{BR,
    author = {K. Banerjee and S. Rayan},
    title = {A generalized spectral correspondence},
    howpublished = {arXiv preprint},
    note = {arXiv:2310.02413},
    year = {2023}
}

@article{BS,
    author = {D. Baraglia and L.P. Schaposnik},
    title = {Cayley and {L}anglands type correspondences for orthogonal {H}iggs bundles},
    journal = {Trans. Amer. Math. Soc.},
    volume = {371},
    pages = {7451-7492},
    year = {2019}
}

@article{BNR,
    author = {A. Beauville and M.S. Narasimhan and S. Ramanan},
    title = {Spectral curves and the generalised theta divisor},
    journal = {J. reine angew. Math.},
    volume = {398},
    pages = {169-179},
    year = {1989}
}

@misc{BD,
    author = {A. Beilinson and V. Drinfeld},
    title = {Quantization of {H}itchin's integrable system and {H}ecke eigensheaves},
    howpublished = {unpublished manuscript},
    note = {\href{https://math.uchicago.edu/~drinfeld/langlands/QuantizationHitchin.pdf}{https://math.uchicago.edu/~drinfeld/langlands/QuantizationHitchin.pdf}},
    year = {1991}
}

@article{BiR,
    author = {I. Biswas and S. Ramanan},
    title = {An infinitesimal study of the moduli of {H}itchin pairs},
    journal = {J. London Math. Soc.},
    volume = {49},
    number = {2},
    pages = {219-231},
    year = {1994}
}

@article{Borho1,
    author = {W. Borho},
    title = {Definition einer {D}ixmier-{A}bbildung für $\mathfrak{sl}(n,\mathbb{C})$},
    journal = {Invent. Math.},
    volume = {40},
    pages = {143-169},
    year = {1977}
}

@article{Borho2,
    author = {W. Borho},
    title = {\"{U}ber {S}chichten halbeinfacher {L}ie-{A}lgebren},
    journal = {Invent. Math.},
    volume = {65},
    pages = {283-317},
    year = {1981}
}

@article{BB,
    author = {W. Borho and J-L. Brylinski},
    title = {Differential operators on homogeneous spaces. {I}. {I}rreducibility of the associated variety for annihilators},
    journal = {Invent. Math.},
    volume = {69},
    pages = {437-476},
    year = {1982}
}

@article{BK,
    author = {W. Borho and H. Kraft},
    title = {\"{U}ber {B}ahnen und deren {D}eformationen bei linearen {A}ktionen reduktiver {G}ruppen},
    journal = {Comment. Math. Helv.},
    volume = {54},
    number = {1},
    pages = {61 - 104},
    year = {1979}
}

@book{BLR,
    author = {S. Bosch and W. L\"{u}tkebohmert and M. Raynaud},
    title = {N\'{e}ron Models},
    series = {Ergebnisse der Mathematik und ihrer Grenzgebiete 3. Folge/A Series of Modern Surveys in Mathematics},
    volume = {21},
    publisher = {Springer-Verlag},
    address = {Berlin, Heidelberg, DE},
    year = {1990}
}

@article{BGPG,
    author = {S.B. Bradlow and O. Garc\'{i}a-Prada and P.B. Gothen},
    title = {Surface group representations and {U}$(p,q)$-{H}iggs bundles},
    journal = {J. Differential Geom.},
    volume = {64},
    number = {1},
    pages = {111-170},
    year = {2003}
}

@phdthesis{Branco,
    author = {L. Branco},
    title = {{H}iggs bundles, {L}agrangians and mirror symmetry},
    school = {University of Oxford},
    year = {2017}
}

@article{Broer1,
    author = {A. Broer},
    title = {Decomposition varieties in semisimple {L}ie algebras},
    journal = {Can. J. Math.},
    volume = {50},
    number = {5},
    pages = {929-971},
    year = {1998}
}

@unpublished{Bulois,
    author = {M. Bulois},
    title = {Geometry of sheets in ordinary and symmetric {L}ie algebras},
    note = {In preparation}
}

@incollection{BIW,
    author = {M. Burger and A. Iozzi and A. Wienhard},
    title = {Higher {T}eichm\"{u}ller Spaces: from ${S}{L}(2,\mathbb{R})$ to other {L}ie groups},
    editor = {A. Papadopoulos},
    booktitle = {Handbook of Teichm\"{u}ller Theory, Volume IV},
    series = {IRMA Lectures in Mathematics and Theoretical Physics},
    volume = {19},
    publisher = {EMS Press},
    address = {Z\"{u}rich, CH},
    pages = {539-618},
    year = {2014}
}

@phdthesis{Carbone,
    author = {R.M. Carbone},
    title = {The {N}orm map on the compactified {J}acobian, the {P}rym stack and spectral data for {G}-Higgs pairs},
    school = {Università Roma Tre},
    year = {2019}
}

@article{deCataldo,
    author = {M.A.A. de Cataldo},
    title = {A support theorem for the {H}itchin fibration: the case of ${S}{L}_n$},
    journal = {Compos. Math.},
    volume = {153},
    number = {6},
    pages = {1316-1347},
    year = {2017}
}

@misc{dCFFHM,
    author = {M.A.A. de Cataldo and R. Fringuelli and A. {Fernandez Herrero} and M. Mauri},
    title = {Hitchin fibrations are {N}gô fibrations},
    howpublished = {arXiv preprint},
    note = {arXiv:2502.04966},
    year = {2025}
}

@book{CMcG,
    author = {D.H. Collingwood and W.M. McGovern},
    title = {Nilpotent Orbits in Semisimple Lie Algebras},
    series = {Mathematics},
    publisher = {Van Nostrand Reinhold},
    address = {New York, NY},
    year = {1993}
}

@inproceedings{Donagi,
    author = {R.Y. Donagi},
    title = {Decomposition of spectral covers},
    booktitle = {Journée de géométrie algébrique d'Orsay - Juillet 1992},
    series = {Astérisque},
    volume = {218},
    pages = {145-175},
    year = {1993}
}

@article{DG,
    author = {R.Y. Donagi and D. Gaitsgory},
    title = {The gerbe of {H}iggs bundles},
    journal = {Transformation Groups},
    volume = {7},
    number = {2},
    pages = {109-153},
    year = {2002}
}

@article{DP,
    author = {R.Y. Donagi and T. Pantev},
    title = {{L}anglands duality for {H}itchin systems},
    journal = {Invent. Math.},
    volume = {189},
    pages = {653-735},
    year = {2012}
}

@article{FMN,
    author = {B. Fantechi and E. Mann and F. Nironi},
    title = {Smooth toric {D}eligne-{M}umford stacks},
    journal = {J. reine angew. Math.},
    volume = {640},
    pages = {201-244},
    year = {2010}
}

@article{FHZ,
    author = {A. {Fernandez Herrero} and S. Zhang},
    title = {Meromorphic {H}odge moduli spaces for reductive groups in arbitrary characteristic},
    journal = {Michigan Mat. J.},
    volume = {Advance Publication},
    pages = {1-62},
    year = {2026}
}

@article{FraPN,
    author = {E. Franco and A. Peón-Nieto},
    title = {Branes on the singular locus of the {H}itchin system via {B}orel and other parabolic subgroups},
    journal = {Math. Nachr.},
    volume = {296},
    number = {5},
    pages = {1803-1841},
    year = {2023}
}

@unpublished{FruPN,
author = {A. Fr\"{u}h and A. Pe\'{o}n-Nieto},
title = {Spectral and cameral data for {$U(p,q)$-H}iggs bundles via central gerbe extensions},
note = {In preparation}
}

@article{FRW,
    author = {B. Fu and Y. Ruan and Y. Wen},
    title = {Mirror symmetry of special nilpotent orbit closures},
    journal = {Sci. China Math.},
    volume = {67},
    year = {2024}
}

@article{GMN,
    author = {D. Gaiotto and G.W. Moore and A. Neitzke},
    title = {Wall-crossing, {H}itchin systems, and the {WKB} approximation},
    journal = {Adv. Math.},
    volume = {234},
    pages = {239-403},
    year = {2013}
}

@article{GMN1,
    author = {D. Gaiotto and G.W. Moore and A. Neitzke},
    title = {Spectral networks},
    journal = {Annales Henri Poincaré},
    volume = {14},
    pages = {1643-1731},
    year = {2013}
}

@article{GPGMIR,
    author={O. Garc\'{i}a-Prada and P.B. Gothen and I. Mundet I Reira},
  title={The {H}itchin-{K}obayashi correspondence, {H}iggs pairs and surface group representations},
  journal={arXiv preprint arXiv:0909.4487},
  year={2009}
}

@article{GPPN,
    author = {O. Garc\'{i}a-Prada and A. Pe\'{o}n-Nieto},
    title = {Abelianization of {H}iggs bundles for quasi-split real groups},
    journal = {Transformation Groups},
    volume = {28},
    pages = {285-325},
    year = {2023}
}

@book{Giraud,
    author = {J. Giraud},
    title = {Cohomologie non abélienne},
    series = {Die Grundlehren der mathematischen Wissenschaften},
    volume = {179},
    publisher = {Springer-Verlag},
    address = {Berlin, Heidelberg, DE},
    year = {1971}
}

@book{GW,
    author = {R. Goodman and N.R. Wallach},
    title = {Symmetries, Representations and Invariants},
    series = {Graduate Texts in Mathematics},
    volume = {255},
    publisher = {Springer},
    address = {New York, NY},
    year = {2009} 
}

@article{GWZ,
    author = {M. Groechenig and D. Wyss and P. Ziegler},
    title = {Mirror symmetry for moduli spaces of {H}iggs bundles via p-adic integration},
    journal = {Invent. Math.},
    volume = {221},
    pages = {505-596},
    year = {2020}
}

@article{HR,
    author = {J. Hall and D. Rydh},
    title = {Coherent {T}annaka duality and algebraicity of {H}om-stacks},
    journal = {Algebra Number Theory},
    volume = {13},
    number = {7},
    pages = {1633-1675},
    year = {2019}
}

@misc{HLM,
    author = {T. Hameister and Z. Luo and B. Morrissey},
    title = {Relative {D}olbeault {G}eometric {L}anglands via the {R}egular {Q}uotient},
    howpublished = {arXiv preprint},
    note = {arXiv:2409.15691},
    year = {2024}}

@article{HM,
    author = {T. Hameister and B. Morrissey},
    title = {The {H}itchin Fibration for Symmetric Pairs},
    journal = {Adv. Math.},
    volume = {482},
    number = {A},
    pages = {110560},
    year = {2025}
}

@article{HT,
    author = {T. Hausel and M. Thaddeus},
    title = {Mirror symmetry, {L}anglands duality, and the {H}itchin system},
    journal = {Invent. Math.},
    volume = {153},
    pages = {197-229},
    year = {2003}
}

@article{Hesselink,
    author = {W.H. Hesselink},
    title = {Polarizations in the classical groups},
    journal = {Math. Z.},
    volume = {160},
    pages = {217-234},
    year = {1978}
}

@article{Hitchin1,
    author = {N. Hitchin},
    title = {The self-duality equations on a {R}iemann surface},
    journal = {Proc. London Math. Soc.},
    volume = {55},
    number = {3},
    pages = {59-126},
    year = {1987}
}

@article{Hitchin2,
    author = {N. Hitchin},
    title = {Stable bundles and integrable systems},
    journal = {Duke Math. J.},
    volume = {54},
    number = {1},
    pages = {91-114},
    year = {1987}
}

@article{Hitchin3,
    author = {N. Hitchin},
    title = {Lie groups and {T}eichmüller space},
    journal = {Topology},
    volume = {31},
    pages = {449-473},
    year = {1992}
}

@article{Hitchin4,
    author = {N. Hitchin},
    title = {Critical loci for {H}iggs bundles},
    journal = {Commun. Math. Phys.},
    volume = {366},
    pages = {841-864},
    year = {2019}
}

@inproceedings{Hitchin5,
   author = {N.J. Hitchin},
    title = {Higgs bundles and characteristic classes},
    booktitle = {Arbeitstagung Bonn 2013: In Memory of Friedrich Hirzebruch},
    editor = {W. Ballmann and C. Blohmann and G. Faltings and P. Teichner and D. Zagier},
    series = {Progress in Mathematics},
    volume = {319},
    publisher = {Birkhäuser},
    address = {Cham, CH},
    pages = {247-264},
    year = {2016}
}

@article{HS,
    author = {N. Hitchin and L.P. Schaposnik},
    title = {Nonabelianization of {H}iggs bundles},
    journal = {J. Differential Geom.},
    volume = {97},
    number = {1},
    pages = {79-89},
    year = {2014}
}

@phdthesis{ImHof,
    author = {A.E. Im Hof},
    title = {The sheets of a classical {L}ie algebra},
    school = {University of Basel},
    year = {2005}
}

@article{Horn1,
    author = {J. Horn},
    title = {Semi-abelian spectral data for singular fibres of the {$\textsf{SL}(2,\mathbb{C})$}-{H}itchin system},
    journal = {Int. Math. Res. Not.},
    volume = {2022},
    number = {5},
    pages = {3860-3917},
    year = {2022}
}

@article{Horn2,
    author = {J. Horn},
    title = {$\mathfrak{sl}_2$-Type singular fibres of the symplectic and odd orthogonal {H}itchin system},
    journal = {J. Topol.},
    volume = {15},
    number = {1},
    pages = {1-38},
    year = {2022}
}

@article{Howlett,
    author = {R.B. Howlett},
    title = {Normalizers of parabolic subgroups of reflection groups},
    journal = {J. Lond. Math. Soc.},
    volume = {21},
    number = {1},
    pages = {62-80},
    year = {1980}
}

@article{Joseph,
    author = {A. Joseph},
    title = {On the associated variety of a primitive ideal},
    journal = {J. Algebra},
    volume = {93},
    pages = {509-523},
    year = {1985}
}

@article{KW,
    author = {A. Kapustin and E. Witten},
    title = {Electric-magnetic duality and the geometric {L}anglands program},
    journal = {Commun. Number Theory Phys.},
    volume = {1},
    number = {1},
    pages = {1-236},
    year = {2007}
}

@article{Katsylo,
    author = {P.I. Katsylo},
    note = {English translation by L. Queen},
    title = {Sections of sheets in a reductive algebraic {L}ie algebra},
    journal = {Math. USSR Izv.},
    volume = {20},
    number = {3},
    pages = {449-458},
    year = {1983}
}

@article{KM,
    author = {S. Keel and S. Mori},
    title = {Quotients by Groupoids},
    journal = {Ann. Math.},
    volume = {145},
    number = {1},
    pages = {193-213},
    year = {1997}
}

@book{Knapp,
    author = {A.W. Knapp},
    title = {Lie Groups Beyond an Introduction},
    series = {Progress in Mathematics},
    volume = {140},
    publisher = {Birkh\"{a}user},
    address = {Boston, MA},
    year = {1996}
}

@article{Kostant,
    author = {B. Kostant},
    title = {Lie Group Representations on Polynomial Rings},
    journal = {Amer. J. Math.},
    volume = {85},
    number = {3},
    pages = {327 - 404},
    year = {1963}
}

@article{KR,
    author = {B. Kostant and S. Rallis},
    title = {Orbits and Representations Associated with Symmetric Spaces},
    journal = {Amer. J. Math.},
    volume = {93},
    number = {3},
    pages = {753-809},
    year = {1971}
}

@article{Kraft,
    author = {H. Kraft},
    title = {Parametrisierung von {K}onjugationsklassen in
$\mathfrak{sl}_n$},
    journal = {Math. Ann.},
    volume = {234},
    pages = {209-220},
    year = {1978}
}

@book{LMB,
    author = {G. Laumon and L. Moret-Bailly},
    title = {Champs algébriques},
    series = {Ergebnisse der Mathematik und ihrer Grenzgebiete. 3. Folge / A Series of Modern Surveys in Mathematics},
    volume = {39},
    publisher = {Springer-Verlag},
    address = {Berlin, Heidelberg, DE},
    year = {2000}
}

@misc{LeePN,
    author = {J.C. Lee and A. Pe\'on-Nieto},
    title ={Spectral correspondence for cyclic Higgs bundles} ,
    howpublished ={arXiv preprint},
    note = {arXiv:2605.03249},
    year = {2026},
    doi =  {10.48550/arXiv.2605.03249}
}

@article{Leslie,
    author = {S. Leslie},
    title = {An analogue of the {G}rothendieck-{S}pringer resolution for symmetric spaces},
    journal = {Algebra Number Theory},
    volume = {15},
    number = {1},
    pages = {69-107},
    year = {2021}
}

@article{Losev3,
    author = {I. Losev},
    title = {Finite dimensional representations of {$W$}-algebras},
    journal = {Duke Math. J.},
    volume = {159},
    number = {1},
    pages = {99-143},
    year = {2011}
}

@article{Losev4,
    author = {I. Losev},
    title = {Deformations of symplectic singularities and {O}rbit method for semisimple {L}ie algebras},
    journal = {Sel. Math. New Ser.},
    volume = {28},
    number = {2},
    note = {Paper No. 30},
    year = {2022}
}

@misc{LMM,
    author = {I. Losev and L. Mason-Brown and D. Matvieievskyi},
    title = {Unipotent ideals and {H}arish-{C}handra bimodules},
    howpublished = {arXiv preprint},
    note = {arXiv:2108.03453},
    year = {2021}
}

@article{LS,
    author = {G. Lusztig and N. Spaltenstein},
    title = {Induced {U}nipotent {C}lasses},
    journal = {J. Lond. Math.},
    volume = {19},
    number = {1},
    pages = {41 - 52},
    year = {1979}
}

@article{Namikawa,
    author = {Y. Namikawa},
    title = {Poisson deformations of affine symplectic varieties, {II}.},
    journal = {Kyoto J. Math.},
    volume = {50},
    number = {4},
    pages = {727-752},
    year = {2010}
}

@article{Ngo1,
    author = {B.C. Ng{\^{o}}},
    title = {Fibration de {H}itchin et endoscopie},
    journal = {Invent. Math.},
    volume = {164},
    pages = {399-453},
    year = {2006}
}

@article{Ngo2,
    author = {B.C. Ng{\^{o}}},
    title = {Le lemme fondamental pour les algèbres de {L}ie},
    journal = {Publ. Mathémathiques},
    volume = {111},
    pages = {1-169},
    year = {2010}
}

@inproceedings{Ngo3,
    author = {B.C. Ng{\^{o}}},
    title = {On generalized {H}itchin fibrations and orbital integrals},
    booktitle = {The Langlands Program III. Periods, $L$-Functions and the Relative Trace Formula},
    editor = {P-H. Chaudouard and W.T. Gan and T. Kaletha and Y. Sakellaridis},
    series = {Proceedings of Symposia in Pure Mathematics},
    volume = {112},
    pages = {313-336},
    publisher = {American Mathematical Society},
    address = {Providence, RI},
    year = {2025},
    note = {\emph{Survey of unpublished work with B. Morrissey}}
}

@article{OW,
    author = {H. Ozeki and M. Wakimoto},
    title = {On polarizations of certain homogeneous spaces},
    journal = {Hiroshima Math. J.},
    volume = {2},
    pages = {445-482},
    year = {1972}
}

@misc{Pantev,
    author = {T. Pantev},
    title = {Classifying abstract {H}iggs bundles},
    howpublished = {Presented at ``Geometry, higher structures, and physics'' at ICMS, Edinburgh},
    month = {December},
    year = {2025},
    note = {\emph{Announcement of joint work in preparation with D. Arinkin and R. Donagi}}
}

@phdthesis{Peon-Nieto1,
    author = {A. Pe\'{o}n-Nieto},
    title = {{H}iggs bundles, real forms and the {H}itchin fibration},
    school = {Universidad Aut\'{o}noma de Madrid},
    year = {2013}
}

@article{Premet1,
    author = {A. Premet},
    title = {Special transverse slices and their enveloping algebras},
    journal = {Adv. Math.},
    volume = {170},
    pages = {1-55},
    year = {2002}
}

@article{Premet3,
    author = {A. Premet},
    title = {Commutative quotients of finite {$W$}-algebras},
    journal = {Adv. Math.},
    volume = {225},
    number = {1},
    pages = {269-306},
    year = {2010}
}

@article{PT,
    author = {A. Premet and L. Topley},
    title = {Derived subalgebras of centralisers and finite {$W$}-algebras},
    journal = {Compos. Math.},
    volume = {150},
    number = {9},
    pages = {1485-1584},
    year = {2014}
}

@article{Riche,
    author = {S. Riche},
    title = {Kostant section, universal centraliser, and a modular derived {S}atake equivalence}, 
    journal = {Math. Z.},
    volume = {286},
    number = {1-2},
    pages = {223-261},
    year = {2017}
}

@article{Romagny,
    author = {M. Romagny},
    title = {Group actions on stacks and applications},
    journal = {Mich. Math. J.},
    volume = {53},
    number = {1},
    pages = {209-236},
    year = {2005}
}

@misc{SV,
    author = {M. Savarese and F. Viviani},
    title = {Moduli of sheaves on ribbons},
    howpublished = {arXiv preprint},
    note = {arXiv:2502.01412},
    year = {2025}
}

@article{Schaposnik,
    author = {L.P. Schaposnik},
    title = {Spectral data for {U}(m,m)-{H}iggs bundles},
    journal = {Int. Math. Res. Not.},
    volume = {2015},
    number = {11},
    pages = {3486-3498},
    year = {2014}
}

@article{Schaub,
    author = {D. Schaub},
    title = {Courbes spectrales et compactifications de jacobiennes},
    journal = {Math. Z.},
    volume = {227},
    pages = {295-312},
    year = {1998}
}

@article{Sekiguchi,
    author = {J. Sekiguchi},
    title = {The nilpotent subvariety of the vector space associated to a symmetric pair},
    journal = {Publ. Res. Inst. Math. Sci.},
    volume = {20},
    number = {1},
    pages = {155-212},
    year = {1984}
}

@article{Simpson1,
    author = {C.T. Simpson},
    title = {Higgs bundles and local systems},
    journal = {Publ. Mathémathiques},
    volume = {75},
    pages = {5-95},
    year = {1992}
}

@book{Slodowy,
    author = {P. Slodowy},
    title = {Simple Singularities and Simple Algebraic Groups},
    series = {Lecture Notes in Mathematics},
    volume = {815},
    publisher = {Springer-Verlag},
    address = {Berlin, Heidelberg, DE},
    year = {1980}
}

@misc{Toda1,
    author = {Y. Toda},
    title = {A proof of {D}olbeault geometric {L}anglands for {$\mathrm{GL}_2$} with reduced spectral curves},
    howpublished = {arXiv preprint},
    note = {arXiv:2602.09359},
    year = {2026}
}

@article{Topley,
    author = {L. Topley},
    title = {One dimensional representations of finite {$W$}-algebras, {D}irac reduction and the orbit method},
    journal = {Invent. Math.},
    volume = {234},
    number = {3},
    pages = {1039-1107},
    year = {2023}
}
\end{document}